\documentclass[11pt,reqno]{amsproc}
\usepackage[margin=1.0in]{geometry}
\usepackage[english]{babel}
\usepackage[utf8]{inputenc}
\usepackage{ulem}
\usepackage{amsmath,amssymb,amsthm, comment, mathtools}
\usepackage{hyperref,xcolor, amsfonts, amsbsy}
\usepackage{enumitem}
\usepackage{genealogytree}
\usepackage{mathrsfs}
\usepackage{cancel}

\newcommand{\R}{\mathbb{R}}

\newcommand{\N}{\mathcal{N}}
\newcommand{\Z}{\mathbb{Z}}

\renewcommand{\L}{\mathcal{L}}
\newcommand{\M}{\mathcal{M}}

\newcommand{\ve}{\varepsilon}
\newcommand{\rmd}{{\rm d}}

\usepackage[mathscr]{euscript}

\newcommand{\sff}{{\mathrm{I\!I}}}
\newcommand{\PP}{\mathbf{P}}

\newcommand{\hess}{{\rm Hess}}

\newcommand{\vr}{{\varrho}}

\DeclareMathOperator*{\wlim}{\mathsf{w}^*-lim}

\newcommand{\eee}{equation}
\newcommand{\be}{\begin{\eee}}
	\newcommand{\ee}{\end{\eee}}

\let\div\relax 

\newcommand{\red}[1]{\textcolor{red}{#1}}

\DeclareMathOperator{\div}{div}
\DeclareMathOperator{\grad}{grad}

\numberwithin{equation}{section}
\newtheorem{lemma}{Lemma}[section]
\newtheorem{prop}[lemma]{Proposition}
\newtheorem{theorem}[lemma]{Theorem}
\newtheorem{cor}[lemma]{Corollary}
\newtheorem{assump}[lemma]{Assumption}

\theoremstyle{definition}
\newtheorem{remark}[lemma]{Remark}
\newtheorem{defi}[lemma]{Definition}

\newtheorem{example}[lemma]{Example}

\makeatletter
\DeclareFontFamily{OMX}{MnSymbolE}{}
\DeclareSymbolFont{MnLargeSymbols}{OMX}{MnSymbolE}{m}{n}
\SetSymbolFont{MnLargeSymbols}{bold}{OMX}{MnSymbolE}{b}{n}
\DeclareFontShape{OMX}{MnSymbolE}{m}{n}{
    <-6>  MnSymbolE5
   <6-7>  MnSymbolE6
   <7-8>  MnSymbolE7
   <8-9>  MnSymbolE8
   <9-10> MnSymbolE9
  <10-12> MnSymbolE10
  <12->   MnSymbolE12
}{}
\DeclareFontShape{OMX}{MnSymbolE}{b}{n}{
    <-6>  MnSymbolE-Bold5
   <6-7>  MnSymbolE-Bold6
   <7-8>  MnSymbolE-Bold7
   <8-9>  MnSymbolE-Bold8
   <9-10> MnSymbolE-Bold9
  <10-12> MnSymbolE-Bold10
  <12->   MnSymbolE-Bold12
}{}

\let\llangle\@undefined
\let\rrangle\@undefined
\DeclareMathDelimiter{\llangle}{\mathopen}%
                     {MnLargeSymbols}{'164}{MnLargeSymbols}{'164}
\DeclareMathDelimiter{\rrangle}{\mathclose}%
                     {MnLargeSymbols}{'171}{MnLargeSymbols}{'171}
\makeatother

\mathtoolsset{showonlyrefs,showmanualtags}

\title{Geometry of strong forces in continuum mechanics}

\author{Theodore D. Drivas}
\address{Department of Mathematics, Stony Brook University, Stony Brook, NY, 11790}
\email{tdrivas@math.stonybrook.edu}

\author{Daniil  Glukhovskiy}
\address{Department of Mathematics, Stony Brook University, Stony Brook, NY, 11790}
\email{daniil.glukhovskiy@stonybrook.edu }

\begin{document}

\begin{abstract}
Consider a material point in finite dimensions moving  under the influence of a  potential force according to Newton's laws.  Suppose the potential energy function is a generalized well,  strictly convex transverse to a smooth submanifold  $\M$ on which it is minimal.   If the potential is steep, one expects the particle will oscillate rapidly about this submanifold  and, if the initial displacement is not too great, should approximately move along $\M$  as if it were ideally constrained (e.g. geodesic).  This expectation is true if the initial conditions are very well prepared but may fail otherwise --  additional potential forces determined by how the Hessian of the potential varies along $\M$ may be present \cite{RU57,T80}. The origin of this force is that the transversal motion acts as a simple harmonic oscillator with a slowly varying frequency, which approximately conserves action, not energy. In this work, we regard continuum mechanical systems such as the elastic thread or  compressible fluid as material points moving in an infinite dimensional space according to Newton's laws for appropriate potential energy functionals.   We show how to arrive at ideally constrained systems such as the inextensible thread and incompressible fluid as a limit of a strong potential force, computing also corrections to the naive predictions when the data is not very well prepared.  For example, for the thread we find a resistance to bending emerge from a strong resistance to compression/expansion.  For the fluid, the effective incompressible dynamics may be driven by a remnant acoustical wavefield.  Both of these emergent features are nonlinear and non-local. Finally, we give examples of some limits for which the naive models robustly hold because the additional force is trivial. These include the homogeneous incompressible Euler, anelastic Euler, as well as the lake and great lake equations.  
\end{abstract}

\vspace*{0mm}
	\maketitle

	\section{Introduction}

	Many physical systems regularly encountered in mechanics are modeled by Newtonian motion of a particle subject to some constraints (e.g., forced to live in a particular region of configuration space).  The prototypical example is the mathematical pendulum: its motion can be viewed as a free motion in a gravitational field constrained to live on a $S^1 \subset \R^2$. In such systems, the equations of motion are usually axiomatically defined by d'Alembert's principle: a reaction force of the constraint is supplied at each instant of time in a direction purely normal to the constraint, and should be no more or less than is required to thwart the acceleration off the surface.
	
	In more detail, d'Alembert's principle says that the equations of motion of a particle under the influence of conservative force field with potential $W$ and constrained to live on given submanifold $\M$ of a (possibly infinite dimensional) Euclidean/Hilbert enveloping space $H$ equipped with an inner product $\langle \cdot ,  \cdot  \rangle$  are given by
		\begin{align} \label{dalembert}
		\ddot{X}(t)+ \nabla W(X(t)) &\in N_{{X(t)}} \M\\
		{X}(t) &\in  \M
	\end{align}
	where $N_{X} \M$ is the normal space to $\M$ at the point $X$, that is the orthogonal complement of the tangent space with respect to the inner product $\langle \cdot ,  \cdot  \rangle$, and the gradient $\nabla$ is computed with this scalar product.  See discussion in \cite[Chapter 21]{A13}.	 Note that, given an ``inertia operator" $M: H \to H$, a self-adjoint positive definite operator, bounded and invertible, one may define the inner product $\langle \cdot ,  \cdot  \rangle_M := \langle M \cdot ,  \cdot  \rangle_{H}$, where $ \langle  \cdot ,  \cdot  \rangle_{H}$ is the standard inner product on $H$.  Such an inner product allows to relate \eqref{dalembert} to the familiar Newton equation, where $M$ represents the ``mass" of the particle. Indeed, in this case  $\nabla_M f =M^{-1} \nabla_Hf$ where $\nabla_M$ is the gradient with respect to  $\langle \cdot ,  \cdot  \rangle_M$ inner product, and $\nabla_H$ to  $\langle \cdot ,  \cdot  \rangle_H$.    From hereon, we keep the inner product structure general and, in examples, specify the relevant inertia operator.

	In most physical examples we are aware of, the constraint is an idealization, and a more truthful representation is given by the unconstrained system that is instead subject to additional strong ``constraining" forces. For example, one might think that a more truthful representation of the rigid pendulum is rather given by the mass attached to the pivot by a very stiff spring.
	
	\begin{figure}[h!]\centering \label{cartoon-fig}
		\includegraphics[width=.52\columnwidth]{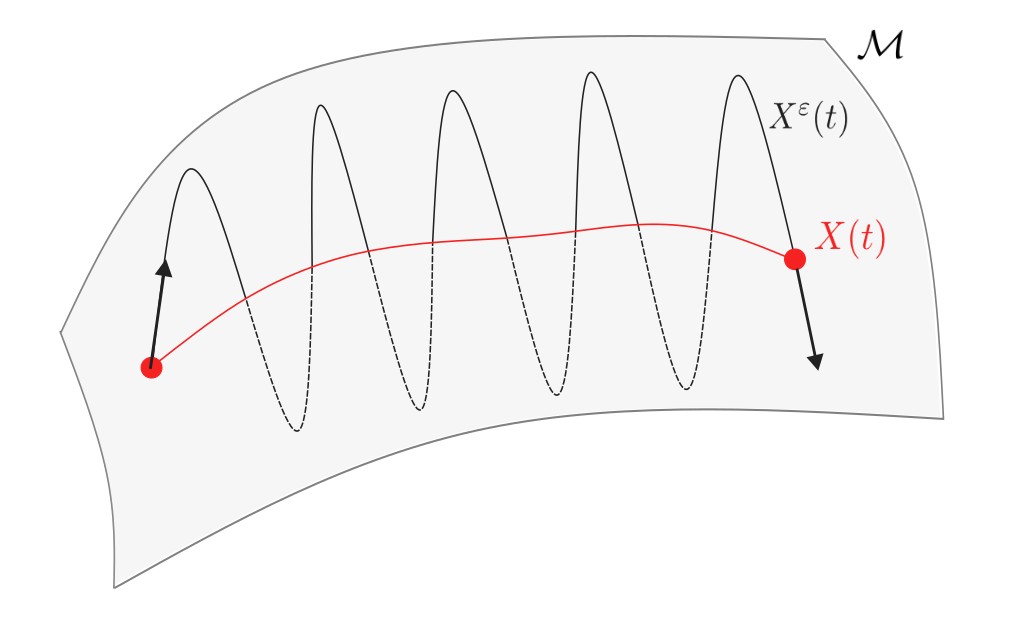}
		\caption{\textbf{Black}: trajectory of the strongly forced system, oscillating around manifold $M$. \textbf{\red{Red}}: trajectory of the constrained model given by d'Alembert's principle.}
	\end{figure}
	
	It is thus reasonable to ask whether the  axiomatically defined d'Alembert motion can be derived from the more fundamental Newton's laws by some sort of a limiting procedure. Specifically, one usually considers taking the true system to be given, for some potential $U$, by
\be\label{strong0}
\begin{cases}
	\ddot{X}^\ve(t) =  -\nabla W(X^\ve(t)) - \tfrac{1}{\ve^2} \nabla U(X^\ve(t)),\\
	X^\ve(0) = X_0^\ve \xrightarrow{\ve\to 0}X_0 \in \M,\\
	\dot{X}^\ve(0) = V^\ve_0 \xrightarrow{\ve\to 0}V_0 \in \mathbb{R}^d.
\end{cases}
\ee
We assume, as in the example of the spring pendulum, that $\M= \{ x : U(x)=0\}$, and is a non-degenerate minimum set for $U$, e.g. $\hess\  U_x(v,v)\geq c_0\|v\|^2$ for some $c_0>0$ holding for all $x\in \M$ and $v\in N_x \M$.

The first mathematical theorem justifying d'Alembert's principle in this context was proven by Rubin and Ungar in \cite{RU57}. There, the authors prove that if initial velocity is \textit{tangential}, that is 
$
V_0 \in T_{X_0}\M,
$ 
then the solutions of \eqref{strong0} converge to solutions of \eqref{dalembert} as $\ve \to 0$. The authors of \cite{RU57} proceed to discuss the non-tangential initial data case, and note that, at least if $\M$ has codimension one, the convergence might fail - and instead the solutions of \eqref{strong0} will converge, as $\ve \to 0$ to the system given by d'Alembert's principle subject to \textit{additional potential force}. This force arises from averaging of oscillations around the constraint. 

A more general theorem for arbitrary codimension was proven by Takens in \cite{T80}, and will be discussed in more detail in the next section. Assuming some non-resonance conditions, the limit system turns out to be given by
\be \label{takenslimit}
\begin{cases} 
	\ddot{X}(t)+ \nabla W(X(t)) + \nabla V(X(t)) \in N_{X(t)} \M\\
	{X}(t) \in  \M
\end{cases}
\ee
where the additional potential is defined by
\be \label{hompot0}
V(X) = \sum c_i \sqrt{\lambda_i(X)},
\ee 
where $\lambda_i(X)$ are eigenvalues of the operator $\hess_X U\big|_{N_X \M}$, and $c_i$ are nonnegative constants that depend on the normal components of initial data. We call potential $V$ given by \eqref{hompot0} the \textit{homogenized potential of} $U$, and the system \eqref{takenslimit} the \textit{Takens limit system}.  

In this work, we are concerned with the application of Takens' theorem to infinite-dimensional systems that arise in continuum mechanics.  Similar to Archimedes's Mechanical Method, our results should be regarded as predictions for the governing equations in this singular limit -- a heuristic method to derive them -- rather than rigorous theorems.  To prove the corresponding theorem in a given context, one must contend with PDE issues which require delicate functional analysis. The predictions thus might fail in the most general settings.\footnote{As V.I. Arnold writes in \cite{A13} about the averaging principle for perturbations of Hamiltonian systems, \textit{``We note that this principle is neither an axiom, nor a definition, but rather a physical proposition, i.e., a vaguely formulated and, strictly speaking, untrue assertion. Such assertions are often fruitful sources of mathematical theorems."}} Thus, we do not make any claims regarding the convergence of considered systems to their Takens limits, and are being intentionally vague regarding the meanings of infinite sums appearing in homogenized potentials and related places. However, rigorous validations of our predictions exist for various model problems (for example, the incompressible  \cite{E77,MS03} and anelastic \cite{M07} limits in fluid dynamics) and, to our knowledge, all such results agree with our theory.  

We now list predictions for several classical systems. Our first example of an infinite-dimensional constrained continuum mechanical model is the equation of inextensible thread. It is described by 
\be\label{inextthread0}
\begin{cases}
		\ddot{X} = \frac{1}{\vr_0}\nabla_a(\sigma \partial_a X),\\
		|\partial_a X| = 1.
\end{cases}
\ee 
Here $X$ is a time-dependent curve on a fixed Riemannian manifold parametrized by $a \in S^1$ and $\vr_0: S^1 \to \mathbb{R}^+$ is a fixed mass density distribution. As derived in \cite{S,P12}, there is an interpretation of \eqref{inextthread0} as constrained motion, the configuration space being $\mathsf{SLoops}(M) := \{X \in C^{\infty}(S^1, M) \ | \ |\partial_aX| = 1 \}$. The corresponding unconstrained model is given by the hyperelastic extensible thread \cite{Al95}:
\be\label{extthread0}
\ddot{X} = \frac{1}{\vr_0}\nabla_a\bigg(\frac{W'(|\partial_a X|, a)\partial_a X}{|\partial_a X|}\bigg),
\ee  
where $W: \R \times S^1 \to \R$ is the physical parameter roughly corresponding to the stored internal energy per unit length (thought of a large in magnitude $W\sim {1}/{\ve^2}$). Taking the constrained (`inextensible') limit of \eqref{extthread0}, we derive the correction that should appear in equations \eqref{inextthread} if one wants to consider ill-prepared initial data. Takens limit system is given, for flat $M$ (we only do it here to highlight the structure, and general manifold $M$ is considered in Theorem \ref{threadthm}), by
\be
\begin{cases}
	\ddot{X} =  \frac{1}{\vr_0}\nabla_a(\sigma \partial_a X) - \frac{1}{\vr_0}\nabla_a\nabla_a(K \nabla_a\partial_a X)\\
	|\partial_a X| = 1\\
	K = \sum_i \frac{ c_i v_i^2}{\sqrt{\lambda_i}\vr_0}.
\end{cases}
\ee
Here $\lambda_i, v_i$ are appropriately normalized eigenpairs of \textit{tension operator}
\be 
W''(1, \cdot)\L =  W''(1, \cdot) \bigg(-\partial_a \bigg(\frac{1}{\vr_0}\partial_a\bigg) + \frac{1}{\vr_0}|\partial_a^2 X |^2\bigg),
\ee  
and $c_i$ are constants depending on initial data.
We note that the correction that appears in Takens limit system is given by the fourth-order bending-resistance term, of the type usually arising in Euler beam theories which arise from potentials that penalize large curvature.  In our framework, bending resistance is explained as an emergent phenomenon triggered by large initial deviations of the system.  Unlike the usual phenomenological theories, in this setting the effect is nonlocal and nonlinear, depending on the (evolving) configuration of thread.

The second example is given by incompressible Euler equations
\be
\begin{cases}
	\partial_t u +\nabla_u u = - \frac{\nabla p}{\vr}\\
	\partial_t \vr + \nabla_{u} \vr = 0\\
	\div u = 0.
\end{cases}
\ee 
Here $M$ is a smooth Riemannian manifold, and the unknowns are a time dependent vector field $u$ on $M$ and a pair of time-dependent scalar functions $p, \vr$ on $M$. Since the seminal work of Arnold \cite{A66,AK}, it is known that Euler equations admit, in Lagrangian coordinates, an interpretation as the free motion in the \textit{group of diffeomorphisms of} $M$, constrained to move on the subspace $\mathsf{SDiff} (M)$ consisting of those diffeomorphisms that preserve volume. The corresponding ``more truthful" unconstrained model is given by compressible Euler equations:
\be\label{compeuler0}
\begin{cases}
	\partial_t u + \nabla_u u = - \frac{1}{\ve^2} \frac{\nabla P}{\vr}\\
	\partial_t \vr + \div(\vr u) = 0\\
	\partial_t s +  \nabla_u s = 0\\
	
\end{cases}
\ee 
where $u$ is the velocity, $\vr$ mass density, $s$ entropy, $P = \widetilde{P}(\vr,s)$ is given by a thermodynamic equation of state, and
where $\ve$ is the Mach number. 
This fact was used by Ebin in \cite{E77} to prove the first result on the incompressible limit for Euler equations. 

In \cite{E77}, only well-prepared data -- equivalently, data tangential to
$\mathsf{SDiff}(M)$ -- is considered, and the unconstrained model is the
barotropic compressible Euler system. In Section 4, we treat this example
in greater generality. For nontangential data in the fully compressible,
inhomogeneous setting, the incompressible constrained model acquires an
additional forcing term, just as in the finite-dimensional mechanical
problem. The Takens limit system corresponding to \eqref{compeuler0} is
\be\label{takens-fluid0}
\begin{cases}
	\partial_t u + \nabla_u u = - \frac{\nabla p}{\vr}-  \frac{\div \Sigma}{\vr}\\
	\partial_t \vr + u\cdot\nabla \vr = 0,\\
	\div u = 0,	
\end{cases}
\ee 
where, as usual, $p$ is the pressure serving as a Lagrange multiplier to enforce the incompressibility constraint, and $\Sigma$ is a symmetric 2-tensor that we call \textit{acoustic stress}, given by
\be 
\Sigma = \sum_i \frac{c_i}{\sqrt{\lambda_i}} \frac{\nabla v_i \otimes \nabla v_i}{\vr},
\ee 
where $\lambda_i, v_i$ are appropriately normalized eigenpairs of \textit{acoustic wave operator}
\be 
\L = - \vr\partial_\vr \tilde{P} \div\left(\tfrac{1}{\vr}\nabla\right),
\ee  
and $c_i$ are adiabatic invariants that are explicitly computed from the initial data in Proposition \ref{EulerTakensPotential}. We note that entropy becomes the function of density, and its distribution is forced by the condition that $\tilde{P}(\vr,s) = {\rm const}$, which is the manifestation of constraint $X \in \M$.
The force we derive agrees with the one found in \cite{MS03}, although it is not cast in this form, where some of the functional-analytic details are carried out and the convergence of solutions of \eqref{compeuler0} to those of \eqref{takens-fluid0} under non-resonance conditions is proven in Sobolev topologies. 

In the remainder of the paper we analyze two  constrained models from geophysical fluid dynamics: that of anelastic equations (Section 5) and of lake/Great Lake equations (Section 6). For each model, we provide the description as constrained motion on the appropriate submanifold of $\mathsf{Diff}(M)$ (that seems to be new). We proceed to describe the corresponding unconstrained models (compressible Euler in the gravity field, and shallow water/Green-Naghdi equations respectively) as Newton equations on $\mathsf{Diff}(M)$. As such, we realize anelastic and lake/Great Lake models as limits of strong constrained force. In both cases, it turns out that the potential correction in the ill-prepared data regime vanishes, and the Takens limit system agrees with the usual constrained motion. This provides geometric reason for the robust applicability of these models.

	\section{Dynamics of particle under strongly constraining force: Takens theorem}
	
We study a mechanical system with a strong force which drives the solution toward submanifold $\M$ of $\R^d$.  More specifically, consider
\be\label{strong}
\begin{cases}
	\ddot{X}^\ve(t) =  -\nabla W(X^\ve(t)) - \tfrac{1}{\ve^2} \nabla U(X^\ve(t)),\\
	X^\ve(0) = X_0^\ve \xrightarrow{\ve\to 0} X_0 \in \M,\\
	\dot{X}^\ve(0) = V^\ve_0 \xrightarrow{\ve\to 0} V_0 \in \mathbb{R}^d.
\end{cases}
\ee
Again, we assume that $\M= \{ x : U(x)=0\}$, and is a non-degenerate minimum set for $U$, e.g. $\hess\  U_x(v,v)\geq c_0\|v\|^2$ for some $c_0>0$ holding for all $x\in \M$ and $v\in N_x \M$.

We remark that almost everything that we are going to say works in the case where $\R^d$ is replaced by a general Riemannian manifold, with the formulas modified accordingly (e.g. $\ddot{X}$ is being replaced by $\nabla_{\dot{X}} \dot{X}$). The Riemannian versions can be found in the book of Bornemann \cite{B98}, on which our work heavily bases. We however choose to present the results with ambient space being Euclidean, in order to focus the attention on mechanical mechanism and not differential-geometric complications.

The question of convergence (and non-convergence) of the solutions of this system to solutions of the ideal system constrained to submanifold $\M$ has been extensively studied (\cite{RU57}, \cite{E77}, \cite{T80}, \cite{G83}, \cite{BS97}).
We proceed to formulate a theorem originally due to Takens \cite{T80} with the extension due to Bornemann \cite{B98}. The following definitions and statement are taken from \cite{B98}.
\begin{defi}\label{constraining}
	Consider a smooth potential $U: \R^d \to \R$, and let $\M$ be a minimum set of $U$. We call $\M$ a \textit{nondegenerate critical submanifold} if it is a smooth submanifold of $\R^d$ and Hessian of $U$ is nondegenerate when restricted to normal spaces of $\M$. We say such  $U$ is \textit{constraining} to $\M$.
\end{defi} 
We remark that \eqref{strong} has solutions for all times if  $U$ is smooth and constraining to $\M$ and $W$ is bounded below.
\begin{defi}
	We say that a potential $U: \R^d \to \R$ is  \textit{constraining spectrally smooth} to nondegenerate submanifold $\M \subset \R^d$ if it is constraining to $\M$ and moreover the spectrum of $\hess U|_{N_X \M}$ can be arranged in a smooth way. Specifically, there exists a collection of smooth fields of projections onto mutually orthogonal subspaces $E_i(X)$, say $k$, of $N_X \M$:
	\be
	\mathbf{P}_i(X): T_X\R^d \to  E_i(X) \subset N_X \M 
	\ee  
	such that at any point $X \in \M$, the Hessian of $U$ can be written as 
	\be 
	\hess U_X|_{N_X \M} = \sum_{i=1}^k \lambda_i(X) \mathbf{P}_i(X).
	\ee 
\end{defi}
Since $\hess U_X|_{N_X \M}$ is a symmetric positive definite operator for each $X$, spectral theorem implies that such projections exist at each point. The nontrivial part of the above definition is the smoothness with respect to varying $X$; in particular, the dimensions of the $E_i$ are constant on $\M$ .
\begin{defi}
	Let $U$ be spectrally smooth constraining to $\M$, and let $\lambda_i: \M \to \R$ be the  eigenvalues of $\hess U_X|_{N_X \M}$. A set of \textit{resonances of order $m$} is given by 
	\begin{align}
	R_m := \bigg\{x \ \bigg|\ \sum_{i=1}^k \gamma_i\sqrt{\lambda_{i}(x)} = 0 
	\text{ for } \gamma_1,\dots,\gamma_k \in \Z \text{ with } \sum_{j=1}^k |\gamma_j| = m \bigg\}.
	\end{align} 
	A curve $X$ on $\M$ is called \textit{non-flatly resonant up to order} $m$ if it crosses $R_1, ..., R_m$ only transversally. 
\end{defi}
In talking about the limit of solutions to systems \eqref{strong}, one of course has to define the appropriate initial conditions. In this discussion, the correct notion is given by the requirement that the (appropriately shifted) energy is bounded uniformly in $\ve$. This is achieved by the following type of initial data:

\begin{defi}\label{initial_data}
	We call a family of initial data $(X_0^\ve, V_0^\ve) \in T\R^d$ for $\ve \to 0$ \textit{mildly ill-prepared} if there exist a point $X_0 \in \M$ and vectors $X_{0N} \in N_{X_0} \M$, $V_0 \in T_{X_0}\R^d$ such that
	\be 
	(X_0^\ve,V_0^\ve) \to (X_0, V_0),
	\ee 
	and 
	\be 
	\frac{\rmd}{\rmd \ve}\bigg|_{\ve = 0} X_0^\ve = X_{0N}
	\ee 
	We call data $(X_0^\ve, V_0^\ve)$ \textit{well-prepared}, if $X_{0N} = 0$ and $V_0 \in T_{X_0}\M$.
\end{defi}
Adding a constant to $U$ to make $U(X_0) = 0$ (which has no influence on the dynamics), we note that for this type of data, the energy of the system \eqref{strong} at initial time is given by:
\be 
E_0^\ve = \tfrac{1}{2}|V_0^\ve|^2 + W(X_0^\ve) + \tfrac{1}{\ve^2}U(X_0^\ve) = \tfrac{1}{2}|V_0|^2 + W(X_0) + \frac{1}{2}\hess U_{X_0} (X_{0N}, X_{0N}) + o(1).
\ee
As such, it stays $O(1)$ as $\ve$ is decreased to zero, giving meaning to being mildly ill-prepared.

We now have  all the concepts required to understand the strongly constraining limit. 
\begin{defi}[Takens Limit]
	Suppose  $U$ is constraining spectrally smoothly  to a submanifold $\M$. We call  $V:\M\to \mathbb{R}$ a \textit{homogenized potential of} $U$ if, for some $c\in [0,\infty)^k$,
	\be\label{hompot}
	V(x) = \sum_{i=1}^k c_i \sqrt{\lambda_{i}(x)}.
	\ee
	We define the  \textit{Takens limit system} of \eqref{strong} with data converging to $(X_0,V_0)$ with $X_0\in \M$ to  be
		\be \label{limitsystem}
	\begin{cases}
		\ddot{X}(t) + \nabla W(X(t)) + \nabla V(X(t)) \in N_X(t) \M,\\
		X(0) = X_0,\\
		\dot{X}(0) = \mathbf{P}_{T_{X_0} \M} V_0.
	\end{cases}
	\ee
\end{defi}

This notion is justified by the following theorem, the codimension one version of which was proven by Rubin and Ungar \cite{RU57}, the full statement for Euclidean ambient space by Takens \cite{T80}, and for general ambient manifold and treatment of transversal resonances by Bornemann \cite{B98}.

\begin{theorem}[Takens \cite{T80}, Bornemann \cite{B98}]\label{maintheorem}
	Consider potential $U$ that is constraining spectrally smooth to  a nondegenerate submanifold $\M \subset \R^d$.
	Consider a sequence $\ve \to 0$, and a mildly ill-prepared sequence of initial data $(X_0^\ve, V_0^\ve)$. Let $(X_0,X_{0N}, V_0) \in \M \times N_{X_0} \M \times T_{X_0}\R^d$ be as in definition
	\ref{initial_data}.
	For each $\ve$, let $X^\ve$ be a solution of \eqref{strong} with initial conditions $X_0^\ve, V_0^\ve$. Define constants $c_i$  by 
	\be \label{constantsci}
	c_i = \frac{|\mathbf{P}_i (X_0) V_0|^2 + \lambda_i(X_0)|\mathbf{P}_i (X_0) X_{0N}|^2}{2\sqrt{\lambda_i(X_0)}},
	\ee 
	where $V_0, X_0, X_{0N}$ are from the definition of mildly ill-prepared data \ref{initial_data}.
	
	Let $X$ be the solution of the {Takens limit system} with homogenized potential $V: \M \to \R$ given in \eqref{hompot} constructed from $c_i$ as above.
	If $X$ is non-flatly resonant up to order 3, then
	\be
	X^\ve \to X \ \text{in} \  C^\alpha([0,T])
	\ee
	for all $\alpha \in(0,1)$ and any $T < \infty$ for which the solutions of \eqref{strong} and \eqref{limitsystem} exist.
\end{theorem}
\begin{remark}[Adiabatic Invariance]
	Constants $c_i$ have dimensions of energy over frequency, that is of \textit{action}. Their appearance is a manifestation of the phenomenon of \textit{adiabatic invariance} for a slowly modulated harmonic oscillator.  See \cite{A12, A14}.
\end{remark}
\begin{remark}[Topology of Convergence]
	In general for mildly ill-prepared data, while position converges uniformly, velocity converges only weakly, so $C^\alpha$ topology of convergence cannot be improved to $C^1$. However for very well-prepared data one can improve the topology of convergence, though no higher than $C^{2-}$ since the accelerations never converge even for very well prepared data.
\end{remark}
\begin{remark}[Damping]
	One can extend the theorem to damped Newton's equation \cite{B98},
	\be 
	\ddot{X}^\ve(t) + \alpha \dot{X}^\ve=  -\nabla W(X^\ve) - \tfrac{1}{\ve^2} \nabla U(X^\ve(t)).
	\ee 
	In this case the tangential component of the damping survives in Takens limit system, while homogenized potential is time dependent and decays exponentially
	\be
	V(x,t) = e^{-\alpha t}\sum_{i=1}^k c_i \sqrt{\lambda_{i}(x)}.
	\ee 
	Thus damping gradually suppresses the normal oscillations and restores the validity of d'Alembert's principle at long times.
\end{remark}
Here we provide an informal argument to justify the necessity and the explicit form of correcting force. While the full proof can be found in \cite{T80,B98} we present a version in Appendix \ref{adiabaticappend}.
\begin{proof}[Informal argument for the need of correcting force]
	We begin by adjusting $U$ by a constant to make $U = 0$ on $\M$, which we are free to do since it does not change the dynamics.
	
	It is plausible that the solution of \eqref{strong} can be split into the slow motion along $\M$ and fast transverse oscillations as in Figure \ref{cartoon-fig}:
	\be\label{splitting}
	X^\ve = X^\ve_\M + \ve X^\ve_N \quad \text{with}\quad X^\ve_\M = \pi(X^\ve),
	\ee
	where $\pi$ is a nearest point projection defined in a tubular neighborhood of $\M$.
	With this splitting, equation of energy conservation is written as
	\begin{align}
		E_0 &= \frac{1}{2}|\dot{X}^\ve|^2 + W(X^\ve) + \frac{1}{\ve^2}U(X^\ve)\\
		&= \bigg(\frac{1}{2}|\dot{X}_\M^\ve|^2 + W(X_\M^\ve)\bigg) + \bigg(\frac{1}{2}|\ve\dot{X}_{N}^\ve|^2  + \frac{1}{2}\hess_{X_\M} U (X^\ve_N,X^\ve_N)\bigg)  + O(\ve)\\
		&=: E^\ve_\M + E^\ve_N + O(\ve).
	\end{align}
	We used Taylor expansion of potential terms, and noted that the kinetic cross term is lower order:
	\be 
	\langle \dot{X}_\M^\ve, \ve \dot{X}_{N}^\ve \rangle = \langle \dot{X}_\M^\ve, \dot{X}^\ve - \dot{X}^\ve_\M\rangle  = \langle \dot{X}_\M^\ve, (I - D\pi(X^\ve))\dot{X}^\ve \rangle  = \langle \dot{X}_\M^\ve, (\mathbf{P}_{N_{X^\ve_\M}} + O(\ve))\dot{X}^\ve \rangle = O(\ve).
	\ee 
	Suppose that $X^\ve_\M$ converges strongly in $C^1$ to a limiting curve $X_\M \in \M$; then 
	\be 
	E_\M^\ve \to E_\M := \frac{1}{2} |\dot{X}_\M|^2 + W(X_\M).
	\ee
	On the other hand, we can plug in the ansatz \eqref{splitting} into \eqref{strong} and expand using Taylor formula (see Lemma \ref{lemeqns} for details) to check that the normal motion satisfies
	\be 
	\ve^2 \ddot{X}_N^\ve + \hess_{X_\M^\ve} U (X^\ve_N,\cdot) = O(\ve).
	\ee 
	If one believes that $X_\M^\ve$ is indeed slow (specifically, $\dot{X}_\M^\ve = O(1)$, i.e. no rapid oscillations in the tangential to $\M$ directions arise), then, upon rescaling time by $\tau = t/\ve$, this equation is the perturbation of a slowly modulated harmonic oscillator. We write it (in the case when minimum set is codimension one) as
	\be\label{slowlymod}
	\ddot{X}_N^\ve(\tau) +  \omega^2(X_\M^\ve) X^\ve_N(\tau) = O(\ve),
	\ee 
	where $\omega(X_\M^\ve) = \sqrt{U''(X_\M^\ve)}$ is varying slowly in $\tau$. It is well known (see, e.g. the classical book of Arnold \cite{A13}) that on the long time scale $\tau \sim 1/\ve$ (equivalently, $t \sim 1$), slowly modulated oscillator need not to conserve energy $E_N^\ve$. On the other hand, it approximately conserves the  \textit{adiabatic invariant} or \textit{action}, named $c$,  that is given as the ratio of energy and frequency 
	\be 
	c = \frac{E^\ve_N}{ \omega(X_\M^\ve)}.
	\ee 
	We remark that while in the case of unperturbed oscillator this conservation is classical, $O(\ve)$ forcing might potentially break it by introducing resonances. This can even happen for the constant frequency case -- for example it is easy to check that the solution of equation 
	\be
	\ddot{x} + \omega^2 x = \ve \cos (\omega \tau) 
	\ee 
	with zero initial data grows in magnitude like $\ve\tau$, so for $\tau \sim 1/\ve$ adiabatic invariant is of order 1, not 0. Nevertheless, with some work it is possible to rule out resonances in our case by carefully analyzing the forcing terms introduced in Taylor expansion. 
	
	In higher codimension, we diagonalize the Hessian matrix, and conclude that each eigencomponent $X_{Ni}$ satisfies the same one-dimensional equation \eqref{slowlymod}, with  
	\be 
	\omega_i^2 = \lambda_i(X_\M^\ve),
	\ee 
	where $\lambda_i(X_\M^\ve)$ are eigenvalues of $\hess_{X_\M^\ve} U\big|_{N_{X_\M^\ve} \M}$. We note that the analysis of resonances in this setting is much more delicate, as, e.g., another component with the same eigenvalue can introduce exactly the type of resonant forcing as in the toy example above. It turns out that excluding the resonances up to order three is sufficient to guarantee that it does not happen.
	
	Nevertheless, after diagonalization and exclusion of resonances, the previous codimension one argument applies to conclude that adiabatic invariants
	\be 
	c_i = \frac{E^\ve_{Ni}}{\sqrt{\lambda_i(X_\M^\ve)}},
	\ee 
	where, naturally,
	\be
	E_{Ni}^\ve = \frac{1}{2}|\ve\dot{X}_{Ni}^\ve|^2 + \frac{1}{2} \omega_i^2(X_{\M^\ve})|X_{Ni}^\ve|^2,
	\ee 
	are approximately conserved.
	
	With that, we return to the energy equation and take the $\ve \to 0$ limit to conclude that
	\begin{align}
		E_0 = E^\ve_\M + E^\ve_N + O(\ve) &= E^\ve_\M + \sum_i E^\ve_{Ni} + O(\ve)\\
		& = E^\ve_\M + \sum_i c_i \sqrt{\lambda_i(X_\M^\ve)} + O(\ve)\\
		\Longrightarrow E_0 &= \frac{1}{2}|\dot{X}_\M|^2 + W(X_\M) + \sum_i c_i \sqrt{\lambda_i(X_\M)}.
	\end{align}
	Thus, the limiting curve $X_\M$ does not satisfy the naive prediction of being a solution to Newton's equation with potential $W$ constrained to $\M$, since the energy conservation law is different. On the other hand, it is plausible that it might be a solution of Newton's equation with modified potential
	\be 
	W(X_\M) + \sum_i c_i \sqrt{\lambda_i(X_\M)}.
	\ee  
	This expectation holds rigorously, as proven in \cite{T80, B98}. We refer to these works for complete treatment, as well as to Appendix \ref{adiabaticappend} for an overview of main ingredients in the proof.
\end{proof}

\begin{example}[Double pendulum spring chain]\label{spring_example}

	Arguably,\footnote{Unfortunately, a single pendulum modeled by an exceedingly  stiff spring cannot serve as such an example for a simple reason: the ideal constraint space is a circle, and potential energy of the spring is symmetric about this set, so that the eigenvalues of the Hessian are constant on the constraint set and so the Takens corrective force is always  zero regardless of the preparedness of the initial data. Possibly the Huygens pendulum or the like would work.}  the simplest nontrivial example of the application of Theorem \ref{maintheorem} is given by a two-spring chain. Consider a mechanical system consisting of two unit point masses connected by a unit-length spring of stiffness $\tfrac{k_2}{\ve^2}$, with the first mass also connected by a unit-length spring of stiffness $\tfrac{k_1}{\ve^2}$ to the fixed pivot (see Figure \ref{figchain}). 
	
					\begin{figure}[h!]\centering
	\includegraphics[width=0.3\columnwidth]{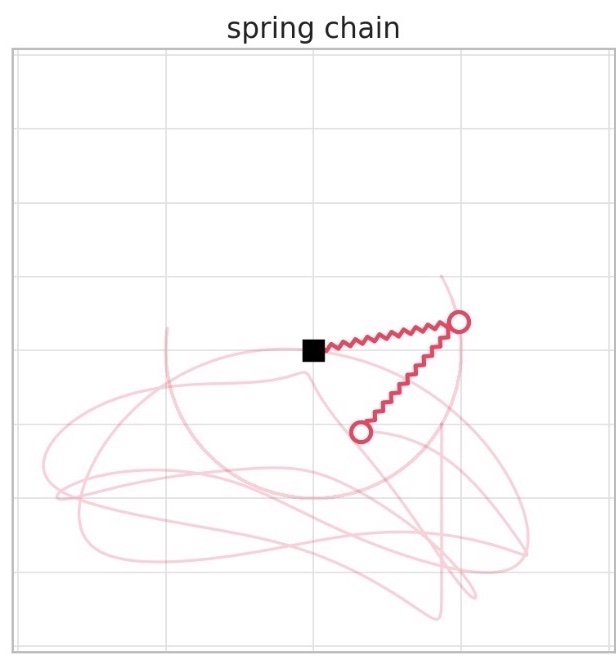} 
		\caption{Spring chain double pendulum with $\ve= 0.01$, $k_1=2$ and $k_2=5$.}\label{figchain}
	\end{figure}

	Viewed as the motion in $\R^4$ with coordinates 
$ 
	\mathbf{z} = (z_1, z_2) = (x_1, y_1, x_2, y_2),
$
	the equations of motion are given by Hooke's law 
	\be\label{spring_mass}
	\begin{cases} 
	\ddot{z}_1 &= -\frac{1}{\ve^2}k_1 (|z_1| -1)\frac{z_1}{|z_1|} - \frac{1}{\ve^2}k_2(|z_2 - z_1|-1)\frac{z_1-z_2}{|z_1 - z_2|}\\
	\ddot{z}_2 &= - \frac{1}{\ve^2}k_2(|z_2 - z_1|-1)\frac{z_2-z_1}{|z_1 - z_2|}
	\end{cases} .
	\ee
	They are Newton's equations with potential 
	\be 
	\tfrac{1}{\ve^2} U(\mathbf{z}) = \tfrac{1}{\ve^2}\tfrac{1}{2} \bigg(k_1 (|z_1| -1)^2 + k_2 (|z_2 - z_1|-1)^2\bigg),
	\ee 
	which is constraining, in the sense of Definition \ref{constraining}, to the torus $\M$ in $\R^4$ given by 
	\be 
	\M  = \{(z_1, z_2) \in \R^4 \ | \ |z_1| = 1, \ |z_2 - z_1| = 1\}.
	\ee 
	Geodesics on $\M$ are of course given by the trajectories of the ideal rigid double pendulum:
	\be\label{double_pend}
	\begin{cases} 
	\ddot{z}_1 &= - \lambda_1 z_1 - \lambda_2 (z_1 - z_2)\\
	\ddot{z}_2 &= - \lambda_2 (z_2 - z_1)\\
	|z_1| &= |z_1 - z_2|= 1
	\end{cases} .
	\ee
			\begin{figure}[h!]\centering
	\includegraphics[width=0.9\columnwidth]{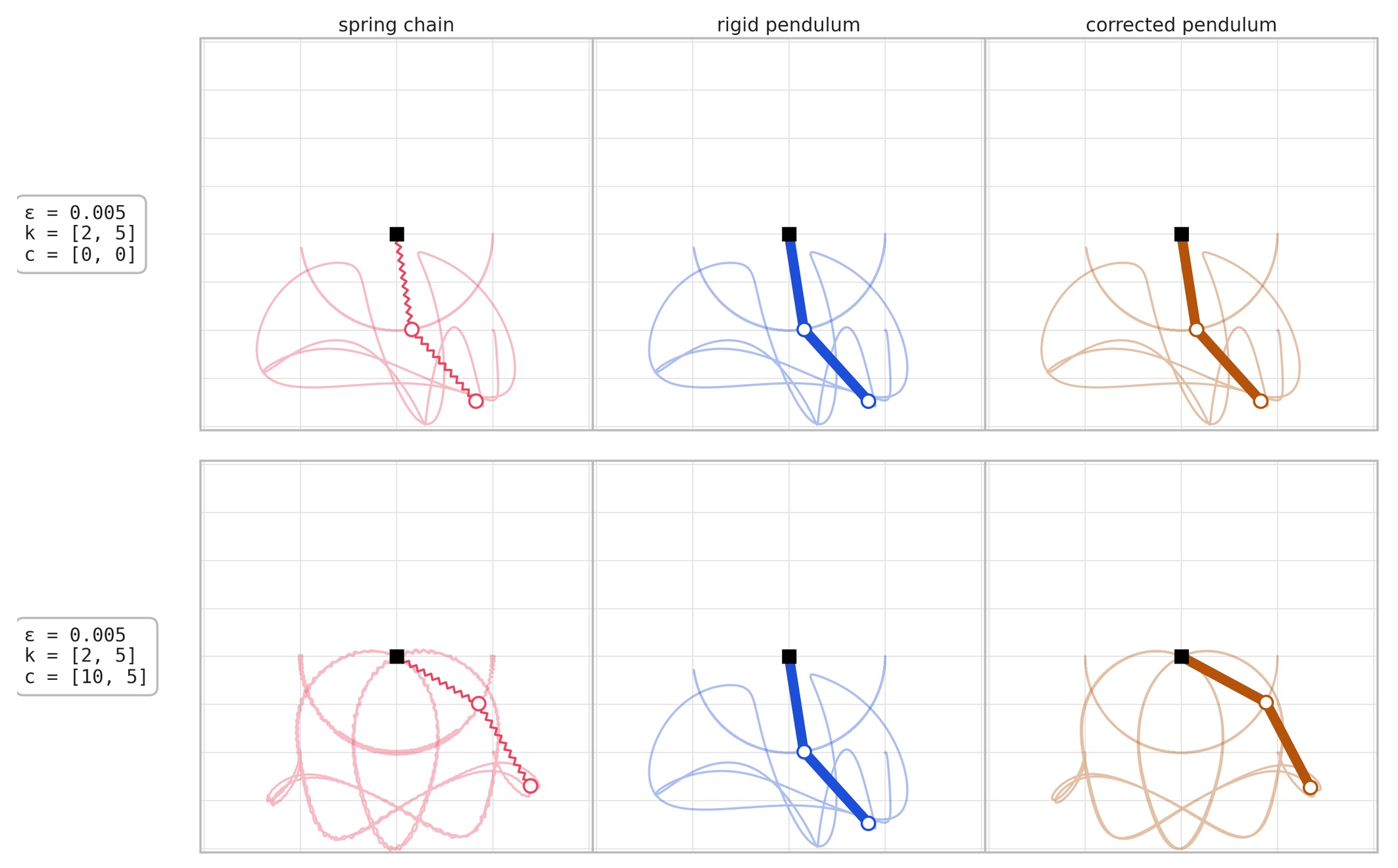} 
		\caption{Three systems: spring chain, rigid double pendulum, Takens system. Row 1: $c_+ =0$, $c_- = 0$, all three systems agree, d'Alembert's principle is justified.\\ Row 2: $c_+ \neq0$, $c_- \neq 0$, rigid pendulum fails to capture the dynamics. The Takens correction is necessary.}\label{figchain2}
	\end{figure}
	According to Theorem \ref{maintheorem}, the limit of solutions of \eqref{spring_mass} as stiffness $ \approx \tfrac{1}{\ve^2}$ increases is given by Takens system that we now derive explicitly. For the homogenized potential, we first compute that in the basis $(\nabla |z_1|, \nabla |z_2 - z_1|)$ of $N\M$, Hessian of $U$ is given by
	\be 
	\hess U|_{N_x \M}= \begin{pmatrix}
		k_1 & k_1 \cos\theta \\
		k_2\cos\theta & 2k_2
	\end{pmatrix} \quad \text{where} \qquad \theta = \theta(\mathbf{z}) = \arccos\langle z_1,z_1 - z_2\rangle.
	\ee 
	Consequently its eigenvalues are given by 
	\be
	\lambda_{\pm} = \tfrac{2k_2 + k_1 \pm \sqrt{(2k_2 - k_1)^2 + 4k_1k_2\cos^2\theta}}{2}.
	\ee 
	Resonance condition is thus an algebraic condition on $(k_1, k_2)$. Consider first a non-resonant situation; in this case with the initial conditions scaled as in Theorem \ref{maintheorem}, the limit of solutions of systems \eqref{spring_mass} as $\ve \to 0$ is given by the trajectory of a double pendulum subject to the force of additional potential
	\begin{align}
	V(\mathbf{z}) &= c_+ \sqrt{\tfrac{2k_2 + k_1 + \sqrt{(2k_2 - k_1)^2 + 4k_1k_2\cos^2\theta(\mathbf{z})}}{2}}  + c_-\sqrt{\tfrac{2k_2 + k_1 - \sqrt{(2k_2 - k_1)^2 + 4k_1k_2\cos^2\theta(\mathbf{z})}}{2}}.
	\end{align}

See Figure \ref{figchain2} for a comparison of the spring system to the double pendulum given by d'Alembert's principle, and to Takens limit system for the well-prepared and ill-prepared data.  In the well-prepared case, it is clearly an unnecessary modification as there is no additional force ($c_\pm=0$).  In the ill-prepared case, since $c_\pm\neq 0$ the system is unlike the double pendulum, but the derived force due to homogenization captures this deviation.

	One observes a curious behavior of this system: if only the positive mode is excited (i.e. $c_+ >0$, $c_- = 0$), configurations with $\theta(\mathbf{z}) = 90^\circ$ are minima of $V$ and hence Lyapunov stable (modulo rigid rotations under which the system is invariant). On the other hand, if only the negative mode is excited, configurations in which $\theta(\mathbf{z}) = 0^\circ$ or $180^\circ$ are Lyapunov stable. See Figure \ref{figchain3} for a demonstration of these phenomena in the non-resonant regime.
			\begin{figure}[h!]\centering
		\includegraphics[width=0.9\columnwidth]{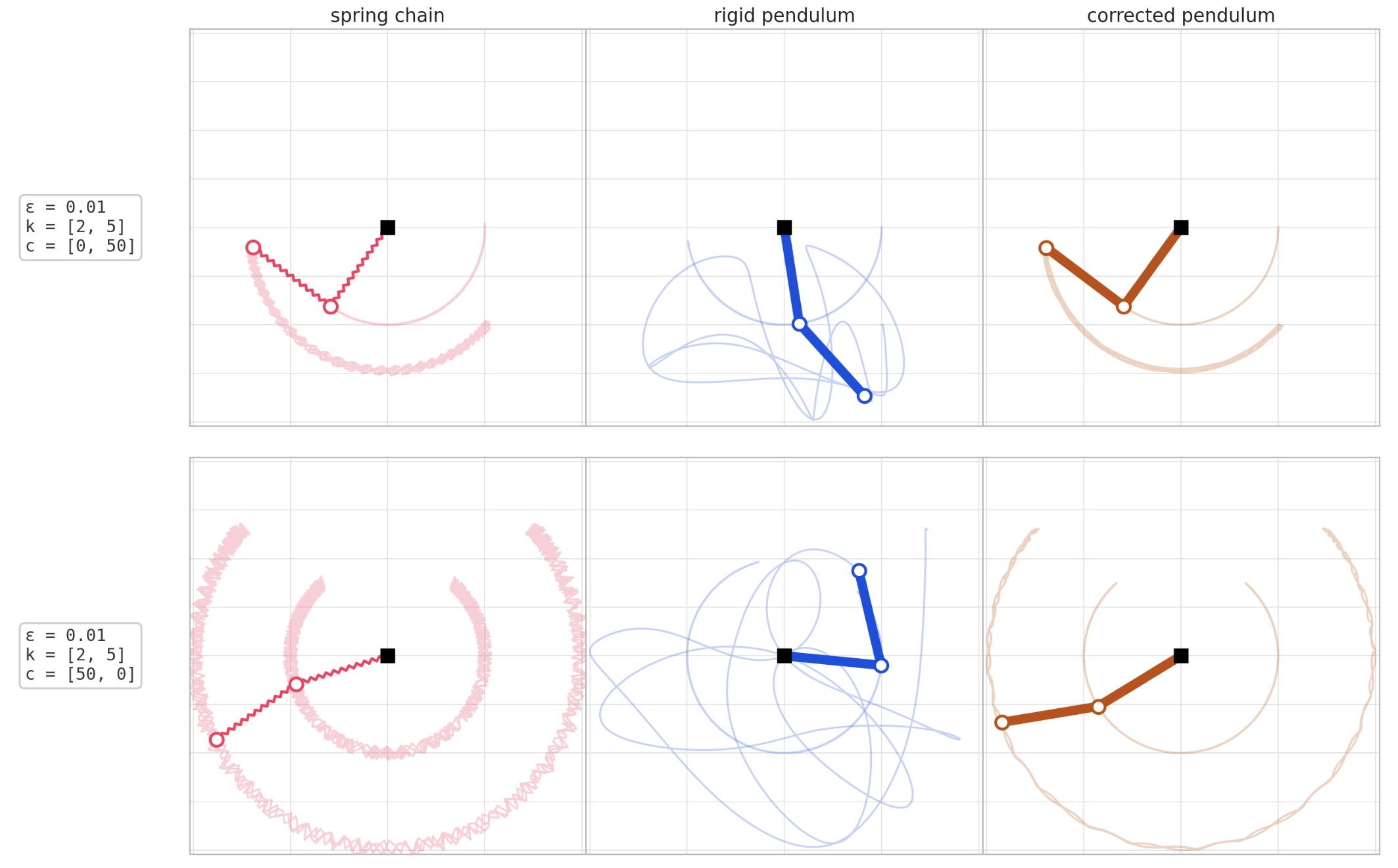} 
		\caption{Three systems: spring chain, rigid double pendulum, Takens system. Row 1: $c_+ >0$, $c_- = 0$. Rigid double pendulum fails to capture the dynamics of stabilization of right angle observed in spring system.\\
		Row 2: $c_+ =0$, $c_- > 0$. Rigid double pendulum fails to capture the dynamics of stabilization of $180^\circ$ angle observed in spring system. The Takens correction is necessary.}\label{figchain3}
	\end{figure}

We also comment on resonances that can happen. An example resonant configuration is shown in Figure \ref{figchain4}. With parameters $k_1 = 2, k_2 = 4$, the system comes to a resonance of order 3 every time  $\cos \theta(\mathbf{z}) = 0$ (in such configuration $\sqrt{\lambda_+} - 2\sqrt{\lambda_-} = 0$). This resonance is non-transversal, as can be seen from the explicit form of potential. This destroys the adiabatic invariance of action. Neither rigid pendulum nor Takens' system capture the limiting behavior, as can be seen in Figure \ref{figchain4}. In fact, we believe that the limiting behavior should be non-unique and might depend on a subsequence along which $\ve \to 0$, a phenomenon dubbed `Takens chaos' \cite{T80}.
			\begin{figure}[h!]\centering
	\includegraphics[width=0.9\columnwidth]{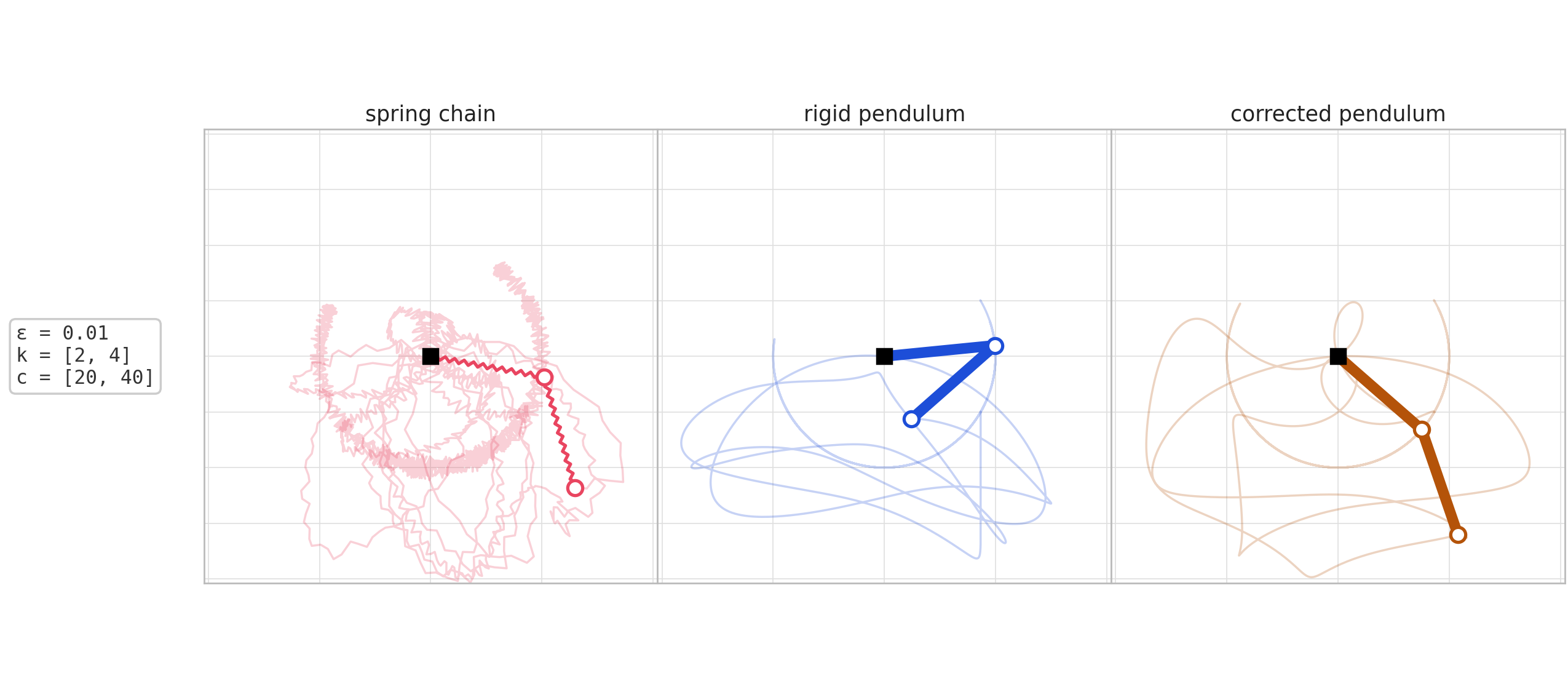} 
	\caption{Three systems: spring chain, rigid double pendulum, Takens system. $k_1 = 2$, $ k_2 = 4$, $c_+ >0$, $c_- > 0$. Non-transversal order three resonance happens at each right-angle configuration, both constrained systems fail to capture behavior of the spring system.}\label{figchain4}
\end{figure}
\end{example}

\newpage
	\section{Extensible Thread and inextensible limit}
	
	Our first continuum-mechanical example is the motion of a stiff elastic thread.  It is indeed a natural generalization of our spring pendulum -- the continuum thread can be viewed as the limit of a spring chain with infinitely many links of infinitesimal length.  Below we describe the continuum system directly, without appealing to this interpretation.
	
	Let $M$ be a compact Riemannian manifold.
	The equation of a hyperelastic extensible loop moving on $M$ is given by
	\be 
	\ddot{X} = \frac{1}{\ve^2}\frac{1}{\vr_0}\nabla_a\bigg(\frac{W'(|\partial_a X|,\cdot)\partial_a X}{|\partial_a X|}\bigg).
	\ee 
	Here $X$ is a time-dependent map $S^1 \to M$ representing the loop position, and $\vr_0: S^1 \to \R^+$, $W: \R \times S^1 \to \R$ are fixed functions with the meaning of density distribution and stored internal energy per unit length respectively. $1/\ve^2$ is a ``spring constant", defining the strength of resistance to extension. 
	Here and below without explicit mention loops are parameterized by $a \in S^1$, and $\partial_a X$ should be interpreted in this fixed parametrization. We will also take the conventions that 
	\be 
	\nabla_a = \nabla_{\partial_a X},
	\ee 
	and all accelerations are understood in the covariant sense
	\be 
	\ddot{X} := \nabla_{\dot{X}}\dot{X}.
	\ee 
	Note that since the metric of $\mathsf{Loops}(M)$ is the pointwise weighted $L^2$ metric, its Levi-Civita connection corresponds to the pointwise Levi-Civita connection of the metric on $M$, so the computations performed thereafter are unambiguous (\cite{P12, B18}).
	
	In Section \ref{extthread-sect}, we show that this equation can be derived as a Newton's equation on the space of loops on $M$; thus, it makes sense to find the corresponding Takens limit. 
		\begin{theorem}\label{threadthm}
		The homogenized potential of $U$ \eqref{threadpot} is given by 
		\be 
		V[X] = \sum c_i \sqrt{\lambda_i[X]},
		\ee 
		where $\lambda_i[X]$ are eigenvalues of the operator 
		\be
		W''(1,\cdot)\L_X = W''(1,\cdot)\bigg(-\partial_a\bigg(\frac{1}{\vr_0}\partial_a\bigg) + \frac{1}{\vr_0}|\nabla_a\partial_a X|^2\bigg).
		\ee
Suppose $\lambda_i$ are simple; then the Takens limit system of extensible thread equation is given by modified inextensible thread
		\be\label{takens-thread}
\begin{cases}
	\ddot{X} =   \frac{1}{\vr_0}\nabla_a(\sigma \partial_a X) - \frac{1}{\vr_0}\nabla_a\nabla_a(K \nabla_a\partial_a X) + \frac{K}{\vr_0} R(\partial_a X, \nabla_a \partial_a X)\partial_a X\\
	|\partial_a X| = 1
\end{cases}
\ee
		where $\sigma$ is determined to enforce inextensibility, $R$ is a curvature tensor of manifold $M$, and 
		\be 
		K = \sum_i \frac{ c_i v_i^2}{\sqrt{\lambda_i}\vr_0}
		\ee 
		where $v_i$ and $\lambda_i$ are $L^2_{1/W''(1,\cdot)}$ normalized eigenfunctions/values of operator $W''(1,\cdot)\L$
		and the $c_i$'s are constants computed from the data for extensible thread (see Proposition \ref{thread_ci}).
	\end{theorem}
	\subsection{Extensible thread}\label{extthread-sect}
	Let $\mathsf{Loops}(M)$ be the Fréchet manifold of smooth immersions of $S^1$ in $M$. The tangent space to $\mathsf{Loops}(M)$ at point $X$ is given by infinitesimal deformations of parametrized loops in $M$, that is by vector fields along the loop $X$:
	\be 
	T_X \mathsf{Loops}(M) = \{\xi  \ | \ (X,\xi): S^1 \to TM \text{ with }  \xi(a) \in T_{X(a)}M\}.
	\ee  
	Given a smooth positive function $\vr_0$ on $S^1$, denote by $\mathsf{Loops}_{\vr_0}(M)$ the space $\mathsf{Loops}(M)$ endowed with the (weak) Riemannian metric given at every point $X$ by 
	\be 
	(\xi, \eta)_{X} = \int_{S^1} \langle \xi(a), \eta(a)\rangle \vr_0(a)\rmd a.
	\ee 
	
	This metric corresponds to the kinetic energy of an inhomogeneous thread with density $\vr_0$ located at $X(S^1)$.
	We now describe the (hyperelastic)\footnote{An elastic material is called hyperelastic if the stress only depends on strain $\partial_a X$ but not on the higher derivatives. In particular, such materials do not resist bending \cite{Al95}.} extensible thread equation as  Newton's law on $\mathsf{Loops}(M)$. 
	
	Let $W: \R \times S^1 \to \R$ be a given smooth function. The physical meaning of it is the following: for a thread parametrized in the reference state by $a \in S^1$, $W(\lambda, a)\rmd a$ is the internal energy stored in the piece of thread with reference coordinates $[a, a+\rmd a]$ if this piece of thread is stretched to a length $\lambda\rmd a $. The dependence of $W$ on $a$ represents nonhomogeneity of thread, either through the non-uniformity of material, or variations of density/cross section. We assume that \be 
	W(1, \cdot) = 0 \quad, W \geq 0, \quad W''> 0,
	\ee 
	where primes denote derivatives with respect to the first argument. These correspond respectively to assumptions that the non-extended thread does not possess internal energy, that stored energy is positive and that stress increases with extension. Note that from the stated convexity and positivity, we also have 
	\be 
	W'(1,\cdot) = 0,
	\ee 
	i.e. unstretched thread is free of stresses.
	
	Consider the potential  $U[X]$ corresponding to the total strain energy of the ``rubber band":
	\be\label{threadpot}
	U[X] = \int_{S^1} W(|\partial_a X(a)|, a)\rmd a
	\ee 
	Newton's law on $\mathsf{Loops}_{\vr_0}(M)$ with the conservative force field given by this potential provides the equation of motion of hyperelastic extensible thread. We check it below explicitly:
	\begin{lemma}
		The (functional) gradient of the potential is
		\be 
		\grad_{\vr_0}U[X] = - \frac{1}{\vr_0}\nabla_a\bigg(\frac{W'(|\partial_a X|,\cdot)\partial_a X}{|\partial_a X|}\bigg).
		\ee 
	\end{lemma}
	\begin{proof}
		Consider a variation $X^\ve$ of loop $X$ satisfying 
$
		\frac{\rmd }{\rmd \ve}\big|_{\ve = 0} X^\ve = \xi. 
	$
		We compute the corresponding variation of $U$:
		\begin{align}
		\frac{\rmd }{\rmd \ve}\bigg|_{\ve = 0} U[X^\ve] &= \int_{S^1} W'(|\partial_a X(a)|,a)\frac{\rmd }{\rmd \ve}\bigg|_{\ve = 0} |\partial_a X^\ve(a)|\rmd a\\
		&= \int_{S^1} W'(|\partial_a X(a)|, a)\frac{\langle\partial_a X(a), \nabla_a\xi(a)\rangle}{|\partial_a X(a)|}\rmd a\\
		&= -\int_{S^1} \bigg\langle \nabla_a\bigg(W'(|\partial_a X(a)|,a)\frac{\partial_a X(a)}{|\partial_a X(a)|}\bigg), \xi(a)\bigg\rangle\rmd a\\
		&=-\int_{S^1} \bigg\langle \frac{1}{\vr_0(a)}\nabla_a\bigg(\frac{W'(|\partial_a X(a)|,a)\partial_a X(a)}{|\partial_a X(a)|}\bigg), \xi(a)\bigg\rangle\vr_0(a)\rmd a.
		\end{align}
	Claimed expression for the gradient follows. 
	\end{proof}
	With that, the postulated equation of motion  is written as
	\be 
	\ddot{X} = \frac{1}{\vr_0}\nabla_a\bigg(\frac{W'(|\partial_a X|, a)\partial_a X}{|\partial_a X|}\bigg).
	\ee  
	In the case of $M = \R^3$, this equation agrees with the ones usually studied, as derived, e.g. in \cite{Al95}. 
	\subsection{Inextensible thread}
	Consider submanifold  $\mathsf{SLoops}(M)$ of $\mathsf{Loops}(M)$ given by 
	\be 
	\mathsf{SLoops}(M) = \{X\in \mathsf{Loops}(M) \ | \ |\partial_a X(a)| = 1 \text{ for all } a \in S^1\}.
	\ee 
	
	Naturally, $\mathsf{SLoops}_{\vr_0}(M)$ will denote $\mathsf{SLoops}(M)$ endowed with metric $(\cdot, \cdot)_{\vr_0}$. The inextensible thread equation is the geodesic equation on $\mathsf{SLoops}_{\vr_0}(M)$.
	\begin{prop}\label{threadtangentnormal}
		The tangent and normal spaces to  $\mathsf{SLoops}_{\vr_0}(M)$ are given by
		\begin{align}
		 T_X \mathsf{SLoops}_{\vr_0}(M) &= \{\xi \in T_X \mathsf{Loops}_{\vr_0}(M) \ | \ \langle \nabla_a\xi, \partial_a X\rangle \equiv 0  \} \text{ and }\\
		 N_X \mathsf{SLoops}_{\vr_0}(M) &= \bigg\{\frac{1}{\vr_0}\nabla_a(\sigma \partial_a X) \ | \ \sigma \in C^\infty(S^1,\R) \bigg\}.
		\end{align}
	\end{prop}
\begin{proof}
	Consider a curve $X^\ve \in \mathsf{SLoops}_{\vr_0}(M)$, with $X^0 = X$, and let $\xi$ be its derivative at $X$. For any point $a \in S^1$ we compute 
	\be
	0 = \frac{\rmd }{\rmd \ve}\bigg|_{\ve = 0} \frac{1}{2}|\partial_a X^\ve(a)|^2 = \langle \nabla_a \xi(a), \partial_a X(a)\rangle,
	\ee  
	which shows the first assertion.  To find the normal space we use the following lemma which is analogous to Helmholtz decomposition of vector fields. 
\begin{lemma}
	Any smooth vector field $\xi$ along $X \in \mathsf{SLoops}_{\vr_0}(M)$ can be uniquely written as 
	\be 
	\xi = \eta + \frac{1}{\vr_0}\nabla_a(\sigma \partial_a X)
	\ee 
	where $ \langle \nabla_a \eta, \partial_a X \rangle = 0$ and $\sigma \in C^\infty(S^1)$, 
	and the two terms are $(\cdot ,\cdot )_{\vr_0}$--orthogonal.
\end{lemma}
\begin{proof}
	We differentiate this ansatz along $X$, and take inner product with $\partial_a X$ to obtain
\be 
\langle \nabla_a \xi, \partial_a X \rangle = \bigg\langle \nabla_a\bigg(\frac{1}{\vr_0}\nabla_a(\sigma \partial_a X)\bigg), \partial_a X \bigg\rangle = \partial_a\bigg(\frac{1}{\vr_0}\partial_a\sigma\bigg) - \frac{1}{\vr_0}|\nabla_a\partial_a X|^2\sigma =: -\L_X \sigma.
\ee   
Suppose $X$ is not a geodesic, so the zeroth-order term is not identically zero. In this case operator $\L_X$ is positive and self-adjoint, so standard elliptic theory provides unique smooth $\sigma$ for any $\xi$.

Thus given $\xi$, the claimed decomposition is given by setting $\sigma = -\L_X^{-1}\langle \nabla_a \xi, \partial_a X \rangle$ and $\eta = \xi - \frac{1}{\vr_0}\nabla_a(\sigma \partial_a X)$. It is unique since $\sigma$ is uniquely determined by above.
The summands are orthogonal with respect to $\vr_0$ - weighted $L^2$ inner product by construction:
\begin{align}
	\bigg(\frac{1}{\vr_0}\nabla_a(\sigma \partial_a X), \eta\bigg)_{\vr_0} &= \int_{S^1} \bigg\langle \frac{1}{\vr_0}\nabla_a(\sigma \partial_a X), \eta\bigg\rangle \vr_0 \rmd a\\
	&=- \int_{S^1} \sigma\langle  \partial_a X, \nabla_a\eta \rangle  \rmd a =0.
\end{align}
In case $X$ is a geodesic, $\L_X$ has a kernel that consists of constant vector fields. Thus uniqueness on the level of $\sigma$ no longer strictly holds;  however note that in this case adding a constant to $\sigma$ does not change $\frac{1}{\vr_0}\nabla_a(\sigma \partial_a X)$, so on the level of claimed decomposition the uniqueness is restored. 
\end{proof}

Now let $\xi$ be any vector field along $X$. Applying the above lemma, we see to be orthogonal to $T_X \mathsf{SLoops}_{\vr_0}$ the first summand must be zero so we must have
$ 
 \xi = \frac{1}{\vr_0}\nabla_a(\sigma \partial_a X)
$
as claimed.
\end{proof}
With that, we can write inextensible thread equation by d'Alembert's principle:

\be\label{inextthread}
\left\{\begin{aligned} 
	\ddot{X} &\in N_X\mathsf{SLoops}_{\vr_0}(M) \\
	X &\in \mathsf{SLoops}_{\vr_0}(M)
\end{aligned}\right. \iff 
\left\{\begin{aligned}
	\ddot{X} &= \frac{1}{\vr_0}\nabla_a(\sigma \partial_a X)\\
	|\partial_a X| &= 1
\end{aligned}\right.
\ee
For completeness, we write down the equation satisfied by Lagrange multiplier $\sigma$ (the \textit{tension}).
\begin{prop}
	Let $X$ solve \eqref{inextthread}. Then $\sigma$ is recovered by solving the elliptic equation
	\be 
	\L_X \sigma = |\nabla_a \dot{X}|^2 + \langle R(\partial_a X,\dot{X})\partial_a X, \dot{X}\rangle,
	\ee 
	where the operator $\L_X$  is defined by
	\be\label{Lthread}
		\L_X := \bigg(-\partial_a\bigg(\frac{1}{\vr_0}\partial_a\bigg) + \frac{1}{\vr_0}|\nabla_a\partial_a X|^2\bigg).
	\ee
\end{prop}
\begin{proof}
	We time-differentiate the constraint $|\partial_a X|^2 = 1$ twice to obtain
	\begin{align} 
	\langle \nabla_a\dot{X}, \partial_a X\rangle &= 0\\
	|\nabla_a \dot{X}|^2 + \langle \nabla_a\ddot{X}, \partial_a X\rangle +\langle R(\partial_a X,\dot{X})\partial_a X, \dot{X}\rangle&= 0.
	\end{align}
Plugging in $\ddot{X}$ from the equation, we obtain an equation for $\sigma$:
\be 
-\bigg\langle \nabla_a\bigg(\frac{1}{\vr_0}\nabla_a(\sigma \partial_a X)\bigg), \partial_a X\bigg\rangle  =  |\nabla_a \dot{X}|^2 + \langle R(\partial_a X,\dot{X})\partial_a X, \dot{X}\rangle
\ee  
Using the fact that $X$ is parametrized by arclength, left hand side is simplified to 
\begin{align}
	-\bigg\langle \nabla_a\bigg(\frac{1}{\vr_0}\nabla_a(\sigma\partial_a X)\bigg),\partial_a X\bigg\rangle & = 
	-\partial_a\frac{1}{\vr_0} \bigg\langle \nabla_a(\sigma\partial_a X),\partial_a X\bigg\rangle -\frac{1}{\vr_0} \bigg\langle \nabla_a\nabla_a(\sigma\partial_a X),\partial_a X\bigg\rangle\\
	&=-\partial_a\frac{1}{\vr_0} \partial_a \sigma -\frac{1}{\vr_0}\partial^2_a \sigma + \frac{1}{\vr_0}|\nabla_a\partial_a X|^2 \sigma\\
	&= -\partial_a\bigg(\frac{1}{\vr_0} \partial_a \sigma\bigg) + \frac{1}{\vr_0}|\nabla_a\partial_a X|^2 \sigma = \L_X \sigma.
\end{align} 
This completes the derivation.
\end{proof}
	For more information on the inextensible thread system, see the excellent works \cite{P11, P12, S}.

	\subsection{Inextensible limit: Takens system}
	We now consider making the extensible thread more and more inextensible by stiffening the potential $U$. To write down Takens limit system, we need a description of critical set $\M$ of the potential \eqref{threadpot}, as well as of $\hess_{\vr_0} U[X]$ on it. Before we proceed to it, we make a technical assumption:
	\begin{assump}\label{thread-assump}
		We assume that the length of the unstretched loop $\ell_0$ ($\ell_0 = 2\pi$ in our normalization) does not belong to the length spectrum of $M$.
	\end{assump}
	
	The reason for it is that we would like to avoid considering closed geodesics, and it is clear from the formula \eqref{threadgrad}, as well as the physical intuition, that they are critical points of $U$: a rubber band placed along a geodesic will stay still. However, generically they are isolated inside critical set $\M$ and differential-geometric constructions mostly do not make sense.
	With this assumption, we outline how to avoid closed geodesics in the discussion of Takens limit. 
	
	\begin{enumerate}
		\item  For the purpose of considering slightly extensible thread we can replace the configuration space $\mathsf{Loops}(M)$ by
		\be 
		\widetilde{\mathsf{Loops}}(M)_{\delta} := \mathsf{Loops}(M) \cap \{ U < \delta \},
		\ee 
		where $\delta$ is a sufficiently small constant. Due to energy conservation, for sufficiently stiff thread ($\ve^2 \ll \delta < 1$), solutions starting with mildly ill-prepared initial data will stay in $\widetilde{\mathsf{Loops}}(M)_{\delta}$, so restricting the configuration space does not affect the dynamics in this regime. 
		\item With that, we claim that critical set $\tilde{\M}_\delta$ of $U$ inside $\widetilde{\mathsf{Loops}}(M)_{\delta}$ is just given by $\mathsf{SLoops}(M)$. To see that we first observe that $\widetilde{\mathsf{Loops}}(M)_{\delta}$ is open in $\mathsf{Loops}(M)$, so $\tilde{\M}_\delta$ is a subset of $\M$. 
		\item Thus it is enough to show that there are no closed geodesics in $\widetilde{\mathsf{Loops}}(M)_{\delta}$. Indeed, since $W''(1,\cdot)$ is positive, we can control  
		\be 
		|\mathsf{Length}(X) - \ell_0|= \bigg|\int_{S^1} (|\partial_a X| - 1) \rmd a \bigg| \leq \sqrt{2\pi} \bigg(\int_{S^1} (|\partial_a X| - 1)^2\rmd a \bigg)^{1/2} \lesssim \sqrt{U[X]} = \sqrt{\delta}.
		\ee  
		Thus lengths of curves in $\widetilde{\mathsf{Loops}}(M)_{\delta}$ equal to $\ell_0$ up to a correction of order $\sqrt{\delta}$, so taking $\delta$ small enough and using that $\ell_0$ is not in the length spectrum of $M$ shows that there are no closed geodesics in $\widetilde{\mathsf{Loops}}(M)_{\delta}$.
	\end{enumerate} 
	This argument shows that while closed geodesics are critical points of $U$, assumption \ref{thread-assump} renders them dynamically inaccessible in the regime of consideration. Thus from now on we will consider $\widetilde{\mathsf{Loops}}(M)_{\delta}$ as our configuration space. We will also abuse the notation and remove tildes and subscripts $\delta$. We proceed to investigate critical set of $U$.
	\begin{lemma}
	The critical set  $\M$ of the potential \eqref{threadpot} is given by $\mathsf{SLoops}_{\vr_0}(M)$.
	\end{lemma} 
\begin{proof}
	Let $X \in \M$. By definition of $\M$, it means that
	\be\label{threadgrad}
	\grad_{\vr_0} U[X] = 0 \Leftrightarrow \nabla_a\bigg(\frac{W'(|\partial_a X(a)|, a)\partial_a X(a)}{|\partial_a X(a)|}\bigg) = 0.
	\ee
	First, we have 
	\begin{align}
	\frac{\rmd}{\rmd a} W'(|\partial_a X(a)|, a) &= \frac{\rmd}{\rmd a} \bigg\langle \frac{W'(|\partial_a X(a)|, a)\partial_a X(a)}{|\partial_a X(a)|}, \frac{\partial_a X(a)}{|\partial_a X(a)|}\bigg\rangle \\
	&=\bigg\langle \nabla_a\bigg(\frac{W'(|\partial_a X(a)|, a)\partial_a X(a)}{|\partial_a X(a)|}\bigg), \frac{\partial_a X(a)}{|\partial_a X(a)|}\bigg\rangle \\
	&\qquad\qquad+\bigg\langle \frac{W'(|\partial_a X(a)|, a)\partial_a X(a)}{|\partial_a X(a)|}, \nabla_a\bigg(\frac{\partial_a X(a)}{|\partial_a X(a)|}\bigg)\bigg\rangle \\
	&=\frac{1}{2}W'(|\partial_a X(a)|, a)\frac{\rmd}{\rmd a}\bigg|\frac{\partial_a X(a)}{|\partial_a X(a)|}\bigg|^2 = 0.
	\end{align}
	This implies that $W'(|\partial_a X(a)|, a) = {\rm const}$. Thus \eqref{threadgrad} implies that 
	\be 
	W'(|\partial_a X(a)|, a) \nabla_{a}\bigg(\frac{\partial_a X(a)}{|\partial_a X(a)|}\bigg) = 0.
	\ee 
	If the first term is a nonzero constant, this implies that $X$ is a geodesic (which we are not considering). Otherwise, due to convexity and nonnegativity of $W$
	\be 
	W'(|\partial_a X(a)|, a) = 0 \ \Rightarrow \ |\partial_a X(a)| = 1,
	\ee 
	so $X \in \mathsf{SLoops}_{\vr_0}(M).$
	\end{proof}
\begin{lemma}\label{hessthread}
	On $\mathsf{SLoops}_{\vr_0}(M)$ we have 
	\be
	\hess_{\vr_0} U[X](\xi, \cdot)=-\frac{1}{\vr_0}\nabla_a (W''(1, \cdot)\langle\partial_a X, \nabla_a\xi\rangle\partial_a X).
	\ee
	When restricted to $N_X\mathsf{SLoops}_{\vr_0}(M)$, $\hess_{\vr_0} U[X]$ is positive definite.
\end{lemma}
\begin{proof}
	Consider again a variation $X^\ve$ of $X$ with derivative $\xi$ at $X$. We then compute on $\mathsf{SLoops}_{\vr_0}(M)$
	\begin{align}
		\hess_{\vr_0} U[X](\xi, \cdot) &= \frac{\rmd }{\rmd \ve}\bigg|_{\ve = 0}\grad_{\vr_0} U[X^\ve] = - \frac{\rmd }{\rmd \ve}\bigg|_{\ve = 0}\bigg( \frac{1}{\vr_0}\nabla_a\bigg(\frac{W'(|\partial_a X^\ve|, \cdot)\partial_a X^\ve}{|\partial_a X^\ve|}\bigg)\bigg)\\
		&= -\frac{1}{\vr_0}\nabla_a\frac{\rmd }{\rmd \ve}\bigg|_{\ve = 0}\bigg(\frac{W'(|\partial_a X^\ve|, \cdot)\partial_a X^\ve}{|\partial_a X^\ve|}\bigg) -\frac{1}{\vr_0}R(\xi, \partial_a X)\frac{W'(|\partial_a X|, \cdot)\partial_a X}{|\partial_a X|}\\
		&=-\frac{1}{\vr_0}\nabla_a (W''(1, \cdot)\langle\partial_a X, \nabla_a\xi\rangle\partial_a X),
	\end{align}
	where all terms with $\ve-$derivative not hitting $W$ cancel from the fact that $W'(1,\cdot) = 0 $. Hence
	\begin{align}
		\hess_{\vr_0} U[X](\xi,\xi) &=-\int_{S^1} \bigg\langle\frac{1}{\vr_0}\nabla_a (W''(1, \cdot)\langle\partial_a X, \nabla_a\xi\rangle\partial_a X), \xi \bigg\rangle \vr_0  \rmd a\\
		&=\int_{S^1} W''(1,\cdot)\langle\nabla_a \xi, \partial_a X\rangle^2 \rmd a.
	\end{align}
Plugging in normal vector $\xi = \frac{1}{\vr_0}\nabla_a(\sigma \partial_a X)$ provided by Lemma \ref{threadtangentnormal}, we get
\be
	\hess_{\vr_0} U[X](\xi,\xi) =\int_{S^1} W''(1,\cdot)\bigg\langle\nabla_a \bigg(\frac{1}{\vr_0}\nabla_a(\sigma \partial_a X)\bigg), \partial_a X\bigg\rangle^2 \rmd a
	= \int_{S^1}W''(1,\cdot)(\L_X \sigma)^2 \rmd a,
\ee 
where $\L_{X}$ is again given by \eqref{Lthread}. 
We see that under our assumptions on $W$ and $X$, $\hess_{\vr_0} U[X]$ is positive restricted to $N_X\mathsf{SLoops}(M)$. If $X$ happens to be a closed geodesic, $\L_X$ has kernel consisting of constants -- but in that case $\frac{1}{\vr_0}\nabla_a(\sigma \partial_a X)$ is zero. 
\end{proof}

	\begin{prop}
		Eigenvalues of $\hess_{\vr_0} U[X]$ restricted to $N_X\mathsf{SLoops}_{\vr_0}(M)$ are the same as those of $W''(1, \cdot)\L_X$, where $\L_X$ is given by \eqref{Lthread}. Hence homogenized potential of $U$ is given by 
		\be\label{hompot-thread}
		V[X] = \sum_i c_i\sqrt{\lambda_i[X]},
		\ee 
		where $\lambda_i[X]$ are eigenvalues of $W''(1, \cdot)\L_X$.
	\end{prop}
	\begin{proof}
	Eigenpairs of the Hessian restricted to the normal space satisfy
					\begin{align}
					  \lambda \frac{1}{\vr_0}\nabla_a(\sigma\partial_a X)
&=
\hess_{\vr_0} U[X]\bigg(\frac{1}{\vr_0}\nabla_a(\sigma\partial_a X), \cdot \bigg) 		\\
&=			  -\frac{1}{\vr_0}\nabla_a\bigg(W''(1, \cdot) \bigg\langle\partial_a X, \nabla_a\bigg(\frac{1}{\vr_0}\nabla_a(\sigma\partial_a X)\bigg)\bigg\rangle\partial_a X\bigg) \\
&=  \frac{1}{\vr_0}\nabla_a\bigg(W''(1,\cdot)\L_X \sigma\partial_a X\bigg).
\end{align}
In the above, we used formula for Hessian from  Lemma \ref{hessthread}.
Therefore $\sigma$ must satisfy
			\begin{align}
			 W''(1,\cdot)\L_X \sigma &= \lambda\sigma,
		\end{align}
		completing the proof.
	\end{proof}
	We then compute the gradients of the eigenvalues ${\lambda_i}[X]$ of $W''(1,\cdot)\L_X$. 
	\begin{lemma} For a simple eigenvalue $\lambda_i$
		\be 
		\grad_{\vr_0}\lambda_i[X] = \frac{2}{\vr_0}\nabla_a\nabla_a\bigg(\frac{v_i^2}{\vr_0}\nabla_a\partial_a X\bigg) - \frac{2v_i^2}{\vr^2_0} R(\partial_a X, \nabla_a \partial_a X)\partial_a X,
		\ee 
		where $v_i$ are eigenfunctions of $W''(1,\cdot)\L_X$ normalized by the condition
		\be 
		\int \frac{|v_i|^2}{W''(1,\cdot)} \rmd a= 1.
		\ee 
	\end{lemma}
	\begin{proof}
		Consider a variation $X^\ve$ of $X$, and denote the objects corresponding to $X^\ve$ by the upper index $\ve$, e.g.
		\be 
		\L^{\ve} = \L_{X^\ve}, \quad W''(1,\cdot)\L^\ve v_i^\ve= \lambda_i^{\ve}v_i^\ve, \quad \int \frac{|v_i^\ve|^2}{W''(1,\cdot)} = 1.
		\ee 
		We will use the weighted Hellmann-Feynman formula \ref{HF1}. For that we first compute 
		\begin{align}
		\bigg(\frac{\rmd }{\rmd \ve}\bigg|_{\ve=0} W''(1,\cdot)\L^\ve \bigg)&= W''(1,\cdot)\frac{\rmd }{\rmd \ve}\bigg|_{\ve=0}\bigg( -\partial_a\bigg(\frac{1}{\vr_0} \partial_a\bigg) + \frac{1}{\vr_0}|\nabla_a\partial_a X^\ve|^2\bigg)\\
		&=  2\frac{W''(1,\cdot)}{\vr_0} \langle\nabla_a\nabla_a\xi + R(\xi,\partial_a X)\partial_a X,\nabla_a \partial_a X\rangle .
		\end{align}
	Hence, setting $A^\ve = W''(1,\cdot)\L^\ve$ and $H = L^2\big(S^1, \frac{\rmd a}{W''(1,\cdot)}\big)$ in lemma \ref{HF1} provides
		\begin{align}
		\frac{\rmd }{\rmd \ve}\bigg|_{\ve=0} \lambda^\ve_i &= \int v_i \bigg(\frac{\rmd }{\rmd \ve}\bigg|_{\ve=0} W''(1,\cdot)\L^\ve \bigg)v_i\frac{1}{W''(1,\cdot)} \\
		&= \int \frac{v_i^2}{\vr_0}2\langle\nabla_a\nabla_a\xi + R(\xi,\partial_a X)\partial_a X, \nabla_a\partial_a X\rangle\\
		&=\int \bigg\langle\xi, 2\nabla_a\nabla_a\bigg(\frac{v_i^2}{\vr_0}\nabla_a\partial_a X\bigg) - 2\frac{v_i^2}{\vr_0} R(\partial_a X, \nabla_a \partial_a X)\partial_a X \bigg\rangle.
		\end{align}
		The claimed expression for the gradient follows.
	\end{proof}
Consequently we get
	\begin{prop}\label{threadlimit}
		Gradient of the potential $V[X]$ is given by
		\be 
		\grad_{\vr_0} V[X] = \frac{1}{\vr_0}\nabla_a\nabla_a(K \nabla_a\partial_a X) -  \frac{K}{\vr_0} R(\partial_a X, \nabla_a \partial_a X)\partial_a X,
		\ee 
		where $K = \sum_i  \frac{c_i v_i^2}{\sqrt{\lambda_i}\vr_0}.$
		The Takens limit equation is written as 
		\be
		\begin{cases}
		\ddot{X} =   \frac{1}{\vr_0}\nabla_a(\sigma \partial_a X) - \frac{1}{\vr_0}\nabla_a\nabla_a(K \nabla_a\partial_a X) + \frac{K}{\vr_0} R(\partial_a X, \nabla_a \partial_a X)\partial_a X,\\
		|\partial_a X| = 1,\\
		K = \sum_i \frac{ c_i v_i^2}{\sqrt{\lambda_i}\vr_0}.
		\end{cases}
		\ee
	\end{prop}
	\noindent To complete the description, we must compute adiabatic invariants $c_i$:
	\begin{prop}\label{thread_ci}
		Consider a family of mildly ill-prepared initial data $(X_0^\ve, V_0^\ve) \to (X_0, V_0)$ with
		\begin{align}
			X_0^0 &= X_0 \in \mathsf{SLoops}_{\vr_0}(M), \quad \frac{\rmd}{\rmd\ve}\bigg|_{\ve = 0} X_0^\ve = \frac{1}{\vr_0}\nabla_a(\varphi \partial_a X_0) =: X_{0N} \in N_{X_0}\mathsf{SLoops}_{\vr_0}(M)\\ 
			V_0^0 &= V_0 \in T_{X_0}\mathsf{Loops}_{\vr_0}(M)
		\end{align}
		Then the actions $c_i$ are explicitly computed as  
		\be 
		c_i = \frac{1}{2\lambda_i^{3/2}}\bigg(\int_{S^1} \langle\nabla_a V_0, \partial_a X_0 \rangle v_i \rmd a  \bigg)^2 + \frac{\lambda_i^{3/2}}{2}\bigg(\int_{S^1} \frac{\varphi v_i}{W''(1,\cdot)}  \rmd a \bigg)^2.
		\ee 
	\end{prop}
\begin{proof}
	We recall that 
	\be 
	c_i = \frac{|\mathbf{P}_i(X_0) V_0|^2 +  \lambda_i(X_0)|\mathbf{P}_i(X_0) X_{0N}|^2}{2\sqrt{\lambda_i(X_0)}},
	\ee 
	where $\mathbf{P}_i(X)$ are $L^2_{\vr_0}$ projectors onto eigenspaces of $\hess_{\vr_0} U_{X}.$ As computed above, $(\lambda, \frac{1}{\vr_0}\nabla_a(v \partial_a X))$ is an eigenpair of $\hess_{\vr_0} U \big|_{N_{X}\mathsf{SLoops}_{\vr_0}}(M)$ if and only if 
	$(\lambda, v)$ is an eigenpair of $W''(1,\cdot)\L_{X_0}$. 
	The action constants $c_i$ are thus given by 
	\begin{align}
		c_i &= \frac{|\mathbf{P}_i(X_0) V_0|^2 +  \lambda_i(X_0)|\mathbf{P}_i(X_0) X_{0N}|^2}{2\sqrt{\lambda_i(X_0)}}\\
		&=\frac{1}{2\sqrt{\lambda_i}}\frac{\big(\frac{1}{\vr_0}\nabla_a(v_i\partial_a X_0), V_0\big)^2_{\vr_0}}{ \|\frac{1}{\vr_0}\nabla_a(v_i\partial_a X_0)\|^2_{\vr_0}} + \frac{\sqrt{\lambda_i}}{2}\frac{\big(\frac{1}{\vr_0}\nabla_a(v_i\partial_a X_0), \frac{1}{\vr_0}\nabla_a(\varphi\partial_a X_0)\big)^2_{\vr_0}}{ \|\frac{1}{\vr_0}\nabla_a(v_i\partial_a X_0)\|^2_{\vr_0}}\\
		&=\frac{1}{2\sqrt{\lambda_i}}\frac{\big(\int_{S^1} \langle \nabla_a(v_i \partial_a X_0), V_0 \rangle\big)^2}{\int_{S^1} \frac{1}{\vr_0}\big|\nabla_a(v_i \partial_a X_0)\big|^2 } + \frac{\sqrt{\lambda_i}}{2}\frac{\big(\int_{S^1}\frac{1}{\vr_0} \langle \nabla_a(v_i \partial_a X_0), \nabla_a(\varphi\partial_a X_0) \rangle\big)^2}{\int_{S^1}\frac{1}{\vr_0} \big|\nabla_a(v_i\partial_a X_0)\big|^2 }\\
		&=\frac{1}{2\lambda_i^{3/2}}\bigg(\int_{S^1} \langle \nabla_a(v_i \partial_a X_0), V_0 \rangle\bigg)^2 + \frac{1}{2\lambda_i^{1/2}}\bigg(\int_{S^1} \frac{1}{\vr_0}\langle \nabla_a(v_i\partial_a X_0), \nabla_a(\varphi \partial_a X_0) \rangle\bigg)^2,
	\end{align}
	where in the last equality we used that 
	\be 
	\int_{S^1}\frac{1}{\vr_0}|\nabla_a(v_i \partial_a X_0)|^2 =  \int_{S^1} \L_{X_0} v_i v_i = \int_{S^1} \frac{W''(1,\cdot)\L_{X_0} v_i v_i}{W''(1,\cdot)} = \lambda_i.
	\ee 
	Continuing, 
	\begin{align}
		c_i &= \frac{1}{2\lambda_i^{3/2}}\bigg(\int_{S^1} \langle \nabla_a(v_i \partial_a X_0), V_0 \rangle\bigg)^2 + \frac{1}{2\lambda_i^{1/2}}\bigg(\int_{S^1} \frac{1}{\vr_0}\langle \nabla_a(v_i\partial_a X_0), \nabla_a(\varphi \partial_a X_0) \rangle\bigg)^2\\
		&= \frac{1}{2\lambda_i^{3/2}} \bigg(\int_{S^1} \langle\nabla_a V_0, \partial_a X_0\rangle v_i \bigg)^2 + \frac{1}{2\lambda_i^{1/2}}\bigg(\int_{S^1}  \L_{X_0} v_i  \varphi \bigg)^2\\
		&= \frac{1}{2\lambda_i^{3/2}} \bigg(\int_{S^1} \langle\nabla_a V_0, \partial_a X_0\rangle v_i \bigg)^2 + \frac{\lambda_i^{3/2}}{2}\bigg(\int_{S^1}  \frac{v_i  \varphi}{W''(1,\cdot)} \bigg)^2. 
	\end{align}
	This completes the derivation.
\end{proof}
 	We see the appearance of the fourth order term that is often encountered in beam theories, where bending is penalized from the onset. At first sight, this appears somewhat paradoxical in our framework, as we have explicitly taken hyperelastic model for the thread. However, one can regard our result as explaining the emergence of this effect as a result of persistent oscillations arising from initial conditions with significant internal energy.

\subsection{Stabilizing effect of Takens correction} Motivated by the nontrivial dynamics of the double spring pendulum in
Example~\ref{spring_example}, we investigate the second variation of the
homogenized potential \eqref{hompot-thread}. Our aim is to understand the
local effect of the Takens correction on perturbations of an inextensible
configuration. We note that while we could directly apply the general results of Appendix \ref{hessian-appendix}, the computation would be unnecessarily bulky; we thus specialize the proof and make use of various cancellations. To keep the formulas transparent, throughout the following calculations we assume
that the ambient manifold is flat and that the thread is homogeneous:
\be
\vr_0=1, \qquad W''(1,\cdot)=1, \qquad R=0.
\ee
The corresponding tension operator is
\be
\L_X=-\partial_a^2+|\partial_a^2 X|^2.
\ee
The general calculation is similar, but contains additional weights and
curvature terms.

\begin{prop}\label{thread_second_var}
	Suppose the manifold $M$ is flat and the thread is homogeneous, and every eigenvalue
	$\lambda_i$ for which $c_i\neq0$ is simple. Then Hessian of potential $V$ in the direction $\xi \in T_X \mathsf{SLoops}(M)$ is given by
	\begin{align}
	\hess_{X}^T V (\xi, \xi) &=\sum_i\bigg(\frac{c_i}{\sqrt{\lambda_i}} \bigg[ \sum_{j\neq i}\frac{4}{\lambda_i - \lambda_j}  \bigg(\int v_i v_j   \langle \partial_a^2 X, \partial_a^2 \xi\rangle\bigg)^2  + \int v_i^2 |\partial_a^2 \xi|^2\rmd a
\\
	&\qquad\qquad	-  \int \L^{-1}_{X} \langle \partial_a^3(v_i^2\partial_a^2 X), \partial_a X\rangle |\partial_a \xi|^2\rmd a -\frac{1}{\lambda_i} \bigg(\int v_i^2 \langle \partial_a^2 X, \partial_a^2 \xi\rangle\bigg)^2\rmd a\bigg]\bigg).
	\end{align} 
\end{prop}
\begin{proof}
	Recall that for a variation $X^\ve$ of $X$ with derivative $\xi$ at zero we have
	\begin{align}
	{\frac{\rmd}{\rmd \ve}\bigg|_{\ve=0}} \sqrt{\lambda_i [X^\ve]} &= \frac{1}{\sqrt{\lambda_i[X]}} \int  {v_i}^2 \langle \partial_a^2 X, \partial_a^2 \xi\rangle, \ \text{so} \ \grad^T\sqrt{\lambda_i[X^\ve]} = \frac{1}{\sqrt{\lambda_i[X]}}\PP_{T_{X}\mathsf{SLoops}(M)}\partial_a^2({v_i^\ve}^2\partial_a^2 X).
	\end{align}
	We compute
	\begin{align}
		\hess_X^T \sqrt{\lambda_i} (\xi, \xi) &= \bigg({\frac{\rmd}{\rmd \ve}\bigg|_{\ve=0}} \grad^T \sqrt{\lambda_i}[X^\ve], \xi\bigg)_{L^2} 
		={\frac{\rmd}{\rmd \ve}\bigg|_{\ve=0}}\frac{1}{\sqrt{\lambda_i[X]}} \bigg((\partial_a^2(v_i^2\partial_a^2 X)), \xi\bigg)_{L^2}\\
		&\qquad+\frac{1}{\sqrt{\lambda_i[X^\ve]}} \bigg(\bigg({\frac{\rmd}{\rmd \ve}\bigg|_{\ve=0}}\PP_{T_{X^\ve}\mathsf{SLoops}(M)}\bigg)(\partial_a^2(v_i^2\partial_a^2 X)), \xi\bigg)_{L^2}\\
		&\qquad\qquad+ 2 \frac{1}{\sqrt{\lambda_i[X]}}\bigg( \partial_a^2\bigg(v_i {\frac{\rmd}{\rmd \ve}\bigg|_{\ve=0}} v^\ve_i\partial_a^2 X^\ve\bigg), \xi\bigg)_{L^2}+ \frac{1}{\sqrt{\lambda_i[X]}}\bigg( \partial_a^2(v_i^2\partial_a^2 \xi), \xi\bigg)_{L^2}.
	\end{align}
	We compute each term separately. The first and fourth one are immediate:
	\begin{align}
		{\frac{\rmd}{\rmd \ve}\bigg|_{\ve=0}}\frac{1}{\sqrt{\lambda_i[X^\ve]}} \bigg((\partial_a^2(v_i^2\partial_a^2 X)), \xi\bigg)_{L^2} &= -\frac{1}{\lambda_i^{3/2}} \bigg(\int v_i^2 \langle \partial_a^2 X, \partial_a^2 \xi\rangle\bigg)^2 \\
		\frac{1}{\sqrt{\lambda_i[X]}}\bigg( \partial_a^2(v_i^2\partial_a^2 \xi), \xi\bigg)_{L^2} &= \frac{1}{\sqrt{\lambda_i}} \int v_i^2 |\partial_a^2 \xi|^2.
	\end{align}
To compute the second term we must vary the projector. We recall from the proof of Lemma \ref{threadtangentnormal}:
\be 
\PP_{T_{X^\ve}\mathsf{SLoops}(M)} = I - \PP_{N_{X^\ve}\mathsf{SLoops}(M)} = I + \partial_a (\L^{-1}_{X^\ve} \langle \partial_a \cdot, \partial_a X^\ve\rangle \partial_a X^\ve).
\ee 
Hence its derivative is computed as 
\begin{align} 
{\frac{\rmd}{\rmd \ve}\bigg|_{\ve=0}} \PP_{T_{X^\ve}\mathsf{SLoops}(M)} &= {\frac{\rmd}{\rmd \ve}\bigg|_{\ve=0}} \partial_a (\L^{-1}_{X^\ve} \langle \partial_a \cdot, \partial_a X^\ve\rangle \partial_a X^\ve)\\
&=  \partial_a\bigg({\frac{\rmd}{\rmd \ve}\bigg|_{\ve=0}}\bigg(\L^{-1}_{X^\ve} \langle \partial_a \cdot, \partial_a X^\ve\rangle\bigg) \partial_a X^\ve\bigg) + \partial_a (\L^{-1}_{X} \langle \partial_a \cdot, \partial_a X\rangle \partial_a \xi).
\end{align}
The range of the first term is normal to $\mathsf{SLoops}(M)$, so for any vector $\eta \in T_X \mathsf{Loops}(M)$ we have 
\be 
\PP_{T_{X}\mathsf{SLoops}(M)} \bigg({\frac{\rmd}{\rmd \ve}\bigg|_{\ve=0}} \PP_{T_{X^\ve}\mathsf{SLoops}(M)} \bigg)\eta = \PP_{T_{X}\mathsf{SLoops}(M)}\partial_a (\L^{-1}_{X} \langle \partial_a \eta, \partial_a X\rangle \partial_a \xi).
\ee 
Applying it to $\eta = \partial_a^2(v_i^2\partial_a^2 X)$, we compute that the second term is given by 
\begin{align}
\frac{1}{\sqrt{\lambda_i[X]}} \bigg(\bigg({\frac{\rmd}{\rmd \ve}\bigg|_{\ve=0}}\PP_{T_{X^\ve}\mathsf{SLoops}(M)}\bigg)&(\partial_a^2(v_i^2\partial_a^2 X)), \xi\bigg)_{L^2} = \frac{1}{\sqrt{\lambda_i[X]}} \bigg(\partial_a (\L^{-1}_{X} \langle \partial_a^3(v_i^2\partial_a^2 X), \partial_a X\rangle \partial_a \xi), \xi\bigg)_{L^2}\\
&= -\frac{1}{\sqrt{\lambda_i[X]}} \int \L^{-1}_{X} \langle \partial_a^3(v_i^2\partial_a^2 X), \partial_a X\rangle |\partial_a \xi|^2\rmd a
\end{align}

To compute the third term we invoke lemma \ref{HF2} by setting $A^\ve = \L^\ve$ and $H = L^2(S^1)$:
	\be 
	\frac{\rmd}{\rmd \ve }\bigg|_{\ve = 0} v_i^\ve = 	 \sum_{j\neq i} \frac{1}{\lambda_i - \lambda_j}\bigg(v_j,\frac{\rmd}{\rmd \ve}\bigg|_{\ve = 0}\L^\ve v_i\bigg)_{L^2} v_j = \sum_{j \neq i}\frac{2}{\lambda_i - \lambda_j} \bigg(\int \langle\partial_a^2 X, \partial_a^2 \xi \rangle v_i v_j\rmd a\bigg) v_j.
	\ee 
	With this variation in hand, the third term in the expression for $\hess_X \sqrt{\lambda_i}(\xi,\xi)$ is given by
	\be
	\frac{1}{\sqrt{\lambda_i}}\int 2 v_i{\frac{\rmd}{\rmd \ve}\bigg|_{\ve=0}} v_i^\ve \langle \partial_a^2 X, \partial_a^2 \xi\rangle\rmd a =  \frac{1}{\sqrt{\lambda_i}}\sum_{j\neq i}\frac{4}{\lambda_i - \lambda_j}  \bigg(\int v_i v_j   \langle \partial_a^2 X, \partial_a^2 \xi\rangle\rmd a\bigg)^2.
	\ee
	Combining, the claimed expression of $\hess^T_X V[X]$ follows.
\end{proof}
Let us examine the result of Proposition \ref{thread_second_var}. The local term
$
\int_{S^1}v_i^2|\partial_a^2\xi|^2\,\rmd a
$
is positive and gives the principal high-frequency contribution to the
quadratic form. This suggests a stabilizing
effect on short-wavelength perturbations, although the sign of the low
modes must be examined separately. We do this explicitly for the circular
loop below.
\begin{example}[Expansion of a  circular loop]
Consider a circular stationary solution
\begin{align}
X(a) &= (\cos a, \sin a), \qquad \dot{X}(a)=0.
\end{align}
The eigenpairs of the tension operator at this state are
\be
v_{i}(a) = \frac{1}{\sqrt{\pi}} \cos(ia), \quad v_0 = \frac{1}{\sqrt{2\pi}}, \quad v_{-i}(a) = \frac{1}{\sqrt{\pi}} \sin(ia) \quad \text{for} \ i > 0\\
\quad  \text{and} \quad \lambda_i = |i|^2 + 1.
\ee
Suppose only the zeroth mode being excited, i.e. $c_0 = C$ while  $c_i = 0 \ \text{for} \ i \neq 0$.  This corresponds to an initial uniform expansion of the stiff rubber band.
Although the nonzero eigenvalues have multiplicity two, the excited
ground-state eigenvalue $\lambda_0=1$ is simple. Therefore
Proposition \ref{thread_second_var} applies to this example (the computation of Hessian is independent of the convergence of the stiff thread to the corresponding Takens limit theorem, which we do not claim).
\end{example}

\begin{prop}\label{circleprop}
Consider mean-zero perturbations
	$\xi\in T_X\mathsf{SLoops}(\R^2)$ of a circular loop:
	\be
	f(a)=a_0+\sum_{i \geq 1} \big( a_i \cos(ia) + b_i \sin(ia) \big), \qquad \partial_a\xi=f X.
	\ee
	The Hessian in the direction of such a perturbation is
	\be
	\hess_X V(\xi,\xi) =
	\frac{C}{2}\sum_{i \geq 2} (i^2 - 4)(a_i^2 +b_i ^2) \geq 0.
	\ee
	In particular, the Hessian is strictly positive on perturbations
	satisfying $a_0=a_2=b_2=0$.
\end{prop}
\begin{proof}
	By using that $\xi \in T_X\mathsf{Sloops}(\R^2)$ as well as that $X$ is a circle, we set
	
	\[
	\partial_a \xi = f X   \Longrightarrow
	\left\{\begin{aligned}
		|\partial_a \xi|^2 &= f^2\\
		\langle\partial_a^2 \xi,\partial_a^2 X\rangle &= -f'\\
		|\partial^2_a \xi|^2 &= f^2 + f'^2
	\end{aligned}\right. .
	\]
	Note that since $\int \partial_a \xi \rmd a =0$ we have 
	$
	\int f(a) \cos a \rmd a = \int f(a) \sin a \rmd a = 0,
	$
	so the first Fourier modes of $f$ vanish, so
$
	f(a) = a_0 + \sum_{i\geq 2} \big(a_i \cos(i a) + b_i \sin(i a)\big).
$
	The Hessian of $V$ is given by 
	\begin{align}
	\hess^T_{X} V (\xi, \xi) &=\frac{C}{\sqrt{\lambda_0}} \bigg[ \sum_{j\neq 0}\frac{4}{\lambda_0 - \lambda_j}  \bigg(\int v_0 v_j   \langle \partial_a^2 X, \partial_a^2 \xi\rangle\bigg)^2 \\
	&\quad\qquad- \int \L^{-1}_{X} \langle \partial_a^3(v_0^2\partial_a^2 X), \partial_a X\rangle |\partial_a \xi|^2
	+ \int v_0^2 |\partial_a^2 \xi|^2 -\frac{1}{\lambda_0} \bigg(\int v_0^2 \langle \partial_a^2 X, \partial_a^2 \xi\rangle\bigg)^2\bigg].
\end{align} 
	The first term is computed using Fourier expansion of $f$:
	\begin{align}
		\sum_{i \neq 0}\frac{4}{\lambda_0 - \lambda_i}  \bigg(\int v_0 v_i   \langle \partial_a^2 X, \partial_a^2 \xi\rangle\rmd a\bigg)^2 &= -\sum_{i>0}\frac{4}{i^2}  \bigg(\bigg(\frac{1}{\sqrt{2}\pi}\int \cos(ia)    f'\bigg)^2 +  \bigg(\frac{1}{\sqrt{2}\pi}\int \sin(ia)    f'\bigg)^2\bigg) \\
		&=-\sum_{i>1}\frac{4}{i^2}\bigg( \bigg(\frac{b_i i}{\sqrt{2}}\bigg)^2 + \bigg(\frac{a_i i}{\sqrt{2}}\bigg)^2\bigg) = - 2\sum_{i>1} (a_i^2 + b_i^2).
	\end{align}
	The second term, using that $v_0 = \tfrac{1}{\sqrt{2\pi}}$, $\partial_a^3 (\tfrac{1}{2\pi} \partial_a^2 X) = \tfrac{1}{2\pi} \partial_a X$, and $\L^{-1}_{X}\tfrac{1}{2\pi} = \tfrac{1}{2\pi}$ is given by 
	\begin{align}
	- \int \L^{-1}_{X} \langle \partial_a^3(v_0^2\partial_a^2 X), \partial_a X\rangle |\partial_a \xi|^2 &= - \int \frac{1}{2\pi} f^2 = - \bigg(a_0^2 +\frac{1}{2}\sum_{i>1} (a_i^2 + b_i^2)\bigg).
	\end{align}
	The third term is 
	\be 
	\int v_0^2 |\partial_a^2 \xi|^2 \rmd a = \frac{1}{2\pi}\int( f^2 + f'^2)\rmd a = a_0^2 + \frac{1}{2}\sum_{i>1} (1 + i^2)(a_i^2 + b_i^2),
	\ee 
	The last term vanishes:
	\be 
	\frac{1}{\lambda_0} \bigg(\int v_0^2 \langle \partial_a^2 X, \partial_a^2 \xi\rangle\rmd a\bigg)^2 = \bigg(\frac{1}{2\pi}\int f'\rmd a\bigg)^2 = 0.
	\ee 
	Plugging the computed terms into the expression for $\hess_X V$,
	\begin{align}
		\hess_{X}^T V (\xi, \xi) &=C \bigg[- 2\sum_{i>1} (a_i^2 + b_i^2) - \bigg( a_0^2 +\frac{1}{2}\sum_{i>1} (a_i^2 + b_i^2)\bigg)+ \bigg( a_0^2 + \frac{1}{2}\sum_{i>1} (1 + i^2)(a_i^2 + b_i^2)\bigg)\bigg]\\
		&= \frac{C}{2}\sum_{i>1} (i^2 - 4)(a_i^2+b_i^2).
	\end{align}
	This completes the proof.
\end{proof}
	Proposition \ref{circleprop} indicates that  the circular loop is quadratically stable with respect to Fourier
	modes of order $i\geq3$. The infinitesimal rotational mode $a_0$ and the
	second Fourier modes remain neutral at quadratic order. Of course, for the inextensible thread without the Takens force, the circular configuration is completely unstable. See Figure \ref{threadfig}.   It would be interesting if this formal stability mechanism inherited from its origin as a stiff rubber band could be observed.

		\begin{figure}[h!]
	 \centering 
		\includegraphics[width=.13\columnwidth]{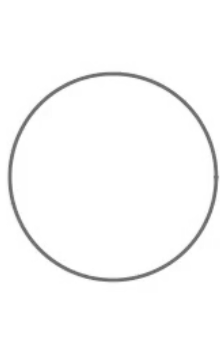}\hspace{5mm}
		\includegraphics[width=.13\columnwidth]{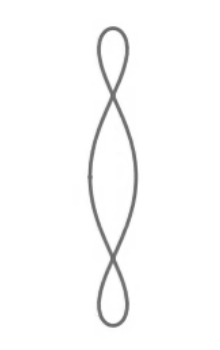}\hspace{5mm}
		\includegraphics[width=.13\columnwidth]{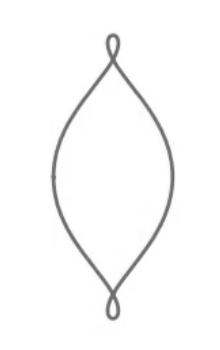}\hspace{7mm}
		\includegraphics[width=.13\columnwidth]{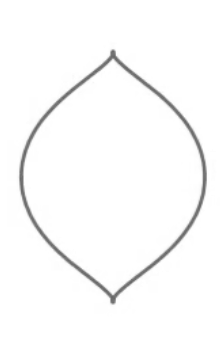}
		\caption{Numerical simulation of inextensible thread without the Takens correction.}\label{threadfig}
	\end{figure}

\section{Compressible Euler and incompressible limit}
Our next example is the prototype of all continuum media, fluid dynamics.
Recall the compressible Euler equations
	\be
	\begin{cases}
		\partial_t u + \nabla_u u = - \frac{1}{\ve^2} \frac{\nabla P}{\vr}\\
		\partial_t \vr + \div(\vr u) = 0\\
		\partial_t s +  \nabla_u s = 0\\
		P = \widetilde{P}(\vr,s)
	\end{cases}
	\ee 
	where $\ve$ is the Mach number. It is known (see e.g. \cite{KMM21, Sm79},  and also Section \ref{compeuler-sect} of the present paper), that it can be written as a Newton equation on $\mathsf{Diff}(M)$ with potential of the form $\frac{1}{\ve^2}U[X]$, with $U$ given by \eqref{comppot}. Thus it makes sense to compute the corresponding Takens limit system.
	
	\begin{theorem}\label{euler-theorem}
		The homogenized potential of $U$(\eqref{comppot}) is given by 
		\be 
		V[X] = \sum c_i \sqrt{\lambda_i[X]},
		\ee 
		where $\lambda_i[X]$ are non-zero eigenvalues of the acoustic wave operator 
		\be\label{awo}
		\L_X = -Q \div \bigg(\frac{1}{\vr} \nabla \bigg),
		\ee
		with $Q = \vr \partial_\vr \tilde{P}(\vr,s)$, and suppose $\lambda_i$ are simple. If $X: [0,T) \to \mathsf{Diff}(M)$ solves Takens limit system of compressible Euler, its Eulerian velocity $u = \dot{X} \circ X^{-1}$, mass density $\vr =  \vr_0/ \det(\nabla X) \circ X^{-1}$, and entropy $s =  s_0 \circ X^{-1}$ satisfy modified incompressible Euler equations
			\be\label{takens-euler}
	\begin{cases}
		\partial_t (\vr u) + \div(\vr u \otimes u) + \div \Sigma  + \nabla p= 0 ,\\
		\partial_t \vr + u\cdot\nabla \vr = 0,\\
		\partial_t s + u\cdot\nabla s = 0,\\
		\div u = 0.
	\end{cases}
	\ee 
	Here, $p$ is determined to enforce incompressibility, and $\Sigma$ is a symmetric 2-tensor (``acoustical stress"), given by
	\be\label{acousticalstress}
	\Sigma = \sum_i \frac{c_i}{\sqrt{\lambda_i}} \frac{\nabla v_i \otimes \nabla v_i}{\vr},
	\ee 
		where $v_i$ are $L^2_{1/Q}$ normalized eigenfunctions of $\L_X$ and $c_i$ are adiabatic invariants computed from the data for compressible Euler \eqref{cidef} as in \eqref{actions-fluid}.
	\end{theorem}
Note that since the initial state $(\vr_0, s_0)$ satisfies $\tilde{P}(\vr_0, s_0) = P = const$ and $P$ is transported, transport of entropy can be replaced with a diagnostic equation $\tilde{P}(\vr, s) = P$ that uniquely determines it.
	\begin{remark}[Barotropic fluids]\label{barotropic-rmk}
		In the case of barotropic flow incompressible limit singles out homogeneous mass density distribution (indeed, if $\vr_0$ is not uniform, pressure forces are going to blow up). This in turn results in the divergence of acoustic stress being a pure gradient, and thus it can be absorbed into pressure. This immediately follows from the expression of the Hessian in lemma \ref{hess-fluid}, as it only depends on $\vr$ and $Q = Q(\vr)$. Thus Takens limit of barotropic flow is just given by incompressible Euler equations, which explains the reason incompressible limit holds for the mildly ill-prepared data \cite{Abaro}.  
		
		In the presence of free boundary there might be an effect even in barotropic flow, although we do not pursue it here. A similar situation arises for odd viscosity in incompressible fluids \cite{ACG,GA}.
	\end{remark}
	The explicit expression of adiabatic invariants is given by the following proposition.
\begin{prop} 
	The actions $c_i$ are explicitly computed from the family of initial data by 
	\be \label{actions-fluid}
	c_i = \frac{1}{2\lambda_i^{3/2}} \bigg(\int_M v_i \div V_0 \bigg)^2 + \frac{\lambda_i^{3/2}}{2}\bigg(\int_M  \frac{v_i  \varphi}{Q} \bigg)^2.
	\ee 
\end{prop}
This equation was derived in \cite{MS03} from a different perspective. The result of \cite{MS03} can be stated with our conventions as follows:
\begin{theorem}[\cite{MS03}, Theorem 8.4]
Consider the family of mildly ill-prepared initial data $(u_0^\ve, \vr_0^\ve, s_0^\ve)$ in $H^s(\mathbb{T}^d)$, $s > 1 + d/2$, converging strongly in $H^s(\mathbb{T}^d)$ to $(u_0, \vr_0, s_0)$ as $\ve \to 0$. 
Let $T>0$, and  let $(u^\ve, \vr^\ve, s^\ve)$ be solutions on $[0,T]$ of \eqref{compeulereq} with this initial data. Suppose that the $(\vr^\ve, s^\ve)$ and the incompressible component of the velocity $u_{in}^\ve$ converge strongly to $(\vr, s, u)$ and that the limiting acoustic operator $\L_{\vr}$ has simple positive spectrum and is non-resonant at orders up to three almost everywhere in time.
Then $u, \vr, s$ satisfy \eqref{takens-euler} with action constants given by \eqref{actions-fluid} and initial data 
\be 
s_{t = 0} = s_0, \ \vr_{t = 0} = \vr_0, \ u_{t = 0} = u_0 + \frac{1}{\vr_0}\nabla \L_{\vr_0}^{-1}(\tilde{Q}(\vr_0, s_0)\div u_0).
\ee 
\end{theorem} 
This theorem justifies the study of Takens system for the setting of incompressible limit. The nonresonance conditions seem to be unavoidable; it would be interesting to understand to what extent they are generically satisfied. A first step in this direction can be found in \cite{BDG}, where the authors perform a study of the resonances and provide a formal argument that for almost all initial data they do not occur. We also mention the work \cite{BDGL} where analogous study is performed for the viscous fluid.

\subsection{Compressible Euler as a Newton equation}\label{compeuler-sect}
	Let $M$ be compact Riemannian manifold. Let $\mathsf{Diff}(M)$ denote the group of smooth diffeomorphisms of $M$ that leave the boundary invariant. It can be given a Frechet manifold structure with the tangent space at $X$ being
\be 
T_{X}\mathsf{Diff}(M) = \{\zeta \circ X\ | \zeta \text{ is a smooth vector field on $M$, tangent to the boundary}\}.
\ee

Let $\vr_0 \in C^{\infty}(M)$ be a positive density. We denote by $\mathsf{Diff}_{\vr_0}(M)$ the space $\mathsf{Diff}(M)$ endowed with the (weak) Riemannian metric given by right-translating the $\vr_0 $-weighted $L^2$ inner product on the space of vector fields from $T_{\mathsf{id}}\mathsf{Diff}(M)$:
\be \label{rho_inner}
\big(\zeta \circ X, \eta \circ X\big)_{X} = \int_M \langle \zeta(X(a)), \eta(X(a))\rangle_{X(a)} \vr_0(a)\rmd a.
\ee

	We now describe the compressible Euler equations as Newton equations on $\mathsf{Diff}_{\vr_0}(M)$. Let $\tilde{e}(\vr, s)$ be a given smooth, positive function.
	Given two functions $\vr_0$ and $s_0$, define the potential function 
	\be\label{comppot}
	U[X]:= \int_M \tilde{e}\left(\vr_X(a), s_X(a)\right) \vr_X(a)\rmd a,
	\ee
	where $\vr_X$ and $s_X$ are functions on $M$ defined by 
	\be 
	\vr_X:= \frac{\vr_0}{\det \nabla X} \circ X^{-1}, \ \quad \ s_X := s_0\circ X^{-1}.
	\ee 
	In what follows we will often abuse the notation and omit the subscript $X$.
	
	The compressible Euler equations are described as an ODE by Newton's equations on $\mathsf{Diff}_{\vr_0}(M)$:
	\be\label{compeulergrad}
	\ddot{X} = - \grad_{\vr_0} U[X].
	\ee
	Here and throughout this section, acceleration is again the covariant one:
	\be 
	\ddot{X} := \nabla_{\dot{X}} \dot{X}.
	\ee 
	Since the metric of $\mathsf{Diff}_{\vr_0}(M)$ is the pointwise weighted $L^2$ metric, its Levi-Civita connection corresponds to the pointwise Levi-Civita connection of the metric on $M$, so the computations performed thereafter are unambiguous (see \cite{EM, BLR, B18}).
	
	We shall compute the gradient explicitly. For it, the following elementary lemma is useful.
	\begin{lemma}\label{rhosvar}
		Consider the variation $X^\ve$ of $X$ through $\mathsf{Diff}(M)$ satisfying 
		\be
		\frac{\rmd }{\rmd \ve }\bigg|_{\ve = 0} X^\ve = \xi \circ X \in T_X\mathsf{Diff}(M).
		\ee 
		The variations of $\vr_X$ and $s_X$ are given respectively by
		\be 
		\frac{\rmd }{\rmd \ve }\bigg|_{\ve = 0} \vr_{X^\ve} = -\div(\vr_X\xi), \ \quad \ \frac{\rmd }{\rmd \ve }\bigg|_{\ve = 0} s_{X^\ve} = - \langle \nabla s_X, \xi \rangle. 
		\ee 
	\end{lemma}
\begin{proof}
	We have 
	\be 
	\vr_X \rmd a  = X_*(\vr_0 \rmd a), \quad s_X = X_*s_0,
	\ee 
	so the variations are given by
	\begin{align}
		&\frac{\rmd }{\rmd \ve }\bigg|_{\ve = 0} \vr_{X^\ve} \rmd a = -\L_{\xi} (\vr_X \rmd a) = -\div(\vr_X \xi)\rmd a \text{, and}\\
		&\frac{\rmd }{\rmd \ve }\bigg|_{\ve = 0} s_{X^\ve} = -\L_{\xi} s = -ds_X(\xi) = - \langle \nabla s_X, \xi \rangle.
	\end{align}
\end{proof}
With that, we compute the gradient of $U$.
\begin{lemma} Define the hydrodynamic pressure law by
	\be
	\widetilde{P}(\vr, s):= \vr^2 \partial_\vr \tilde{e} (\vr,s)
	\ee 
	Then
	\be
	\grad_{\vr_0} U[X]= \frac{\nabla P_X}{\vr} \circ X, \ \text{where} \ 
	P_X(a) = \widetilde{P}(\vr_X(a),s_X(a)),
	\ee
\end{lemma}
\begin{proof}
	Consider again the variation $X^\ve$ of $X$ satisfying 
	\be
	\frac{\rmd }{\rmd \ve }\bigg|_{\ve = 0} X^\ve = \xi \circ X \in T_X\mathsf{Diff}(M).
	\ee 
	Denote $e = e_X$ by
	\be 
	e(a) = \tilde{e}(\vr(a),s(a)).
	\ee 
	We compute the variation of $U$ as follows:
	\begin{align}
	\frac{\rmd }{\rmd \ve }\bigg|_{\ve = 0} U[X^\ve] &= \frac{\rmd }{\rmd \ve }\bigg|_{\ve = 0}\int_M \tilde{e}\left(\vr_{X^\ve}(a), s_{X^\ve}(a)\right) \vr_{X^\ve}(a)\rmd a\\
	&= - \int_M \partial_\vr \tilde{e} \div(\vr \xi) \vr + \partial_s \tilde{e} \langle \nabla s,\xi\rangle \vr + e \div(\vr \xi)\\
	&= - \int_M \partial_\vr \tilde{e} \langle \nabla \vr,\xi\rangle \vr +\partial_\vr \tilde{e} \div \xi \vr^2+ \partial_s \tilde{e} \langle \nabla s,\xi\rangle \vr + \tilde{e} \div(\vr \xi)\\
	&= - \int_M \langle \partial_\vr \tilde{e}\nabla \vr + \partial_s \tilde{e} \nabla s,\xi\rangle \vr +\partial_\vr \tilde{e} \div \xi \vr^2 + e \div(\vr \xi)\\
	&= - \int_M \langle \nabla e,\xi\rangle \vr + e \div (\vr\xi) - \int_M \partial_\vr \tilde{e} \div \xi \vr^2\\ 
	&= \int_M \bigg\langle \frac{\nabla P}{\vr},\xi \bigg\rangle \vr = \bigg(\frac{\nabla P}{\vr} \circ X,\xi \circ  X\bigg)_{\vr_0}.
	\end{align}
By definition of the gradient claim follows.
\end{proof}
With that we can write more familiar Eulerian form of the equation:
\begin{prop}\label{compeuler}
	Let $X \in \mathsf{Diff}_{\vr_0}(M)$ solve \eqref{compeulergrad}. Then 
	\be 
	u(a,t) = \dot{X}(X^{-1}(a,t),t), \quad s(a,t) = s_0(X^{-1}(a,t)),\quad \vr(a,t) = \frac{\vr_0(X^{-1}(a,t))}{\det\nabla X(X^{-1}(a,t))}
	\ee 
	satisfy the Eulerian form of compressible Euler equations:
	\be\label{compeulereq}
	\begin{cases}
		\partial_t u + \nabla_u u = - \frac{\nabla P}{\vr},\\
		\partial_t \vr + \div(\vr u) = 0,\\
		\partial_t s + \langle u,\nabla s \rangle = 0,\\
		P = \widetilde{P}(\vr,s).
	\end{cases}
	\ee 
\end{prop}
\begin{proof}
	It follows from the chain rule that 
	\begin{align}
		(\partial_t u)(X(a,t),t) + (\nabla_{u(X(a,t),t)}u)(X(a,t),t) &= \frac{D}{dt} (u (X(a,t),t)) \\
		&=  \frac{D}{dt}(\dot{X}(a,t)) = \ddot{X}(a,t) = -\frac{\nabla P(X(a,t),t)}{\vr(X(a,t),t)},
	\end{align}
	so evaluating at $(X^{-1}(a,t),t)$ gives first momentum equation.
	The mass and entropy conservation equations  follow similarly from replacing $(\ve, \xi)$ with $(t,u)$ in Lemma \ref{rhosvar}.
\end{proof}
\subsection{Incompressible Euler}

Consider submanifold $\mathsf{SDiff}(M)$ of $\mathsf{Diff}(M)$ consisting of volume-preserving diffeomorphisms of $M$:
\be 
\mathsf{SDiff}(M) = \{X \in \mathsf{Diff}(M) \ | \ X_*(\rmd a) = \rmd a\}.
\ee  
Denote by $\mathsf{SDiff}_{\vr_0}(M)$ $\mathsf{SDiff}(M)$ endowed with the $\vr_0$-weighted metric.
Incompressible Euler equation is given by the geodesic equation on $\mathsf{SDiff}_{\vr_0}(M)$.
\begin{prop}\label{tangentnormal}
	The following characterizations hold:
	\begin{itemize} 
		\item The tangent space of $\mathsf{SDiff}_{\vr_0}(M)$ at $X$ is given by 
		\be
		T_X \mathsf{SDiff}_{\vr_0}(M) = \bigg\{\xi\circ X \ \bigg| \ \div \xi = 0 \bigg\}.
		\ee 
		\item The normal space of $\mathsf{SDiff}_{\vr_0}(M)$ at $X$ is given by 
		\be
		N_X \mathsf{SDiff}_{\vr_0}(M) = \bigg\{\frac{\nabla p}{\vr}\circ X \ \bigg| \ p \in C^{\infty}(M, \R)\bigg\}.
		\ee 
	\end{itemize}
\end{prop}
\begin{proof}
	Let $X^\ve$ be a curve in $\mathsf{SDiff}_{\vr_0}(M)$ through $X  = X^0$  with $\frac{\rmd}{\rmd \ve }\bigg|_{\ve = 0} X^\ve = \xi\circ X$. The condition $X^\ve \in \mathsf{SDiff}_{\vr_0}(M)$ implies that 
	\be
	0 = -\frac{\rmd}{\rmd \ve }\bigg|_{\ve = 0} X^\ve_* \rmd a  =\div \xi \rmd a,
	\ee 
	Thus $\xi$ has to be divergence-free. The same computation reversed shows that any divergence-free vector field provides an element of $T_X \mathsf{SDiff}_{\vr_0}(M)$.
	
	Let $\eta \circ X \in N_X \mathsf{SDiff}_{\vr_0}(M)$. For any tangent vector $\xi \circ X$ we must have 
	\be 
	0 = (\eta \circ X,  \xi \circ X)_{\vr_0} = \int_M \langle \eta\circ X , \xi \circ X\rangle \vr_0 =  \int_M \langle \vr \eta , \xi\rangle. 
	\ee 
	Thus $\vr\eta$ has to be $L^2$ orthogonal to all divergence-free vector fields, so by Helmholtz decomposition it has to be a gradient.
\end{proof}
With this description of the normal spaces we can write the geodesic equation on $\mathsf{SDiff}_{\vr_0}(M)$ by d'Alembert's principle
\be \label{incompeulergrad}
\left\{\begin{aligned} 
\ddot{X} &\in N_X\mathsf{SDiff}_{\vr_0}(M) \\
X &\in \mathsf{SDiff}_{\vr_0}(M)
\end{aligned}\right. \iff 
\left\{\begin{aligned}
\ddot{X} &= -\frac{\nabla p}{\vr}\circ X\\
\det\nabla X &\equiv 1
\end{aligned}\right.
\ee
We can now write the more familiar Eulerian form of the equations.
\begin{prop}
	Let $X\in \mathsf{SDiff}_{\vr_0}(M), P \in C^{\infty}(M, \R)$ solve \eqref{incompeulergrad}. Then 
	\be 
	u(a,t) = \dot{X}(X^{-1}(a,t),t), \quad \vr(a,t) = \vr_0(X^{-1}(a,t))
	\ee 
	satisfy the Eulerian form of incompressible Euler equations:
	\be
	\begin{cases}
		\partial_t u +\nabla_u u = - \frac{\nabla p}{\vr}\\
		\partial_t \vr + \div(\vr u) = 0\\
		\div u = 0
	\end{cases}
	\ee 
\end{prop}
\begin{proof}
	The equations satisfied by $u$ and $\vr$ follow from the chain rule in the same way as in the proof of \ref{compeuler}. The last equation is the Eulerian form of the condition $u \circ X \in T_X \mathsf{SDiff}_{\vr_0}(M)$.
\end{proof}

\subsection{Incompressible limit (Takens system): homogenized potential}
We now consider making the compressible fluid more and more incompressible by stiffening the potential $U$, i.e. replacing it by $\frac{1}{\ve^2} U$ and taking the limit as $\ve\to 0$. In order to take the limit, we must ensure that the force $-\frac{1}{\ve^2}\frac{\nabla P}{\vr}$ experienced by the fluid at time $t = 0$ does not blow up. We do so by considering the initial data satisfying $\tilde{P}(\vr_0, s_0) = {\rm const}$. For concreteness we take $X|_{t=0} = {\rm id}$.

Finite dimensional results of the previous sections suggests that the (non-resonant) limit is given by the corresponding Takens system, which we proceed to derive. Recall that it is given by the geodesic equation on the critical set $\mathcal{M}$ of $U$ containing $X|_{t = 0} = id$, subject to the additional forcing of the form $\grad_{\vr_0}V[X]$, where 
\be 
V[X] = \sum_i c_i \sqrt{\lambda_i[X]},
\ee 
and $\lambda_i[X]$ are eigenvalues of $\hess_{\vr_0} U[X]$ restricted to $N_X\mathcal{M}$. Thus we need to understand $\mathcal{M}$, as well as the Hessian of $U$; this is done in the following lemma.
\begin{lemma}
	Suppose the bulk modulus 
	\be 
	\widetilde{Q}(\vr,s) := \vr \partial_\vr \widetilde{P}(\vr, s)
	\ee 
	is positive. Then the path-connected component of identity of $\mathcal{M}$ is the path-connected component of identity of $\mathsf{SDiff}_{\vr_0}(M)$.
\end{lemma}
\begin{proof} 
	We will abuse notation and denote by $\M$ (resp. $\mathsf{SDiff}_{\vr_0}(M)$) the path-connected component of identity of $\M$ (resp. $\mathsf{SDiff}_{\vr_0}(M)$). Recall that $\M$ consists of diffeomorphisms $ X \in \mathsf{Diff}_{\vr_0}(M)$ for which  
	\be 
	\nabla P \equiv 0 \ \Leftrightarrow \ P \text{ is spatially constant}
	\ee 
	We first claim that under our assumption this constant is the same for all elements of $\M$. To see this, let $X^\ve$ be a path in $\M$, with $\dot{X}^\ve = \xi_\ve \circ X^\ve$. Denoting $\vr_\ve = \vr_{X^\ve},$ $s_\ve = s_{X^\ve}$, we differentiate in $\ve$ the relation 
	\be 
	C(\ve) = \widetilde{P}(\vr_\ve, s_\ve).
	\ee 
	Using Lemma \ref{rhosvar}, we obtain
	\begin{align}
	\dot{C}(\ve) &= -\div(\vr_\ve\xi_\ve)\partial_\vr \widetilde{P}(\vr_\ve, s_\ve) -\langle\nabla s_\ve, \xi_\ve\rangle\partial_s \widetilde{P}(\vr_\ve, s_\ve)\\
		&=-\div\xi_\ve \vr_\ve \partial_\vr \widetilde{P}(\vr_\ve, s_\ve) -\langle \nabla \vr_\ve, \xi_\ve\rangle \partial_\vr \widetilde{P}(\vr_\ve, s_\ve) -\langle \nabla s_\ve, \xi_\ve \rangle \partial_s \widetilde{P}(\vr_\ve, s_\ve)\\
		&=-\div\xi_\ve \widetilde{Q}(\vr_\ve,s_\ve) - \langle \nabla P_\ve, \xi \rangle\\
		&= -\div\xi_\ve \widetilde{Q}(\vr_\ve,s_\ve).
	\end{align}
	Thus 
	\be 
	0 = \int_M \div\xi_\ve\rmd a = -\dot{C}(\ve) \int_M \frac{\rmd a}{\widetilde{Q}(\vr_\ve,s_\ve)}.
	\ee 
	Since we assumed that $\widetilde{Q}$ is positive, this forces $\dot{C}(\ve) = 0$ and consequently 
	\be 
	C(\ve) = C(0) =: C_0.
	\ee 
	With that, we proceed to proving the main claim.
	Defining
	\be 
	f(a,\mu) = \widetilde{P}(\vr_0(a) \mu, s_0(a)),
	\ee  
	the condition $X \in \mathcal{M}$ is rewritten as 
	\be 
	f\bigg(a, \frac{1}{\det \nabla X(a)}\bigg) = C_0 \ \text{for all } a.
	\ee 
	
	We regard it as the family of equations prescribing $\det \nabla X(a)$ for each $a$. Since we took initial configuration $(\vr_0,s_0)$ in such a way that $ {\rm id} \in \mathcal{M}$, $f(a, 1) = C_0$, so there is one solution $\det \nabla X(a) = 1$. To see that there are no more, we compute 
	\be 
	\partial_\mu f(a,\mu) = \vr_0(a) \partial_\vr \widetilde{P}(\vr_0(a) \mu, s_0(a)) = \frac{\widetilde{Q}(\vr_0(a) \mu, s_0(a))}{\mu} > 0,
	\ee 
	so $f(a,\cdot)$ is strictly increasing. Thus the only diffeomorphisms $X$ in $\M$ are those satisfying $\frac{1}{\det \nabla X(a)} \equiv 1$, that is the ones that preserve volume.  
\end{proof}
We proceed to investigate the Hessian of $U$ restricted to $\mathsf{SDiff}_{\vr_0}(M)$.
\begin{lemma}\label{hess-fluid}
	For the point $X \in \M$, the ambient Hessian of $U$ is given by 
	\be
	\hess_{X} U(\xi \circ X, \cdot) = - \frac{\nabla(Q\div\xi )}{\vr}\circ X,
	\ee
	where again
	\be 
	\widetilde{Q}(\vr, s) = \vr \partial_\vr \widetilde{P} (\vr, s), \ \quad \
	Q(a) = \widetilde{Q}(\vr(a),s(a)).
	\ee
	In particular, if $\widetilde{Q}>0$ it is positive definite on $N_X \mathsf{SDiff}_{\vr_0}(M)$.
\end{lemma}
\begin{proof}
	With $X \in \M$, the action of $\hess_{\vr_0} U$ on $\xi \circ X \in T_X \mathsf{Diff}_{\vr_0}(M)$ is given by 
	\begin{align}
		\hess_{X} U (\xi \circ X, \cdot) &= \frac{\rmd}{\rmd \ve }\bigg|_{\ve = 0}\grad_{\vr_0} U[X^\ve]  = \frac{\rmd}{\rmd \ve }\bigg|_{\ve = 0}\bigg( \frac{\nabla P}{\vr}\bigg)\circ X\\
		&=\cancel{\xi\circ X\cdot \nabla\left(\frac{\nabla P}{\vr}\right)\circ X} +\frac{\nabla \frac{\rmd}{\rmd \ve }\bigg|_{\ve = 0} \widetilde{P}(\vr^\ve,s^\ve)}{\vr} \circ X- \cancel{\bigg(\frac{\rmd}{\rmd \ve }\bigg|_{\ve = 0}\vr^\ve\bigg)  \frac{\nabla P}{\vr^2}} \circ X\\
		&= -\frac{\nabla \Big(\div( \xi \vr) \partial_\vr \widetilde{P}(\vr,s) +\langle\nabla s,\xi \rangle\partial_s \widetilde{P}(\vr,s)\Big)}{\vr}\circ X\\
		&= -\frac{\nabla \Big(\cancel{\langle \nabla P, \xi\rangle} + Q \div \xi \Big)}{\vr}\circ X = - \frac{\nabla (Q \div \xi)}{\vr} \circ X.
	\end{align}
	From Proposition \ref{tangentnormal}, normal spaces of $\mathsf{SDiff}_{\vr_0}(M)$ are given by the $\nabla\varphi/\vr \circ X$. Plugging it in,
	\be 
	\hess_{X} U \bigg(\frac{\nabla\varphi}{\vr} \circ X, \frac{\nabla\varphi}{\vr}\circ X\bigg) = -\int_M \bigg\langle \frac{\nabla (Q \div\frac{\nabla\varphi}{\vr})}{\vr}, \frac{\nabla\varphi}{\vr}\bigg\rangle\vr = \int_M Q \bigg(\div \frac{\nabla\varphi}{\vr} \bigg)^2 > 0.
	\ee 
\end{proof}
With that we conclude that potential $U$ is constraining to $\mathsf{SDiff}_{\vr}(M)$. We proceed to investigate the behavior of oscillations around $\mathsf{SDiff}_{\vr}(M)$; from the Theorem \ref{maintheorem}, this requires analysis of the spectral decomposition of $\hess_{X} U \big|_{N_X\mathsf{SDiff}_{\vr_0}}(M)$.  From Proposition \ref{tangentnormal}, the eigenproblem for $\hess_{X} U \big|_{N_X\mathsf{SDiff}_{\vr_0}}(M)$ is written as  
\begin{align}
	-\frac{1}{\vr}\nabla \bigg(Q \div \bigg(\frac{1}{\vr}\nabla v\bigg)\bigg)\circ X &=  \lambda \frac{1}{\vr}\nabla v \circ X\qquad \iff \qquad  \L_X v = \lambda v + const,
\end{align}
where the acoustic wave operator \eqref{awo} is 
$\L_X  = - Q \div \bigg(\frac{1}{\vr} \nabla \cdot\bigg).$
Note that the constant in the eigenproblem does not affect the normal vector $\tfrac{1}{\vr}\nabla v \circ X$, so we can choose a representative $v$ for which the constant is zero. Analogously, zero eigenvalues of $\L_X$ correspond to zero eigenvectors of Hessian, and will be subsequently discarded. 
We thus arrive at proposition:
\begin{prop}\label{EulerTakensPotential}
	The Takens limit system for potential $U$ is given by incompressible Euler equation subject to additional potential force
	\be 
	V[X] = \sum c_i \sqrt{\lambda_i[X]},
	\ee 
	where 
	$\lambda_i[X]$ are non-zero eigenvalues of operator $\L_X$ and $c_i$ are constants given by the ill-preparedness of initial data
	\be \label{cidef}
	c_i = \frac{1}{2\lambda_i^{3/2}} \bigg(\int_M v_i \div V_0 \bigg)^2 + \frac{\lambda_i^{3/2}}{2}\bigg(\int_M  \frac{v_i  \varphi}{Q} \bigg)^2.
	\ee
\end{prop}
\begin{proof}[Computation of $c_i$]
	We take $X_0 = id$, $ X_{0N} = \frac{\nabla \varphi}{\vr_0}$. 
	We recall that 
	\be 
	c_i = \frac{|\mathbf{P}_i({\rm id}) V_0|^2 +  \lambda_i({\rm id})|\mathbf{P}_i({\rm id}) X_{0N}|^2}{2\sqrt{\lambda_i({\rm id})}},
	\ee 
	where $\mathbf{P}_i(X)$ are $L^2_{\vr_0}$ projectors onto eigenspaces of $\hess_{\vr_0} U_{X}.$ As computed above, $(\lambda, \frac{\nabla v}{\vr_0})$ is an eigenpair of $\hess_{\vr_0} U \big|_{N_{id}\mathsf{SDiff}_{\vr_0}}(M)$ if and only if 
	$(\lambda, v)$ is an eigenpair of $\L_{id}$. 
	The action constants $c_i$ are thus given by 
	\begin{align}
	c_i &= \frac{|\mathbf{P}_i({\rm id}) V_0|^2 +  \lambda_i({\rm id})|\mathbf{P}_i({\rm id}) X_{0N}|^2}{2\sqrt{\lambda_i({\rm id})}}=\frac{1}{2\sqrt{\lambda_i}}\int_M |\mathbf{P}_i V_0(a)|^2 \vr_0(a) \rmd a + \frac{\sqrt{\lambda_i}}{2}\int_M |\mathbf{P}_i X_{0N}(a)|^2 \vr_0(a)\rmd a\\
	&=\int_M\left[ \frac{1}{2\sqrt{\lambda_i}}\bigg|\frac{\int_M \langle \frac{\nabla v_i(b)}{\vr_0(b)}, V_0(b) \rangle \vr_0(b) \rmd b}{\int_M \big|\frac{\nabla v_i(b)}{\vr_0(b)}\big|^2 \vr_0(b) \rmd b} \frac{\nabla v_i(a)}{\vr_0(a)} \bigg|^2  + \frac{\sqrt{\lambda_i}}{2}\int_M \bigg|\frac{\int_M \langle \frac{\nabla v_i(b)}{\vr_0(b)}, \frac{\nabla \varphi(b)}{\vr_0(b)} \rangle \vr_0(b) \rmd b}{\int_M \big|\frac{\nabla v_i(b)}{\vr_0(b)}\big|^2 \vr_0(b) \rmd b} \frac{\nabla v_i(a)}{\vr_0(a)} \bigg|^2 \right]\vr_0(a) \rmd a\\
	&=\frac{1}{2\sqrt{\lambda_i}}\frac{(\int_M \langle \nabla v_i(b), V_0(b) \rangle\rmd b)^2}{\int_M \frac{|\nabla v_i(b)|^2}{\vr_0(b)}  \rmd b } + \frac{\sqrt{\lambda_i}}{2}\frac{(\int_M \langle \nabla v_i(b), \nabla \varphi(b) \rangle\frac{\rmd b}{\vr_0(b)})^2}{\int_M \frac{|\nabla v_i(b)|^2}{\vr_0(b)}  \rmd b }\\
	&= \frac{1}{2\lambda_i^{3/2}}\bigg(\int_M \langle \nabla v_i(b), V_0(b) \rangle\rmd b\bigg)^2 + \frac{1}{2\lambda_i^{1/2}}\bigg(\int_M \langle \nabla v_i(b), \nabla \varphi(b) \rangle\frac{\rmd b}{\vr_0(b)}\bigg)^2,
	\end{align}
where in the last equality we used that 
\be 
\int_M \frac{|\nabla v_i|^2}{\vr_0}\rmd a =  \int \frac{\L v_i v_i}{Q}\rmd a = \lambda_i.
\ee 
Continuing with the computation of the constants:
\begin{align}
		c_i &= \frac{1}{2\lambda_i^{3/2}}\bigg(\int_M \langle \nabla v_i(b), V_0(b) \rangle\rmd b\bigg)^2 + \frac{1}{2\lambda_i^{1/2}}\bigg(\int_M \langle \nabla v_i(b), \nabla \varphi(b) \rangle\frac{\rmd b}{\vr_0(b)}\bigg)^2\\
		&= \frac{1}{2\lambda_i^{3/2}} \bigg(\int_M v_i \div V_0 \bigg)^2 + \frac{1}{2\lambda_i^{1/2}}\bigg(\int_M  \frac{\L v_i  \varphi}{Q} \bigg)^2\\
		&= \frac{1}{2\lambda_i^{3/2}} \bigg(\int_M v_i \div V_0 \bigg)^2 + \frac{\lambda_i^{3/2}}{2}\bigg(\int_M  \frac{v_i  \varphi}{Q} \bigg)^2. 
\end{align}
This completes the derivation.
\end{proof}
\begin{remark} [Data of Métivier--Schochet]
	In \cite{MS03}, other variables are taken which results in the slightly different expression of adiabatic invariants. Authors of \cite{MS03} quantify ill-preparedness of initial data by the deviation $q_0$ of the physical pressure from constant:
	\be 
	\ve q_0 = \widetilde{P}(X_{0*}(\vr_0 \rmd a)/\rmd a, X_{0*} s_0) - \widetilde{P}(\vr_0, s_0).
	\ee 
	Denote by $F_\ve$ the flow of vector field $\nabla\varphi/\vr_0$ at time $\ve$. With that, going between the two formulations amounts to computing
	\begin{align}
	\ve q_0 &=  \widetilde{P}(F_{\ve*}(\vr_0 \rmd a)/\rmd a, F_{\ve*} s_0) - \widetilde{P}(\vr_0, s_0) + o(\ve)\\
	&=\widetilde{P}(\vr_0 - \ve \Delta \varphi, s_0 - \ve \langle\nabla\varphi,\nabla s_0\rangle/\vr_0) - \widetilde{P}(\vr_0, s_0) + o(\ve)\\
	&=-\ve(\partial_\vr \widetilde{P}(\vr_0, s_0)\Delta\varphi + \partial_s\widetilde{P}(\vr_0, s_0)\langle\nabla\varphi,\nabla s_0\rangle/\vr_0) + o(\ve)\\
	&=-\ve(\partial_\vr \widetilde{P}(\vr_0, s_0)\Delta\varphi - \partial_\vr\widetilde{P}(\vr_0, s_0)\langle\nabla\varphi,\nabla \vr_0\rangle)/\vr_0 + o(\ve)\\
	&=-\ve\vr_0\partial_\vr \widetilde{P}(\vr_0, s_0) \div(\nabla \varphi/\vr_0)+ o(\ve)\\
	&= \ve\L_{id}\varphi + o(\ve).
	\end{align}
	Thus in terms of $q_0$, the constants $c_i$ are written as
	\begin{align}
		c_i &= \frac{1}{2\lambda_i^{3/2}} \bigg(\int_M v_i \div V_0 \bigg)^2 + \frac{\lambda_i^{3/2}}{2}\bigg(\int_M  \frac{\L_{id}v_i  \varphi}{\lambda_i Q} \bigg)^2 \\
		&= \frac{1}{2\lambda_i^{3/2}} \bigg(\int_M v_i \div V_0 \bigg)^2 + \frac{1}{2\lambda_i^{1/2}}\bigg(\int_M  \frac{v_i  \L_{id}\varphi}{Q} \bigg)^2\\
		&= \frac{1}{2\lambda_i^{3/2}} \bigg(\int_M v_i \div V_0 \bigg)^2 + \frac{1}{2\lambda_i^{1/2}}\bigg(\int_M  \frac{v_i  q_0}{Q} \bigg)^2,
	\end{align}
	in agreement with the result in \cite{MS03}.
\end{remark}
\subsection{Incompressible limit: Computation of gradients of eigenvalues of $\hess_{\vr_0} U$  restricted to normal spaces}
In this section we compute the gradients of $\lambda_i[X]$ and consequently of $V[X]$.
We will consider a variation $X^\ve$ of $X$ satisfying  $\frac{\rmd}{\rmd \ve}\Big|_{\ve=0} X^\ve = \xi\circ X$, and denote as usual
\begin{align}
\vr^\ve = \tfrac{\vr_0 }{\det\nabla X^\ve} \circ {X^\ve}^{-1}, \ s^\ve= s_0 \circ {X^\ve}^{-1}, \ Q^\ve = \widetilde{Q}(\vr^\ve, s^\ve).
\end{align}
Additionally denote by $\mathcal{L}^\ve$ operator given by
\be 
\mathcal{L}^\ve = -Q^\ve\div\Big(\frac{1}{\vr^\ve}\nabla \Big),
\ee 
and by $v_i^\ve$ and $\lambda_i^\ve$ the corresponding eigenvalues and eigenvectors:
\be\label{eigenproblem}
\mathcal{L}^\ve v_i^\ve=\lambda_i^\ve v_i^\ve.
\ee
\begin{prop}\label{gradlambda}
	Let $\lambda = \lambda_i$ be an simple eigenvalue of $\L$ with eigenvector $v = v_i$. Then
	\be
	\grad^T_{\vr_0} \lambda [X]= - \PP_{T_{X}\mathsf{SDiff}(M)}\bigg(\frac{\lambda v^2}{ \vr Q^2}\nabla Q+\nabla\Big(\frac{|\nabla v|^2}{\vr^2}\Big) \bigg)\circ X,
	\ee
	where 
	$v$ is normalized by the condition 
	\be 
	\int{v^2}\frac{\rmd x}{Q} = 1.
	\ee 
\end{prop}
\begin{proof}
	Applying \ref{HF1} with $H = H^2(M)$ endowed with $L^2_{1/Q^\ve}$-weighted inner products and $A^\ve = \L^\ve$ we obtain
		\be\label{dlambda}
		\frac{\rmd}{\rmd \ve}\bigg|_{\ve = 0} \lambda^\ve = \int v\bigg(\frac{\rmd}{\rmd \ve}\bigg|_{\ve = 0}\mathcal{L}^\ve\bigg)v \frac{\rmd x}{Q}.
		\ee
	With this identity in hand, we proceed taking the variation. To that end, we have 
	\begin{align}
		\frac{\rmd}{\rmd \ve}\bigg|_{\ve = 0} \mathcal{L}^\ve &= - \frac{\rmd}{\rmd \ve}\bigg|_{\ve = 0} Q^\ve \div\bigg(\frac{1}{\vr}\nabla\bigg) + Q \div\bigg(\frac{\rmd}{\rmd \ve}\bigg|_{\ve = 0}\vr^\ve \frac{1}{\vr^2} \nabla\bigg)\\
		&= - \bigg({\partial_\vr \widetilde{Q}}\frac{\rmd}{\rmd \ve}\bigg|_{\ve = 0} \vr^\ve + {\partial_s \widetilde{Q}} \frac{\rmd}{\rmd \ve}\bigg|_{\ve = 0} s^\ve\bigg)\div\bigg(\frac{1}{\vr}\nabla\bigg) + Q \div\bigg(\frac{\rmd}{\rmd \ve}\bigg|_{\ve = 0}\vr^\ve \frac{1}{\vr^2} \nabla\bigg)\\
		&= ({\partial_\vr \widetilde{Q}}\div(\vr \xi) + {\partial_s \widetilde{Q}} \langle\xi, \nabla s\rangle) \div\bigg(\frac{1}{\vr}\nabla\bigg) - Q \div\bigg(\div(\vr\xi) \frac{1}{\vr^2} \nabla\bigg)\\
		&= (\vr {\partial_\vr \widetilde{Q}}\div\xi + \langle\xi,\nabla Q\rangle)\div\bigg(\frac{1}{\vr}\nabla\bigg) - Q \div\bigg(\div(\vr\xi) \frac{1}{\vr^2} \nabla\bigg).
	\end{align}
	Applying it to $v$ and using \eqref{eigenproblem},
	\begin{align}
		\frac{\rmd}{\rmd \ve}\bigg|_{\ve = 0} \mathcal{L}^\ve v &= (\vr {\partial_\vr \widetilde{Q}}\div\xi + \langle\xi,\nabla Q\rangle)\div\bigg(\frac{1}{\vr}\nabla v\bigg) - Q \div\bigg(\div(\vr\xi) \frac{1}{\vr^2} \nabla v\bigg)\\
		&= -(\vr {\partial_\vr \widetilde{Q}}\div\xi + \langle\xi,\nabla Q\rangle)\frac{\lambda v}{Q} - Q \div\bigg(\div(\vr\xi) \frac{1}{\vr^2} \nabla v\bigg).
	\end{align}
	Hence formula \eqref{dlambda} provides
	\begin{align}
		\frac{\rmd}{\rmd \ve}\bigg|_{\ve = 0} \lambda^\ve&= \int -(\vr {\partial_\vr \widetilde{Q}}\div\xi + \langle\xi,\nabla Q\rangle)\frac{\lambda v^2}{Q^2} - v \div\bigg(\div(\vr\xi) \frac{1}{\vr^2} \nabla v\bigg) \rmd x\\
		&=\int \bigg\langle \xi,\nabla\frac{\vr {\partial_\vr \widetilde{Q}}\lambda v^2}{Q^2} - \frac{\lambda v^2}{Q^2}\nabla Q -\vr\nabla\Big(\frac{|\nabla v|^2}{\vr^2}\Big)\bigg\rangle  \rmd x\\
		&= \bigg(\xi \circ X, \bigg(\frac{1}{\vr}\nabla\frac{\vr {\partial_\vr \widetilde{Q}}\lambda v^2}{Q^2} - \frac{\lambda v^2}{ \vr Q^2}\nabla Q-\nabla\Big(\frac{|\nabla v|^2}{\vr^2}\Big) \bigg)\circ X\bigg)_{\vr_0}.
	\end{align}
	Since $\xi \circ X$ is tangential to $\mathsf{SDiff}(M)$ and the first term is normal to it, this concludes the proof.  
\end{proof}

\begin{cor}
	Gradient of potential $V$ is given by 
	\be
	\grad^T_{\vr_0} V[X] =  - \PP_{T_{X}\mathsf{SDiff}(M)}\bigg( \frac{K_1}{ \vr Q^2}\nabla Q+\nabla\Big(\frac{K_2}{\vr^2}\Big) \bigg)\circ X,
	\ee 
	where 
	\be 
	K_1 = \sum_i \frac{c_i}{2}\sqrt{\lambda_i}v_i^2 \quad\text{and}\quad K_2 = \sum_i \frac{c_i}{2\sqrt{\lambda_i}}|\nabla v_i|^2.
	\ee 
	Takens limit system written in Eulerian formulation is 
	\be\label{takens-fluid}
	\begin{cases}
		\partial_t u + u\cdot\nabla u = - \frac{\nabla p}{\vr} + \frac{K_1}{ \vr Q^2}\nabla Q+\nabla\Big(\frac{K_2}{\vr^2}\Big),\\
		\partial_t \vr + u\cdot\nabla \vr = 0,\\
		\partial_t s + u\cdot\nabla s = 0,\\
		\div u = 0,
	\end{cases}
	\ee 
	where $p$ is determined to enforce incompressibility. 
\end{cor}
\begin{proof}
	The formula for the gradient of the potential is immediate from lemma \ref{gradlambda}. The Eulerian form follows analogously to the proof of \ref{compeuler}.
\end{proof}
\begin{proof}[Proof of theorem \ref{euler-theorem}]
	We only need to rewrite the additional forcing term appearing in \eqref{takens-fluid} in divergence form. This is a simple calculation:
	\begin{align}
		\vr\Big(\frac{K_1}{ \vr Q^2}\nabla Q+\nabla\frac{K_2}{\vr^2}\Big) &= -K_1\nabla\frac{1}{Q}+\vr\nabla\frac{K_2}{\vr^2} \\
		&= -\nabla\frac{K_1}{Q} + \frac{1}{Q}\nabla K_1 + \nabla\frac{K_2}{\vr} - \frac{K_2}{\vr^2}\nabla \vr\\
		&= -\nabla\frac{K_1}{Q} + \frac{1}{Q}\nabla K_1 + \nabla\frac{K_2}{\vr} + K_2\nabla \frac{1}{\vr}\\
		&= \nabla\Big(2\frac{K_2}{\vr} -\frac{K_1}{Q}\Big) + \frac{1}{Q}\nabla K_1 - \frac{1}{\vr}\nabla K_2\\
		&= \nabla\Big(2\frac{K_2}{\vr} -\frac{K_1}{Q}\Big) + \sum_i \frac{c_i}{2\sqrt{\lambda_i}} \Big( \frac{\lambda_i \nabla v_i^2}{Q}  - \frac{\nabla|\nabla v_i|^2}{\vr} \Big).
	\end{align}
	The terms inside the sum are computed from the eigenfunction equations as follows. Multiplying them by $\nabla v_i$, dividing by $Q$ and integrating by parts, we obtain
	\be 
	- \div\Big(\frac{\nabla v_i \otimes \nabla v_i}{\vr}\Big) + \frac{\nabla |\nabla v_i|^2}{2\vr} = \frac{\lambda_i \nabla v_i^2}{2Q},
	\ee 
	so the forcing term is given by
	\be
	\vr\Big(\frac{K_1}{ \vr Q^2}\nabla Q+\nabla\frac{K_2}{\vr^2}\Big) = \nabla\Big(2\frac{K_2}{\vr} -\frac{K_1}{Q}\Big) - \sum_i \frac{c_i}{\sqrt{\lambda_i}}  \div\Big(\frac{\nabla v_i \otimes \nabla v_i}{\vr}\Big).
	\ee
	After absorbing the gradient terms into pressure, the full equations \eqref{takens-fluid} are then written in conservation form as 
	\be
	\begin{cases}
		\partial_t (\vr u) + \div(\vr u \otimes u) + \div \Sigma + \nabla p= 0 ,\\
		\partial_t \vr + u\cdot\nabla \vr = 0\\
		\partial_t s + u\cdot\nabla s = 0\\
		\div u = 0,
	\end{cases}
	\ee 
	with acoustical stress $\Sigma$ given by \eqref{acousticalstress}.
\end{proof} 
\begin{remark}[Force of Métivier--Schochet]
	The extra force was derived in \cite{MS03} by a different method. Translating to our notation, they write it as 
	\be 
	F^{\rm MS} =  \frac{1}{Q\vr}\nabla K_1 - \frac{1}{\vr^2}\nabla K_{2},
	\ee 
	whereas we have 
	\be
	F^{our} =\frac{K_1}{ \vr Q^2}\nabla Q+\nabla\Big(\frac{K_2}{\vr^2}\Big).
	\ee
	It is immediate that our expressions are equivalent as they differ by the term normal to $\mathsf{SDiff}_{\vr_0}(M)$:
	\begin{align}
		F^{MS} - F^{our} &= \frac{1}{Q\vr}\nabla K_1 - \frac{1}{\vr^2}\nabla K_{2} - \frac{K_1}{ \vr Q^2}\nabla Q-\nabla\Big(\frac{K_2}{\vr^2}\Big)= \frac{1}{\vr}\nabla\bigg(\frac{K_1}{Q} - \frac{2 K_2}{\vr}\bigg).
	\end{align}
\end{remark}
\begin{remark}[Ideal gas law]
	If pressure is given by the ideal gas law
	\be 
	\widetilde{P}(\vr, s) = \frac{1}{\gamma}\vr^\gamma e^{s},
	\ee 
	$\widetilde{Q}$ is computed to be
$
	\widetilde{Q} = \vr\,\partial_\vr \widetilde{P}= \gamma \widetilde{P},
$
	so $Q$ is spatially constant on $\M$.
	Thus the effective force is given by $
	\nabla\frac{K_2}{\vr^2}.$
\end{remark}
\begin{remark}
	Our result also applies to shallow water equations and thermal shallow water equations (also known as Ripa model, see \cite{HL} and references within) with flat bottom. They arise as rigid lid 2D approximations of 3D compressible Euler equations for thin layer of fluid. Mathematically, they have the form of compressible Euler with internal energy given  by 
	\be 
	e_{\rm SW}(\vr) = \frac{1}{2}\vr \quad \text{and} \quad e_{\rm TSW}(\vr,s)= \frac{1}{2}\vr\bigg(1 + \alpha s\bigg). 
	\ee 
	Unfortunately, physical meaning of $\vr$ is depth of the fluid layer, while $s$ is fluctuation of density. Here $\alpha$ is a constant stratification parameter.  Corresponding equations of state are given by 
	\be 
	\widetilde{P}_{\rm SW}(\vr) = \frac{1}{2}\vr^2 \quad \text{and} \quad \widetilde{P}_{\rm TSW}(\vr,s)=  \frac{1}{2}\vr^2\bigg(1 + \alpha s\bigg).
	\ee 
	In the shallow water case we conclude from Remark \ref{barotropic-rmk} that ``incompressible" (although in this case it is rather "rigid lid") limit results in just 2D Euler. For the thermal shallow water, 
	\be  
	Q = 2P, 
	\ee  
	similar to the ideal gas case. The force is thus again given by $	\nabla\frac{K_2}{\vr^2}$.
\end{remark}
\section{Anelastic limit}
Consider Newton's equations on $\mathsf{Diff}_{\vr_0}(M)$ subject to the potential 
\be
U[X] := \int_M \bigg[e(\vr,s)\vr + G(\vr, x )\bigg].
\ee   
The Eulerian form of the equation is the compressible Euler with the force given by $-\nabla \partial_\vr G$.
\be 
\begin{cases}\label{compeulergravity}
	\partial_t u + u\cdot\nabla u = -  \frac{\nabla P}{\vr} -\nabla \partial_\vr G\\
	\partial_t \vr + \div(\vr u) = 0\\
	\partial_t s + \nabla_u s = 0
\end{cases}
\ee 
We will consider  the geophysical context, where the considered domain is
\be 
M  = \Omega \times [0,h] \subset \R^3, \quad \Omega \subset \R^2
\ee 
and $G$ is the gravitational potential
\be 
G(\rho, x) = g \rho z.
\ee
As before, we are interested in the singular limit of stiffening the potential $U$, and would like to derive the limiting constrained model as well as Takens system of $U$. In this setting, the operation of rescaling $U$ by $1/\ve^2$ and taking $\ve \rightarrow 0$ is equivalent to sending Mach and Froude number to zero at the same rate. The result of this operation is summarized in the following theorem.
\begin{theorem} We have that
	\begin{itemize}
		\item  If the fluid is isentropic (that is $s_0$ is constant and we can treat $\tilde{P} = \tilde{P}(\vr)$), and the equation of state satisfies $\partial_{\vr} \tilde{P} > 0$, then potential $U$ is constraining to the submanifold  
		\be 
	\M = \bigg\{ X \in \mathsf{Diff}_{\vr_0}(M) \ | \ X_* \vr_0 = \bar{\vr}\bigg\},
	\ee
	where $\bar{\vr} = \bar{\vr}(z)$ is given as a unique solution of
	\be 
	\frac{\widetilde{P}'(\bar{\vr})}{\bar{\vr}}\partial_z \bar{\vr} = - g, \ \text{with} \  \int_M \bar{\vr} \rmd x = \int_M \vr_0\rmd x .
	\ee
	Eulerian form of geodesic equation on $\M$ is given by anelastic equations:
	\be 
	\begin{cases}
		\partial_t u + u\cdot\nabla u = -  \nabla p\\
		\div (\bar{\vr}(z) u) = 0.
	\end{cases}
	\ee
	\item  Suppose that the fluid is non-isentropic, the reference state is hydrostatic
	\be 
	\partial_z \widetilde{P}(\vr_0, s_0) = - g\vr_0,
	\ee
	and the equation of state satisfies
	\be
	\vr \partial_{\vr}\tilde{P} > 0, \ \partial_s \tilde{P} > 0. 
	\ee 
	Suppose additionally that the reference state has strong entropy variations ($\partial_z s_0 \neq 0$), and that the square of Brunt–Väisälä frequency is positive
	\be 
	\vr_0 = \vr_0(z), \ s_0 = s_0(z), \ N^2 := -g\bigg(\frac{\partial_z\vr_0}{\vr_0} + \frac{g}{\partial_\vr \tilde{P}(\vr_0, s_0)}\bigg) > 0.
	\ee 
	 Then potential $U$ is constraining to the identity component of the submanifold of area-preserving rearrangements of each horizontal slice:
	\be 
	\M = \bigg\{ X = (X_{\rm h}(x,y, z), z) \in \mathsf{Diff}_{\vr_0}(M) \ | \ X_{\rm{h} *} \rmd x\rmd y = \rmd x\rmd y \bigg\},
	\ee
	Eulerian form of geodesic equation on $\M$ is given by decoupled 2D Euler equations:
	\be 
	\begin{cases}
		\partial_t u_{\rm h} + u_{\rm h}\cdot\nabla_{\rm h} u_{\rm h} = -  \nabla_{\rm h} p\\
		\div_{\rm h} u_{\rm h} = 0\\
		u_z = 0\\
		\vr = \vr_0\\
		s = s_0
	\end{cases}.
	\ee
\end{itemize}
In both cases, homogenized potential of $U$ is constant on $\M$, so Takens limit systems agree with the geodesic equations on the corresponding critical submanifolds.
\end{theorem}
It is interesting to contrast two scenarios: in the isentropic regime, we obtain anelastic equations which can be thought of as 3D incompressible Euler with modified incompressibility condition, while in the non-isentropic regime strong entropy stratification prevents vertical flow and results in just 2D Euler. Between these two regimes lie Boussinesq equations, which are obtained if one scales density/entropy stratification as $\vr_0, s_0 = {\rm const} + O(\ve^2)$. We do not pursue this route in detail here, as the Boussinesq buoyancy terms are non-singular and can be obtained in the usual way from the expansion of pressure.
\begin{proof}
Independent of the equation of state we have 
\be 
\grad_{\vr_0} U[X] = \bigg(\frac{\nabla P}{\vr} + \nabla \partial_\vr G\bigg)\circ X.
\ee 
The critical set $\M$ of $U$ is thus given by the force balance:
\be \label{anelasticforcebalance}
\nabla P = - \vr \nabla \partial_\vr G = - \vr g e_z.
\ee 
We now split the discussion into two cases, done in the next subsections.
\end{proof}
\subsection{Isentropic fluid}
We first consider barotropic fluid, with the equation of state given by $\widetilde{P} = \widetilde{P}(\vr)$. In this case, it is easy to check that the critical condition fixes elements of $\M$ to have a prescribed density stratification.
\begin{prop}\label{M-anelasticiso}
	For barotropic equation of state with $\widetilde{P}' > 0$, the connected component of critical set $\M$ of potential $U$ is given by
	\be
	\M = \bigg\{ X \in \mathsf{Diff}_{\vr_0}(M) \ | \ X_* \vr_0 = \bar{\vr}\bigg\},
	\ee
	where $\bar{\vr} = \bar{\vr}(z)$ is given as a unique solution of
	\be 
	\frac{\widetilde{P}'(\bar{\vr})}{\bar{\vr}}\partial_z \bar{\vr} = - g, \ \text{with} \  \int_M \bar{\vr} \rmd x = \int_M \vr_0 \rmd x
	\ee
	its tangent and normal spaces are given by
	\begin{align}
		T_X\M = \bigg\{\xi\circ X \ | \ \div(\bar{\vr}\xi) = 0\bigg\},\qquad N_X\M &= \bigg\{\nabla \varphi \circ X \ \bigg| \ \int_{M}\varphi  \rmd  x= 0\bigg\}.
	\end{align}
\end{prop}
\begin{proof}
Force balance \eqref{anelasticforcebalance} implies that on $\M$ pressure and density are functions of $z$ only. This in turn implies the ODE for density stratification:  
\be \label{strat}
\vr = \bar{\vr}(z), \ \text{where } \bar{\vr} \text{ solves} \ \widetilde{P}'(\bar{\vr})\partial_z \bar{\vr} = - g \bar{\vr}.
\ee 
Since we assumed that $\widetilde{P}' > 0$, the solution is unique up to determination of integration constant. The fact that total mass $\int_M \vr\rmd x$ is preserved under the action of $\mathsf{Diff}(M)$ fixes it.

Tangent vectors to $\M$ are given by those $\xi \circ X$ that do not change the density stratification to the first order so, $\xi$ has to satisfy
$0 = -\div(\bar{\vr}\xi).$
It is then immediate from Helmholtz decomposition that normal vectors are given by gradients.
\end{proof}
In the usual way, we rewrite the geodesic equation on $\M$ in the Eulerian form: 
\be 
\begin{cases}\label{anelastic}
	\partial_t u + u\cdot\nabla u = -  \nabla p\\
	\div (\bar{\vr}(z) u) = 0.
\end{cases}
\ee
The justification of taking the limit from \eqref{compeulergravity} to \eqref{anelastic} is carried out from PDE perspective in \cite{M07}; here we identify the corresponding formal Takens system by computing the critical manifold and its normal Hessian. For our formal identification, we need to check that potential $U$ is constraining to $\M$, as well as compute the potential Takens correction.
\begin{prop}
	The Hessian of $U$ on $\M$ is given by
	\be
		\hess_{\vr_0} U_X (\xi\circ X, \cdot)= \bigg(- \frac{\nabla \big(\widetilde{P}'(\bar{\vr})\div(\bar{\vr}\xi)\big)}{\bar{\vr}} - \frac{g\div(\bar{\vr}\xi)e_z}{\bar{\vr}}\bigg)\circ X.
	\ee
	In particular, its eigenvalues are constant on $\M$ and when restricted to $N_X \M$ it is positive definite.
\end{prop}
\begin{proof}
	The computation of Hessian is simple:
\begin{align}
\hess_{\vr_0} U_X (\xi\circ X, \cdot)&= \frac{\rmd}{\rmd\ve}\bigg|_{\ve = 0} \grad_{\vr_0} U[X^\ve] = \frac{\rmd}{\rmd\ve}\bigg|_{\ve = 0}\bigg(\frac{\nabla P}{\vr} + g e_z\bigg)\circ X\\
&= \bigg(- \frac{\nabla \big(\widetilde{P}'(\bar{\vr})\div(\bar{\vr}\xi)\big)}{\bar{\vr}} + \frac{\div(\bar{\vr}\xi)\nabla P}{\bar{\vr}^2}\bigg)\circ X\\
&= \bigg(- \frac{\nabla \big(\widetilde{P}'(\bar{\vr})\div(\bar{\vr}\xi)\big)}{\bar{\vr}} - \frac{g\div(\bar{\vr}\xi)e_z}{\bar{\vr}}\bigg)\circ X
\end{align}
It is then immediate that eigenvalues are the same as those of the operator 
\be 
\xi \rightarrowtail - \frac{\nabla \big(\widetilde{P}'(\bar{\vr})\div(\bar{\vr}\xi)\big)}{\bar{\vr}} - \frac{g\div(\bar{\vr}\xi)e_z}{\bar{\vr}}
\ee 
which is constant on $\M$.
With that we first compute that $\hess_{\vr_0} U_X$ is positive  definite when restricted to normal spaces to $\M$. Using characterization of normal vectors from Proposition \eqref{M-anelasticiso}, 
\begin{align}
\hess_{\vr_0} U_X(\nabla\varphi \circ X, \nabla\varphi \circ X) &= \int_{M} \bigg\langle- \frac{\nabla \big(\widetilde{P}'(\bar{\vr})\div(\bar{\vr}\nabla\varphi)\big)}{\bar{\vr}}  -\frac{g\div(\bar{\vr}\nabla\varphi)e_z}{\bar{\vr}},\nabla\varphi\bigg\rangle\bar{\vr}\\
&=\int_M P'(\bar{\vr})\Delta \varphi \div(\bar{\vr}\nabla\varphi) - g\div(\bar{\vr}\nabla\varphi)\partial_z \varphi\\
&= \int_M P'(\bar{\vr})\Delta \varphi \partial_z\bar{\vr} \partial_z \varphi + P'(\bar{\vr})\bar{\vr}(\Delta \varphi)^2- g\partial_z\bar{\vr}(\partial_z \varphi)^2 - g \bar{\vr}\Delta\varphi\partial_z\varphi\\
&=-g\int_M 2\bar{\vr}\Delta\varphi\partial_z \varphi + \frac{\bar{\vr}^2(\Delta\varphi)^2}{\partial_z \bar{\vr}} + \partial_z \bar{\vr} (\partial_z \varphi)^2 \\
&= -g \int_M\frac{1}{\partial_z\bar{\vr}}\bigg( \bar{\vr}\Delta\varphi + \partial_z \bar{\vr}\partial_z \varphi\bigg)^2 = -g \int_M\frac{1}{\partial_z\bar{\vr}}(\div(\bar{\vr}\nabla\varphi))^2 \geq 0.
\end{align}
A little more work shows in fact that the following coercive bound holds
\be
\hess_{\vr_0} U_X(\nabla\varphi \circ X, \nabla\varphi \circ X) \gtrsim \|\nabla \varphi \circ X\|_{L^2_{\vr_0}}^2>0.
\ee
This completes the proof.
\end{proof}
This shows that potential $U$ does indeed attract the solution towards $\M$. Moreover, the eigenvalues of operator $\hess_{\vr_0} U\big|_{\N_X \M}$ are independent of $X$, so the homogenized potential of $U$ is constant on $\M$. Hence the Takens limit system is also given by \eqref{anelastic}, and we formally expect the convergence of solutions of \eqref{compeulergravity} to those of \eqref{anelastic} to hold for mildly ill-prepared data as well.
\subsection{Non-isentropic fluid}

Now we consider general baroclinic equation of state $P = \widetilde{P}(\vr,s)$. We assume that reference state has non-trivial entropy variations - that is 
\be 
\partial_z s_0 \neq 0,
\ee 
as well as that the equation of state satisfies 
\be 
\widetilde{Q}(\vr,s) := \vr \partial_\vr \widetilde{P} > 0 ,\quad  \partial_s\widetilde{P} > 0.
\ee 
\begin{prop}
	Under the stated assumptions, $\M$ consists of area-preserving rearrangements on each horizontal slice. Its tangent and normal spaces are given by
	\begin{align}
	T_X \M &= \{\xi \circ X \ | \ \xi = (\xi_{\rm h}, 0), \div_{\rm h} \xi_{\rm h} = 0\} \ \text{and}\\
	N_X \M &= \{\xi \circ X \ | \ \xi = (\nabla_{\rm h} \varphi, \xi_z), \varphi = \varphi(x,y,z), \xi_z = \xi_z(x,y,z), \int \varphi \rmd x \rmd y = 0\}
	\end{align} 
\end{prop}
\begin{proof}
Fix $X \in \M$. We have at $X$
\be \label{anelastic_constraint}
\nabla P = - g\vr e_z,
\ee 
so $P = P(z)$. This implies that $\vr = \vr(z)$ is also horizontally uniform. Now we have
\be
P(z) = \widetilde{P}(\vr(z),s(x,y,z)).
\ee
From the assumption that $\partial_s \widetilde{P} > 0$, we conclude that entropy is also stratified 
\be 
s = s(z).
\ee 
With that we claim that the only elements of $\M$ are those that are given by volume-preserving rearrangements on each horizontal slice. To see this, we consider a variation $X^\ve$ of $X$ through $\M$ with $\frac{\rmd}{\rmd \ve }\Big|_{\ve = 0} X^\ve = \xi\circ X$. We compute the corresponding variation of constraint \eqref{anelastic_constraint}:
\begin{align} \label{anis-var}
0 &= -\nabla(Q\div\xi + \langle\nabla P,\xi\rangle) - g\div(\vr\xi)e_z\\
 &= -\nabla(Q\div\xi  -\vr g \xi_z) - g\div(\vr\xi)e_z
\end{align}
This implies that $\div(\vr \xi)$ is a function of $z$ only.

We then look at the variation of $s$. To this end we compute,
\be
\frac{\rmd}{\rmd \ve}\bigg|_{\ve = 0} s_{X^\ve} = - \langle \nabla s, \xi\rangle = - \partial_z s \xi_z.
\ee 
Since we assumed that $\partial_z s \neq 0$, this implies that $\xi_z = \xi_z(z)$. Combining these two, 
\be 
\div(\vr\xi) - \partial_z(\vr\xi_z)  = \vr\div_{\rm h} \xi_{\rm h},
\ee 
so  that $\div_{\rm h} \xi_{\rm h}$ is also a function of $z$. Integrating it over horizontal slices implies 
\be 
\div_{\rm h} \xi_{\rm h} = 0.
\ee 
Consequently, the vertical component of \eqref{anis-var} becomes an ODE for $\xi_z(z)$
\be 
0 = -\partial_z(Q \partial_z \xi_z - \vr g\xi_z) - g\partial_z(\vr \xi_z) = -\partial_z (Q\partial_z \xi_z).
\ee 
It can be integrated to 
$
\xi_z = \int^{z} \frac{\rm const}{Q}\rmd z.
$
Since we assumed that $Q$ is sign-definite, the only solution satisfying boundary conditions is 
$ 
\xi_z = 0.
$
Thus all tangent vectors to $\M$ are given by divergence-free horizontal vector fields, from which characterization of $\M$ and $N_X \M$ follows.
\end{proof}
With that, it is clear that geodesic equation on $\M$ splits into decoupled 2D incompressible Euler equations on each horizontal slice:
\be 
\begin{cases}\label{inceuler2D}
	\partial_t u_{\rm h} + u_{\rm h}\cdot\nabla_{\rm h} u_{\rm h} = -  \nabla_{\rm h} p\\
	\div_{\rm h} u_{\rm h} = 0\\
	u_z = 0\\
	\vr = \vr_0\\
	s = s_0
\end{cases}
\ee
 As mentioned, entropy stratification provides buoyancy effects preventing mixing of layers. We proceed to investigate $\hess_{\vr_0} U\big|_{N_X\M}$.
\begin{prop}
	Hessian of $U$ is given by 
	\be
	\hess_{\vr_0} U_X (\xi \circ X, \cdot) = \bigg(- \frac{\nabla \big(Q\div\xi - \vr g \xi_z\big)}{\vr} - \frac{g\div(\vr\xi)e_z}{\vr}\bigg)\circ X
	\ee 
	When restricted to $N_X \M$, it is positive definite provided the square of Brunt–Väisälä frequency is positive
	\be 
	N^2 := -g\bigg(\frac{\partial_z\vr}{\vr} + \frac{\vr g}{Q}\bigg) > 0.
	\ee 
\end{prop}
\begin{proof}
	The computation of Hessian is again immediate using constraint \eqref{anelastic_constraint}:
	\begin{align}
	\hess_{\vr_0} U_X (\xi\circ X, \cdot) &= \bigg(- \frac{\nabla \big(Q\div \xi + \langle \nabla P, \xi\rangle\big)}{\vr} + \frac{\div(\vr\xi)\nabla P}{\vr^2}\bigg)\circ X\\
	&= \bigg(- \frac{\nabla \big(Q\div\xi - \vr g \xi_z\big)}{\vr} - \frac{g\div(\vr\xi)e_z}{\vr}\bigg)\circ X.
	\end{align}
	Thus, we have
	\begin{align}
	\hess_{\vr_0} U_X (\xi\circ X, \xi \circ X) 
	&= \int_M \bigg\langle- \nabla (Q\div\xi - \vr g \xi_z) - g\div(\vr \xi) e_z, \xi \bigg\rangle\\ 
	&=\int_M Q(\div\xi)^2 - \vr g \xi_z\div\xi - g\div(\vr \xi) \xi_z\\
	&= \int_M Q(\div \xi)^2 - 2\vr g \xi_z\div\xi - g \partial_z \vr \xi_z^2\\
	&= \int_M Q\bigg(\div\xi - \frac{\vr g}{Q}\xi_z\bigg)^2 - g\bigg(\partial_z\vr + \frac{\vr^2 g}{Q}\bigg)\xi_z^2,
\end{align}
so it is positive if
$ 
 -g\big(\frac{\partial_z\vr}{\vr} + \frac{\vr g}{Q}\big) > 0,
$
which can be recognized as a baroclinic version of Brunt–Väisälä frequency.
Finally elements of $N_X \M$ are given by $\xi \circ X = (\nabla_{\rm h} \varphi, \xi_z) \circ X$, so if the quadratic form vanishes, then $N^2>0$ implies
$\xi_z=0$. The first square then gives
\be
\Delta_{\rm h}\varphi=0
\ee
on every horizontal slice. The slice-wise mean-zero condition implies
$\varphi=0$, and hence $\xi=0$. Therefore Hessian is positive
definite on $N_X\mathcal M$.
\end{proof}

We now observe that the expression for the normal Hessian only depends on $X$ through profiles $\vr, s$ which are the same for all elements of critical set $\M$. Thus homogenized potential is constant and Takens limit system is given by \eqref{inceuler2D} as well. We conclude that we formally expect the convergence of \eqref{compeulergravity} to \eqref{inceuler2D} to hold for mildly ill-prepared data.

%

\section{Lake models as limits of strong constraining force}
In this section we realize lake and Great Lake equations as limits of Newton's equations with strong constraining force. The corresponding compressible models are given by shallow water equations and Green-Naghdi equations respectively. In both cases, the Takens correction vanishes.
\subsection{Lake equation}
Recall shallow water equations in a domain $\Omega \subset \R^2$
\be\label{SW}
\begin{cases}
	\partial_t u + u\cdot \nabla u = -\frac{1}{\ve^2}\nabla(\vr - b)\\
	\partial_t \vr + \div(\vr u ) = 0\\
	\vr|_{t = 0} = \vr_0
\end{cases}
\ee 
They arise from depth-integrating incompressible 3D Euler equations under the influence of gravity in a thin domain $\Omega \times [0, \ve]$ after assuming that the horizontal velocity is effectively constant in each vertical column of the fluid. Here $\vr_0, b : \Omega \to \R$ are given functions, representing  initial height of the fluid layer and depth of the bottom respectively, and $\ve$ is the Froude number quantifying the relative importance of inertial and gravitational forces. 

It is easy to see that shallow water equations are in fact ($1/\ve^2$ -scaled) Newton's equation on $\mathsf{Diff}_{\vr_0}(\Omega)$ with potential corresponding to the total weight of the fluid
\be 
U[X] = \int_\Omega (\tfrac{\vr^2}{2} - \vr b)\rmd x, \qquad 	\grad_{\vr_0} U[X] = \nabla(\vr-b) \circ X.
\ee  
Thus, the Eulerian form of shallow water equations is given by \eqref{SW}.

We consider the low-Froude limit $\ve \to 0$, and compute the corresponding Takens system.
\begin{theorem}
	The potential $U$ is constraining to the set
	\be 
	\M =  \{X\in \mathsf{Diff}_{\vr_0}(\Omega) \ | \ \vr = \bar{H} + b\},
	\ee
	where $\bar{H} = \frac{1}{|\Omega|}\int_{\Omega} (\vr_0 - b)\rmd x$ is the average actual depth of the fluid.
	The Eulerian form of geodesic equation on $\M$ is given by lake equation
	\be
	\begin{cases}\label{lake}
		\partial_t u + u\cdot \nabla u = -\nabla P\\
		\div((\bar{H}+b) u ) = 0.
	\end{cases}
	\ee
	Homogenized potential of $U$ is constant on $\M$, so Takens limit system agrees with \eqref{lake}.
\end{theorem}
\begin{proof} 
We begin with the following Lemma:
\begin{lemma}
The critical set $\M$ of $U$ is given by
	\be 
	\M =  \{X\in \mathsf{Diff}_{\vr_0}(\Omega) \ | \ \vr = \bar{H} + b\},
	\ee
	where $\bar{H} = \frac{1}{|\Omega|}\int_{\Omega} (\vr_0 - b)\rmd x$ is the average  depth of the fluid.
	Its tangent and normal spaces are 
	\begin{align}
	T_X \M = \{ \xi\circ X  \ | \ \div((\bar{H} + b)\xi) = 0 \} \quad \text{and}\quad
	N_X \M &= \bigg\{ \nabla \varphi \circ X \ | \ \int_{\Omega} \varphi = 0\bigg\}.
	\end{align}
\end{lemma} 
\begin{proof}
The characterization of $\M$ is obvious; one must notice that under action of $\mathsf{Diff}_{\vr_0}(\Omega)$
\be 
\frac{1}{|\Omega|}\int_\Omega (\vr - b)\rmd x =  \frac{1}{|\Omega|}\int_\Omega (\vr_0 - b)\rmd x = \bar{H}.
\ee 
Tangent spaces to $\M$ are thus given by vector fields that preserve $\vr$ to the first order:
\be 
T_X \M = \{ \xi\circ X  \ | \ \div(\vr\xi) = 0 \} = \{ \xi\circ X  \ | \ \div((\bar{H} + b)\xi) = 0 \}. 
\ee 
Normal spaces are, as usual, given by gradients.
\end{proof}

Consequently, geodesic equation on $\M$ is given by the lake equation:
\be
\begin{cases}
	\partial_t u + u\cdot \nabla u = -\nabla P\\
	\div((\bar{H}+b) u ) = 0.
\end{cases}
\ee 
We proceed to compute the Hessian of potential $U$ on $\M$.
\begin{lemma}One has
	\be 
	\hess_{\vr_0} U_X (\xi\circ X, \cdot) = -\nabla\div(\vr\xi) \circ X.
	\ee
	It is positive definite on $N_X\M$ and its eigenvalues are constant on $\M$.
\end{lemma}
\begin{proof}
	The computation is trivial:
	\be 
	\hess_{\vr_0} U_X (\xi\circ X, \cdot) = -\nabla\div(\vr\xi) \circ X.
	\ee
	
	Consequently, restricted to normal space we have 
	\be 
	\hess_{\vr_0} U_X (\nabla \varphi \circ X, \nabla \varphi \circ X) = \int_{\Omega} \div(\vr\nabla\varphi)^2  \rmd x.
	\ee 
	Since $\vr = \bar{H} + b$ is the same profile for every element $X$ of $\M$, so are eigenvalues of the Hessian, and, consequently, homogenized potential of $U$. Thus, Takens limit system is given by \eqref{lake} as well.
\end{proof}
This completes the proof of the theorem.
\end{proof}
\subsection{Great Lake equation}
We now treat the Great Lake equations.
Let $\Omega \subset \R^2$ be a bounded domain, and consider $\mathsf{Diff}(\Omega)$ with the following metric 
\be
\bigg( \xi \circ X, \eta \circ X \bigg)_{\vr_0, \rm GN} := \int_\Omega \vr \langle \xi, \eta\rangle + \frac{1}{3}\vr^3 \div \xi \div \eta + \frac{1}{2} \vr^2 \div \xi \langle \eta, \nabla b\rangle +  \frac{1}{2} \vr^2 \div \eta \langle \xi, \nabla b\rangle + \vr \langle \xi, \nabla b\rangle\langle\eta, \nabla b\rangle.
\ee
Here $b$ is a fixed positive function representing the depth of bottom, and $\vr = X_*\vr_0(\rmd a)/\rmd a$ is transported density as before representing the height of fluid layer.
Defining $A_\vr$ by 
\be 
A_\vr \eta = \eta - \frac{1}{3}\frac{\nabla(\vr^3 \div \eta)}{\vr} - \frac{1}{2}\frac{\nabla(\vr^2 \langle \eta, \nabla b\rangle)}{\vr} +  \frac{1}{2} \vr \div \eta \nabla b + \langle\eta, \nabla b\rangle
\nabla b \ee 
the metric is compactly written as 
\be 
\bigg( \xi \circ X, \eta \circ X \bigg)_{GN, \vr} = \int_{\Omega} \vr \langle \xi, A_{\vr} \eta\rangle. 
\ee 
We remark that $A_\rho$ is invertible on the class of smooth vector fields tangent to the boundary.
\begin{remark}
We note that compared to all the previous examples this metric is not "diagonal", so the distinction between Levi-Civita connection of $\mathsf{Diff}_{\vr_0, \rm GN}(\Omega)$ and the pointwise connection on the Euclidean domain $\Omega$ has to be made. For this reason we keep the covariant notation.
\end{remark}
 
Consider, as before, the following potential corresponding to the total weight of the fluid
\be 
U[X] = \int_\Omega (\tfrac{\vr^2}{2} - \vr b)\rmd x. 
\ee

\begin{prop}
Newton's equations with potential $\tfrac{1}{\ve^2}U$ on $\mathsf{Diff}_{\vr_0, \rm GN}(\Omega)$ are given by 
\be\label{Lagrangian-GN}
\begin{cases}
	\nabla_{\dot{X}}\dot{X} = - \frac{1}{\ve^2}\eta \circ X \\
	A_{\vr} \eta = \nabla(\vr - b) 
\end{cases}
\ee
The Eulerian velocity field $u = \dot{X} \circ X^{-1}$ satisfies Green-Naghdi equations
\be 
\begin{cases}
	
	\partial_t v + \nabla_u v + (\nabla u)v + \nabla\bigg( \frac{1}{\ve^2}(\vr - b) - \frac{1}{2}|u|^2 - \frac{1}{2}(\vr\div u + \langle u, \nabla b\rangle)^2\bigg)  = 0 ,\\
	v = A_{\vr}u,\\
	\partial_t \vr + \div(\vr u) = 0.
\end{cases}
\ee 
\end{prop}
\begin{proof}
	Consider a variation $X^\ve$ with $\frac{\rmd}{\rmd \ve}\big|_{\ve = 0} X^\ve = \xi \circ X$. The variation of potential is given by
	\be 
	\frac{\rmd}{\rmd \ve}\bigg|_{\ve = 0} U[X^\ve] = \int_{\Omega} \vr \langle\xi, \nabla(\vr - b)\rangle\rmd x
	\ee 
	so that
	\be 
	\grad_{\vr_0, \rm GN} U[X] = \eta \circ X \ \text{with} \ A_{\vr} \eta = \nabla(\vr - b).
	\ee 
	The Lagrangian form of Newton's equations follows.
	To get the Eulerian form, we compute $\nabla_{\dot{X}} \dot{X} \circ X^{-1}$. In principle, this can be done using Koszul formula but to avoid writing it out in full we use metric compatibility of Levi-Civita connection. For any time-independent field $\eta$ we must have
	\be 
	\frac{\rmd }{\rmd t} \bigg(\dot{X} , \eta \circ X \bigg)_{\vr_0, \rm GN}= \bigg(\nabla_{\dot{X}} \dot{X} , \eta \circ X \bigg)_{\vr_0, \rm GN} + \bigg(\dot{X} , \nabla_{\dot{X}}(\eta \circ X) \bigg)_{\vr_0, \rm GN} .
	\ee  
	Expanding the left-hand side we get
	\be 
	\frac{\rmd }{\rmd t} \bigg(\dot{X} , \eta \circ X \bigg)_{\vr_0, \rm GN} = \frac{\rmd }{\rmd t} \int_{\Omega} \vr \langle A_{\vr} (\dot X \circ X^{-1}), \eta\rangle = \frac{\rmd }{\rmd t} \int_{\Omega} \vr \langle v, \eta\rangle = \int_{\Omega} \langle \partial_t(\vr v), \eta\rangle.
	\ee 
	On the other hand, right-hand side is computed to be 
	\begin{align}
	\bigg(&\nabla_{\dot{X}} \dot{X} , \eta \circ X \bigg)_{\vr_0, \rm GN} + \bigg(\dot{X} , \nabla_{\dot{X}}(\eta \circ X) \bigg)_{\vr_0, \rm GN} \\
	&=\bigg(\nabla_{\dot{X}} \dot{X} , \eta \circ X \bigg)_{\vr_0, \rm GN} + \bigg(\dot{X} , \nabla_{\eta \circ X}\dot{X} \bigg)_{\vr_0, \rm GN} + \bigg(\dot{X} , [\dot{X}, \eta \circ X]\bigg)_{\vr_0, \rm GN}\\
	&=\int_{\Omega} \vr \langle A_\vr(\nabla_{\dot{X}} \dot{X} \circ X^{-1}),  \eta\rangle + \frac{1}{2}(\eta \circ X )\bigg(\dot{X} ,\dot{X} \bigg)_{\vr_0, \rm GN} +\int_{\Omega} \vr \langle A_\vr(\dot{X} \circ X^{-1}),  [\dot{X}, \eta \circ X]\circ X^{-1}\rangle \\
	&=\int_{\Omega} \vr \langle A_\vr(\nabla_{\dot{X}} \dot{X} \circ X^{-1}),  \eta\rangle +  \frac{1}{2}(\eta \circ X )\int_{\Omega}\vr\langle u, A_\vr u\rangle +\int_{\Omega} \vr \langle v,  [u, \eta]\rangle.
	\end{align}
The scalar variation term is computed for any right-invariant field (e.g. $u \circ X$) as 
\begin{align}
(\eta \circ X )\int_{\Omega}\vr\langle u, A_{\vr} u\rangle &= (\eta \circ X )\int_{\Omega}\vr |u|^2 + \frac{1}{3}\vr^3 \div u ^2 + \vr^2 \div u \langle u,\nabla b\rangle + \vr \langle u, \nabla b\rangle^2 \\
&=-\int_{\Omega}\div (\vr\eta)(|u|^2 + \vr^2 \div u ^2 + 2\vr \div u \langle u,\nabla b\rangle + \langle u, \nabla b\rangle^2) \\
&=\int_{\Omega}\bigg\langle\vr\eta,\nabla(|u|^2 + (\vr\div u + \langle u, \nabla b\rangle)^2 )\bigg\rangle.
\end{align}
Continuing on, we have  
	\begin{align}
	\int_{\Omega} &\langle \partial_t(\vr v), \eta\rangle = \frac{\rmd }{\rmd t} \bigg(\dot{X} , \eta \circ X \bigg)_{\vr_0, \rm GN} = \bigg(\nabla_{\dot{X}} \dot{X} , \eta \circ X \bigg)_{\vr_0, \rm GN} + \bigg(\dot{X} , \nabla_{\dot{X}}(\eta \circ X) \bigg)_{\vr_0, \rm GN} \\
	&=\int_{\Omega} \vr \langle A_\vr(\nabla_{\dot{X}} \dot{X} \circ X^{-1}),  \eta\rangle +  \frac{1}{2}(\eta \circ X )\int_{\Omega}\vr\langle u, v\rangle +\int_{\Omega} \vr \langle v,  [u, \eta]\rangle\\
	&=\int_{\Omega} \vr \langle A_\vr(\nabla_{\dot{X}} \dot{X} \circ X^{-1}),  \eta\rangle +  \frac{1}{2}\int_{\Omega}\bigg\langle\vr\eta,\nabla(|u|^2 + (\vr\div u + \langle u, \nabla b\rangle)^2 )\bigg\rangle +\int_{\Omega} \vr \langle v,  \nabla_u \eta - \nabla_{\eta} u \rangle\\
	&=\int_{\Omega} \bigg\langle  \vr A_\vr(\nabla_{\dot{X}} \dot{X} \circ X^{-1}) + \frac{1}{2}\vr\nabla(|u|^2 + (\vr\div u + \langle u, \nabla b\rangle)^2)  -\div(\vr u)v - \vr\nabla_u v - \vr (\nabla u)v, \eta \bigg\rangle.
\end{align}
Comparing two expressions, we deduce that 
\be 
\partial_t(\vr v) =\vr A_\vr(\nabla_{\dot{X}} \dot{X} \circ X^{-1}) + \frac{1}{2}\vr\nabla(|u|^2 + (\vr\div u + \langle u, \nabla b\rangle)^2)  -\div(\vr u)v - \vr\nabla_u v - \vr (\nabla u)v,
\ee 
or
\be \label{covariant-GN}
\vr A_\vr(\nabla_{\dot{X}} \dot{X} \circ X^{-1}) = \partial_t(\vr v) +\div(\vr u)v + \vr\nabla_u v + \vr (\nabla u)v - \frac{1}{2}\vr\nabla(|u|^2 + (\vr\div u + \langle u, \nabla b\rangle)^2).
\ee
 
Composing  Lagrangian equation \eqref{Lagrangian-GN} with $X^{-1}$, applying $\vr A_\vr$ to and using \eqref{covariant-GN}, 
\be 
\partial_t(\vr v) +\div(\vr u)v + \vr\nabla_u v + \vr (\nabla u)v - \frac{1}{2}\vr\nabla(|u|^2 + (\vr\div u + \langle u, \nabla b\rangle)^2)  = -\frac{1}{\ve^2}\vr \nabla (\vr - b).
\ee 
Since the density $\vr$ is transported, we conclude that
\be 
\partial_t v + \nabla_u v + (\nabla u)v + \nabla\bigg( \frac{1}{\ve^2}(\vr - b) - \frac{1}{2}|u|^2 - \frac{1}{2}(\vr\div u + \langle u, \nabla b\rangle)^2\bigg)  = 0.
\ee 
This completes the derivation.
\end{proof}
With that we proceed to take small $\ve$ limit, that is computing geodesic motion and the Takens system on the critical manifold $\M$ of potential $U$. As in the case of lake equation, it is given by 
\be 
\M = \{ X \ | \ \vr - b = \bar{H}\}.
\ee  
Tangent space is again given by vector fields that preserve $\vr$ to first order.
\be 
T_X \M = \{ \xi\circ X  \ | \ \div(\vr\xi) = 0 \}.
\ee 
The normal space, however, is modified from the lake case, since the metric is different. 
It is
\be 
N_X \M = \{ \eta \circ X \ | A_{\vr}\eta = \nabla \varphi\}.
\ee 
\begin{prop}
	If $X$ solves geodesic equation on $\M$, $u = \dot{X} \circ X^{-1}, \vr = \bar{H} + b$ and $v = A_{\vr} u$ solve Great Lake equations
	\be
	\begin{cases}\label{greatlake} 
		\partial_t v + \nabla_u v +  (\nabla u)v - \frac{1}{2}\nabla|u|^2 = \nabla \tilde{\varphi} \\
		v = u + \frac{1}{6}\vr^2\nabla \div u\\
		\div ( \vr u) = 0\\
		\vr = \bar{H} + b.
	\end{cases} 
	\ee
	The homogenized potential of $U$ is constant on $\M$, so the Takens limit system agrees with \eqref{greatlake}.
\end{prop}
\begin{proof}
Geodesic equation on $\M$ is written as usual:
\be 
\begin{cases} 
	\nabla_{\dot{X}}\dot{X} = \eta \circ X\\
	A_{\vr} \eta = \nabla \varphi\\ 
	\div(\vr u) = 0\\
	\vr = \bar{H} + b
\end{cases}.\ee 
Applying $\vr A_{\vr}$ to the first equation and reusing \eqref{covariant-GN}, we deduce that $u, v= A_{\vr} u$ and $\vr$ satisfy
\be
\begin{cases} 
\partial_t v + \nabla_u v +  (\nabla u)v - \frac{1}{2}\nabla|u|^2 = \nabla \tilde{\varphi} \\
v = A_{\vr} u\\
\div ( \vr u) = 0\\
\vr = \bar{H} + b
\end{cases} .
\ee
We note that when computed on $\M$ with a tangent vector $\xi \circ X $, $A_\vr$ simplifies to
\begin{align}
	A_\vr \xi &= \xi - \frac{1}{3}\frac{\nabla(\vr^3 \div \xi)}{\vr} - \frac{1}{2}\frac{\nabla(\vr^2 \langle \xi, \nabla b\rangle)}{\vr} +  \frac{1}{2} \vr \div \xi \nabla b + \langle\xi, \nabla b\rangle \nabla b\\
	&= \xi - \frac{1}{3}\frac{\nabla(\vr^3 \div \xi)}{\vr} + \frac{1}{2}\frac{\nabla(\vr^3 \div \xi)}{\vr} +  \frac{1}{2} \vr \div\xi \nabla \vr - \vr \div \xi\nabla \vr\\
	&= \xi + \frac{1}{6}\frac{\nabla(\vr^3 \div \xi)}{\vr} - \frac{1}{2} \vr \div\xi \nabla \vr\\
	&= \xi + \frac{1}{6}\vr^2\nabla \div\xi.
\end{align} 
Here in the second equality we used that since $X \in \M, \xi \in T_X\M$
\begin{align}
\nabla b &= \nabla \vr \qquad \text{and} \qquad
0 = \div(\vr \xi) = \vr \div \xi + \langle\xi, \nabla \vr\rangle = \vr \div \xi + \langle\xi, \nabla b\rangle.
\end{align}
Consequently, the equations can be rewritten into the form  
\be
\begin{cases} 
	\partial_t v + \nabla_u v +  (\nabla u)v - \frac{1}{2}\nabla|u|^2 = \nabla \tilde{\varphi} \\
	v = u + \frac{1}{6}\vr^2\nabla \div u\\
	\div ( \vr u) = 0\\
	\vr = \bar{H} + b.
\end{cases} 
\ee
Upon setting $\bar{H} = 0$ and eliminating $\vr$ in favor of $b$, these are Great Lake equations of \cite{CHL}.

We proceed to compute the Hessian of potential $U$ restricted to $\M$. The computation is surprisingly simple and turns out to be the same as for the lake equation.

We first compute that if $X^\ve$ is a variation of $X$ and $\eta^\ve$ is the solution of
\be 
A_{\vr^\ve} \eta^\ve = \nabla (\vr^\ve - b),
\ee
then the variation of $\eta$ is given by
\begin{align}
\frac{\rmd}{\rmd \ve} A_{\vr^\ve} \eta + A_{\vr} \frac{\rmd}{\rmd \ve} \eta^\ve &= \frac{\rmd}{\rmd \ve}\nabla(\vr - b), \qquad A_\vr \frac{\rmd}{\rmd \ve} \eta^\ve = -\nabla\div(\vr\xi) -\frac{\rmd}{\rmd \ve} A_{\vr^\ve} \eta.
\end{align}
 For $X \in \M$ $\eta$ vanishes, so 
 $ 
 \frac{\rmd}{\rmd \ve} \eta^\ve = -A_{\vr}^{-1}\nabla\div(\vr\xi).
$
 Consequently, we compute at point $X \in \M$
\begin{align}
\hess_{\vr_0, \rm GN} U_X (\xi \circ X, \cdot) &=\frac{\rmd}{\rmd \ve} \grad_{\vr_0, \rm GN} U [X^\ve]  =  \frac{\rmd}{\rmd \ve}(\eta^\ve \circ X^\ve)\\
&=  \frac{\rmd}{\rmd \ve}\eta^\ve \circ X +  \cancel{(\xi \cdot \nabla\eta) \circ X}= -A_{\vr}^{-1}\nabla\div(\vr\xi) \circ X.
\end{align}
Thus in particular it is constant on $\M$ and 
\be 
\hess_{\vr_0, \rm GN} U_X (\xi \circ X, \xi \circ X) = \bigg(-A_\vr^{-1}\nabla\div(\vr\xi) \circ X, \xi \circ X\bigg)_{\vr_0, \rm GN} = \int_{\Omega} (\div(\vr\xi))^2 \rmd x
\ee
so by the standard elliptic theory it is positive definite on normal space. 
\end{proof}  
\appendix 
\section{Hellmann-Feynman type formulas}\label{HFform}
\begin{lemma}\label{HF1}
	Suppose $A^\ve$ is a $C^1$ family of linear operators with a common dense domain on a real Hilbert space $H$. Suppose additionally that $A^\ve$ are self-adjoint with respect to (possibly $\ve$-dependent) inner product $\big(\cdot,\cdot\big)_\ve$ on $H$. Let $\lambda^\ve, \varphi^\ve$ be a $C^1$ family of eigenpairs of $A^\ve$, with $\varphi^\ve$ normalized to have $\|\varphi^\ve\|_\ve = 1$. Then 
	\be 
	\frac{\rmd}{\rmd \ve}\lambda^\ve = \bigg(\varphi^\ve,\frac{\rmd}{\rmd \ve}A^\ve \varphi^\ve\bigg)_\ve.
	\ee   
\end{lemma}
\begin{proof}
	Differentiating eigenproblem with respect to $\ve$ we obtain
	\be 
	\frac{\rmd}{\rmd \ve}A^\ve \varphi^\ve + A^\ve \frac{\rmd}{\rmd \ve}\varphi^\ve = \frac{\rmd}{\rmd \ve}\lambda^\ve \varphi^\ve + \lambda^\ve \frac{\rmd}{\rmd \ve}\varphi^\ve
	\ee 
	Taking inner product with $\varphi^\ve$ and using normalization of $\varphi^\ve$, we obtain
	\be 
	\bigg(\varphi^\ve,\frac{\rmd}{\rmd \ve}A^\ve \varphi^\ve\bigg)_\ve + \bigg(\varphi^\ve, A^\ve \frac{\rmd}{\rmd \ve}\varphi^\ve\bigg)_\ve = \frac{\rmd}{\rmd \ve}\lambda^\ve + \lambda^\ve \bigg(\varphi^\ve,\frac{\rmd}{\rmd \ve}\varphi^\ve\bigg)_\ve
	\ee 
	so that
	\be 
	\frac{\rmd}{\rmd \ve}\lambda^\ve =  \bigg(\varphi^\ve,\frac{\rmd}{\rmd \ve}A^\ve \varphi^\ve\bigg)_\ve + \bigg(\varphi^\ve, A^\ve \frac{\rmd}{\rmd \ve}\varphi^\ve\bigg)_\ve - \lambda^\ve \bigg(\varphi^\ve,\frac{\rmd}{\rmd \ve}\varphi^\ve\bigg)_\ve.
	\ee 
	Since $A^\ve$ are self-adjoint, the last two terms cancel:
	\begin{align} 
		\bigg(\varphi^\ve, A^\ve \frac{\rmd}{\rmd \ve}\varphi^\ve\bigg)_\ve - \lambda^\ve \bigg(\varphi^\ve,\frac{\rmd}{\rmd \ve}\varphi^\ve\bigg)_\ve &= \bigg(A^\ve\varphi^\ve,  \frac{\rmd}{\rmd \ve}\varphi^\ve\bigg)_\ve - \lambda^\ve \bigg(\varphi^\ve,\frac{\rmd}{\rmd \ve}\varphi^\ve\bigg)_\ve\\
		&=\bigg(\lambda^\ve\varphi^\ve,  \frac{\rmd}{\rmd \ve}\varphi^\ve\bigg)_\ve - \lambda^\ve \bigg(\varphi^\ve,\frac{\rmd}{\rmd \ve}\varphi^\ve\bigg)_\ve = 0.
	\end{align}
\end{proof}
\begin{lemma}\label{HF2}
	Suppose $A^\ve$ is a $C^1$ family of self-adjoint operators with a common dense domain and compact resolvents on a real Hilbert space $H$. Let $\lambda_i^\ve, \varphi_i^\ve$ be eigenvalue-eigenvector pairs of $A^\ve$, with $\varphi_i^\ve$ normalized to have $\|\varphi_i^\ve\|_H = 1$ with $\lambda^\ve_i$ simple. Then 
	 \be 
	\frac{\rmd}{\rmd \ve}\varphi_i^\ve = \sum_{j\neq i} \frac{1}{\lambda_i^\ve - \lambda_j^\ve}\bigg(\varphi_j^\ve,\frac{\rmd}{\rmd \ve}A^\ve \varphi_i^\ve\bigg)_H \varphi_j.
	\ee 
\end{lemma}
\begin{proof}
	 Differentiating eigenproblem with respect to $\ve$ we obtain
	 \be 
	 \frac{\rmd}{\rmd \ve}A^\ve \varphi_i^\ve + A^\ve \frac{\rmd}{\rmd \ve}\varphi_i^\ve = \frac{\rmd}{\rmd \ve}\lambda^\ve \varphi_i^\ve + \lambda_i^\ve \frac{\rmd}{\rmd \ve}\varphi_i^\ve.
	 \ee 
	 Taking inner product with any other eigenvector $\varphi_j^\ve$ with eigenvalue $\lambda_j^\ve$ we obtain
	 \be 
	 \bigg(\varphi_j^\ve,\frac{\rmd}{\rmd \ve}A^\ve \varphi_i^\ve\bigg)_H + \bigg(\varphi_j^\ve, A^\ve \frac{\rmd}{\rmd \ve}\varphi_i^\ve\bigg)_H = \frac{\rmd}{\rmd \ve}\lambda_i^\ve \bigg(\varphi_j^\ve, \varphi_i^\ve\bigg)_H + \lambda_i^\ve \bigg(\varphi_j^\ve,\frac{\rmd}{\rmd \ve}\varphi_i^\ve\bigg)_H.
	 \ee
	 Since eigenvectors with different eigenvalues are orthogonal, $(\varphi_j^\ve, \varphi_i^\ve)_H = 0 $. Thus using self-adjointness of $A^\ve$,
	 	\begin{align} 
	 	\bigg(\varphi_j^\ve,\frac{\rmd}{\rmd \ve}A^\ve \varphi_i^\ve\bigg)_H  &=  \lambda_i^\ve \bigg(\varphi_j^\ve,\frac{\rmd}{\rmd \ve}\varphi_i^\ve\bigg)_H - \bigg(\varphi_j^\ve, A^\ve \frac{\rmd}{\rmd \ve}\varphi_i^\ve\bigg)_H\\
	 	&=\lambda_i^\ve \bigg(\varphi_j^\ve,\frac{\rmd}{\rmd \ve}\varphi_i^\ve\bigg)_H - \bigg(A^\ve\varphi_j^\ve,  \frac{\rmd}{\rmd \ve}\varphi_i^\ve\bigg)_H\\
	 	&=\lambda_i^\ve \bigg(\varphi_j^\ve,\frac{\rmd}{\rmd \ve}\varphi_i^\ve\bigg)_H - \lambda_j^\ve\bigg(\varphi_j^\ve,  \frac{\rmd}{\rmd \ve}\varphi_i^\ve\bigg)_H,
	 \end{align}
 so that we obtain:
 \be 
 \bigg(\varphi_j^\ve,  \frac{\rmd}{\rmd \ve}\varphi_i^\ve\bigg)_H = \frac{1}{\lambda_i^\ve - \lambda_j^\ve}\bigg(\varphi_j^\ve,\frac{\rmd}{\rmd \ve}A^\ve \varphi_i^\ve\bigg)_H.
 \ee 
 Additionally, normalization of $\varphi_i^\ve$ implies 
$\bigg(\varphi_i^\ve,  \frac{\rmd}{\rmd \ve}\varphi_i^\ve\bigg)_H = 0.$
 By spectral theorem and the assumption of simplicity of eigenvalues, this implies that 
 \be 
  \frac{\rmd}{\rmd \ve}\varphi_i^\ve = \sum_{j\neq i} \frac{1}{\lambda_i^\ve - \lambda_j^\ve}\bigg(\varphi_j^\ve,\frac{\rmd}{\rmd \ve}A^\ve \varphi_i^\ve\bigg)_H \varphi_j^\ve.
 \ee 
 This completes the proof.
\end{proof}

\section{Proof of Takens-Bornemann Theorem away from resonance crossings}\label{adiabaticappend}
Here present a version of a proof of Takens-Bornemann theorem. Our proof will largely follow \cite{B98}, and is provided mostly for expository reasons, to showcase the main ideas that must go into the rigorous proofs of convergence to the formal Takens limit systems derived in the present paper. For that reason over this presentation we make several simplifying assumptions:
\begin{itemize}
	\item We assume that the enveloping space  is Euclidean. This simplifies formulas notably; in order to obtain the stronger version of Bornemann, one has to redo the whole proof in Riemannian-differential notation, keeping track of derivatives of the metric, curvature terms and other complications.
	\item We assume over the course of the proof that there are no non-singular forcing (so that $W$ in \eqref{strong} is zero). The inclusion of $W$ only introduces the terms that are either strongly convergent or lower order, and does not affect the substance of the proof.
	\item We will also ignore the details regarding transversal resonances, and assume that non-resonance conditions are satisfied pointwise. While a substantial simplification, the details can be carried out as done in \cite{B98}: morally, transversal resonances are invisible under weak limits.
\end{itemize} 
For a more complete treatment, see \cite{B98}. With these simplifications, the statement and proof of Takens-Bornemann theorem \ref{maintheorem} is given below.
\begin{theorem}\label{maintheorem_restricted}
	Consider potential $U$ that is constraining spectrally smooth to  a nondegenerate submanifold $\M \subset \R^d$.
	Consider a sequence $\ve \to 0$, and a mildly ill-prepared sequence of initial data $(X_0^\ve, V_0^\ve)$. Let $(X_0,X_{0N}, V_0) \in \M \times N_{X_0} \M \times T_{X_0}\R^d$ be as in definition
	\ref{initial_data}.
	
	For each $\ve$, let $X^\ve$ be a solution of \eqref{strong} with initial conditions $X_0^\ve, V_0^\ve$. Define constants $c_i$  by 
	\be \label{constantsci}
	c_i = \frac{|\mathbf{P}_i (X_0) V_0|^2 + \lambda_i(X_0)|\mathbf{P}_i (X_0) X_{0N}|^2}{2\sqrt{\lambda_i(X_0)}},
	\ee 
	where $V_0, X_0, X_{0N}$ are from the definition of mildly ill-prepared data \ref{initial_data}.
	
	Let $X$ be the solution of the {Takens limit system} with homogenized potential $V: \M \to \R$ given in \eqref{hompot} constructed from $c_i$ as above.
	If $X$ is pointwise non-resonant up to order 3, then
	\be
	X^\ve \to X \ \text{in} \  C^\alpha([0,T])
	\ee
	for all $\alpha \in(0,1)$ and any $T < \infty$.
\end{theorem}
\begin{proof}[Proof of Theorem \ref{maintheorem_restricted}] 
	
	Let $N$ be the tubular neighborhood of $\M$ on which the nearest point projection $\pi: N \to \M$ is defined. One deduces from the conservation of energy that if $\ve$ is sufficiently small then $X^\ve$ never leaves $N$. Thus the splitting
	\be
	X^\ve = X^\ve_T + \ve X^\ve_N \quad \text{with}\quad X^\ve_T = \pi(X^\ve)
	\ee
	makes sense.
	Over the course of the proof we will abuse notation and write 
	\be 
	f^\ve = f(X_T^\ve),\ f = f(X_T)
	\ee
	for any tensor field $f$ defined on a $\M$.
	We also remark for future use that for $a\in \M$,
	\be
	D\pi(a)= \mathbf{P}_{T_a \M} \qquad \text{ and}\qquad \begin{cases}\label{D^2pi}
		D^2\pi(a):(\xi\otimes\eta) = 0 \text{ for } \xi, \eta \in N_a\M\\
		D^2\pi(a):(\xi\otimes\eta) = \sff(a)(\xi,\eta) \text{ for } \xi, \eta \in T_a\M\\
		D^2\pi(a):(\xi\otimes\eta) \in T_a\M \text{ for } \xi \in N_a\M, \eta \in T_a\M\\
	\end{cases}.
	\ee 
	\begin{lemma}[Uniform energy estimates]\label{estimates}
		We have the following
		\be 
		|X_N^\ve| + |\ve\dot{X}_N^\ve| + 		|\mathbf{P}_{T_{X^\ve_T}} \dot{X}_N^\ve| +  |\dot{X}^\ve_T| + |\ddot{X}^\ve_T| = O(1).
		\ee
	\end{lemma}
	\begin{proof}[Proof of  Lemma \ref{estimates}]
		From the conservation of energy and  non-degeneracy of the minimum of $U$ on $\M$, we see that
		$
		X^\ve_N = O(1).
		$
		Since $D\pi$ is bounded on $N$ and $\dot{X}^\ve$ is bounded from the energy, we also note that 
		\be 
		\dot{X}^\ve_T = D\pi(X^\ve)\dot{X}^\ve = O(1).
		\ee 
		For the normal component we observe that 
		\be 
		\ve \dot{X}^\ve_N = \dot{X}^\ve - \dot{X}^\ve_T = (I - D\pi(X^\ve))\dot{X}^\ve = (\mathbf{P}_{N_{X^\ve_T}} + O(\ve))\dot{X}^\ve,
		\ee
		and thus 
		\be 
		|\ve\dot{X}_N^\ve|, |\mathbf{P}_{T_{X^\ve_T}} \dot{X}_N^\ve| = O(1).
		\ee 
		Finally, by Taylor expansion and the fact that $\hess_{X^\ve_T} U (X^\ve_N, \cdot)$ only has normal component we conclude that 
		\begin{align} 
			\ddot{X}^\ve_T &=  D\pi(X^\ve) \ddot{X}^\ve+ D^2 \pi(X^\ve) : (\dot{X}^\ve \otimes \dot{X}^\ve) \\
			&= -\frac{1}{\ve^2}(\mathbf{P}_{T_{X^\ve_T}} + O(\ve))(\nabla U(X^\ve_T) + \ve \hess_{X^\ve_T} U (X^\ve_N, \cdot) + O(\ve^2)) + O(1)= O(1).
		\end{align}
		This concludes the proof of the Lemma.
	\end{proof}
	With these bounds, we proceed to write down the equations satisfied by $X^\ve_T$ and $X^\ve_N$.
	\begin{lemma}[Equations of tangential and normal motion]\label{lemeqns}
		$X^\ve_T$ and $X^\ve_N$ solve
		\begin{align}
			\label{tangent}\ddot{X}^\ve_T &= -\frac{1}{2}\mathbf{P}_{T_{X^\ve_T}}\nabla^3 U(X^\ve_T):(X^\ve_N\otimes X^\ve_N) \\
			&\qquad\qquad+ \sff(X_{T}^\ve)(\dot{X}^\ve_T,\dot{X}^\ve_T) + 2D^2\pi(X_{T}^\ve):(\dot{X}^\ve_T\otimes\ve\dot{X}^\ve_N) + O(\ve),\\
			\label{normal}\ve\ddot{X}^\ve_N &= -\frac{1}{\ve}\hess_{X^\ve_T} U (X^\ve_N, \cdot) -\frac{1}{2}\mathbf{P}_{N_{X^\ve_T}}\nabla^3 U(X^\ve_T):(X^\ve_N\otimes X^\ve_N)\\
			&\qquad\qquad - \sff(X_{T}^\ve)(\dot{X}^\ve_T,\dot{X}^\ve_T) - 2D^2\pi(X_{T}^\ve):(\dot{X}^\ve_T\otimes\ve\dot{X}^\ve_N) + O(\ve).
		\end{align}
		
	\end{lemma}
	\begin{proof}[Proof of  Lemma \ref{lemeqns}]
		From Taylor expansion we have,
		\begin{align}
			\ddot{X}^\ve_T &=  D\pi(X^\ve) \ddot{X}^\ve+ D^2 \pi(X^\ve) : (\dot{X}^\ve \otimes \dot{X}^\ve) \\
			&= -\frac{1}{\ve^2}\bigg(\mathbf{P}_{T_{X^\ve_T}} + \ve D^2 \pi(X^\ve_T) : X^\ve_N\bigg)\bigg(\nabla U(X^\ve_T) + \ve \hess_{X^\ve_T} U (X^\ve_N, \cdot) + \frac{\ve^2}{2} \nabla^3 U(X^\ve_T):(X^\ve_N\otimes X^\ve_N)\bigg) \\
			&\qquad+D^2 \pi(X^\ve_T) : (\dot{X}^\ve_T + \ve\dot{X}^\ve_N)\otimes(\dot{X}^\ve_T + \ve\dot{X}^\ve_N) + O(\ve)\\
			&= -\frac{1}{2}\mathbf{P}_{T_{X^\ve_T}}\nabla^3 U(X^\ve_T):(X^\ve_N\otimes X^\ve_N) + \sff(X_{T}^\ve)(\dot{X}^\ve_T,\dot{X}^\ve_T) + 2D^2\pi(X_{T}^\ve):(\dot{X}^\ve_T\otimes\ve\dot{X}^\ve_N) + O(\ve).
		\end{align}
		
		For $X^\ve_N$ we have 
		\begin{align}
			\ve\ddot{X}^\ve_N &= \ddot{X}^\ve - \ddot{X}^\ve_T \\
			&= -\frac{1}{\ve^2}\bigg(\nabla U(X^\ve_T) + \ve \hess_{X^\ve_T} U (X^\ve_N, \cdot) + \frac{\ve^2}{2} \nabla^3 U(X^\ve_T):(X^\ve_N\otimes X^\ve_N)\bigg) -\ddot{X}_T^\ve + O(\ve)\\
			&= -\frac{1}{\ve}\hess_{X^\ve_T} U (X^\ve_N, \cdot) -\frac{1}{2}\mathbf{P}_{N_{X^\ve_T}}\nabla^3 U(X^\ve_T):(X^\ve_N\otimes X^\ve_N) \\
			&\qquad\qquad- \sff(X_{T}^\ve)(\dot{X}^\ve_T,\dot{X}^\ve_T) - 2D^2\pi(X_{T}^\ve):(\dot{X}^\ve_T\otimes\ve\dot{X}^\ve_N) + O(\ve).
		\end{align}
		The stated equations  follow.
	\end{proof}
	We would like to take the limits as $\ve \to 0$ in \eqref{tangent} and \eqref{normal}. For that we need the objects provided by the following lemma.
	\begin{lemma}[Limiting objects] There exist $X_T, \sigma_i, \eta_i$ such that up to extraction of subsequences
		\begin{align}
			X^\ve_T &\to X_T \text{ strongly in } C^1([0,T], \M),\\
			\ve^2 |\mathbf{P}_i^\ve\dot{X}^\ve_N|^2 &\to \sigma_i \text{ weakly}^* \text{ in } L^\infty([0,T]),\\
			|\mathbf{P}_i^\ve X^\ve_N|^2 &\to \eta_i \text{ weakly}^* \text{ in } L^\infty([0,T]).
		\end{align}
	\end{lemma}
	\begin{proof}
		The limits follow from using the uniform estimates of lemma \ref{estimates} and applying the theorems of Arzela-Ascoli and Banach-Alaoglu.
	\end{proof}
	Of crucial importance is the following lemma about equipartition of normal energies:
	\begin{lemma}\label{equipartition}
		Quantities $\sigma_i$ and $\eta_i$ are related by
		\be 
		\sigma_i = \lambda_i(X_T)\eta_i.
		\ee 
	\end{lemma} 
	\begin{proof}[Proof of Lemma \ref{equipartition}]
		Multiplying \eqref{normal} by an $O(\ve)$ quantity $\ve \mathbf{P}_i^\ve X_N^\ve$, we obtain 
		\begin{align}
			\ve^2 \langle \ddot{X}^\ve_N, \mathbf{P}_i^\ve X_N^\ve \rangle &= -\hess_{X^\ve_T} U (X^\ve_N, \mathbf{P}_i^\ve X_N^\ve) + O(\ve)\\
			&= -\lambda_i(X_T^\ve) \langle X^\ve_N, \mathbf{P}_i^\ve X_N^\ve\rangle + O(\ve)\\
			&= -\lambda_i(X_T^\ve) |\mathbf{P}_i^\ve X_N^\ve|^2 + O(\ve), 
		\end{align}
		so that
		\begin{align}\label{weak-energy}
			\frac{\rmd}{\rmd t}\langle \ve  \dot{X}_N^\ve, \ve \mathbf{P}_i^\ve X_N^\ve \rangle &= \ve^2 |\mathbf{P}_i^\ve \dot{X}_N^\ve|^2 + \ve^2 \langle \ddot{X}^\ve_N, \mathbf{P}_i^\ve X_N^\ve \rangle
			+ \ve^2 \langle \dot{X}^\ve_N, \dot{\mathbf{P}}_i^\ve X_N^\ve \rangle\\
			&= \ve^2 |\mathbf{P}_i^\ve \dot{X}_N^\ve|^2 - \lambda_i(X_T^\ve) |\mathbf{P}_i^\ve X_N^\ve|^2 + O(\ve) = O(1).
		\end{align}
		Since 
		$ 
		\langle \ve  \dot{X}_N^\ve, \ve \mathbf{P}_i^\ve X_N^\ve \rangle = O(\ve)
		$
		and its derivative is bounded, it follows that this derivative in fact weakly$^*$ converges to zero. Thus 
		taking the weak$^*$ limit in \eqref{weak-energy}, we conclude that 
		\be 
		0 = \sigma_i - \lambda_i(X_T)\eta_i.
		\ee 
		This concludes the proof of the Lemma.
	\end{proof}
	\begin{lemma}[Weak$^*$ convergence to zero of off-diagonal quadratic terms]\label{weakquadratic}
		\begin{align}
			\ve \dot{X}_N^\ve &\to 0 \text{ weakly}^* \text{ in } L^\infty([0,T]), \\
			\label{cross_zero} \mathbf{P}_i^\ve X_N^\ve \otimes \mathbf{P}_j^\ve X_N^\ve  &\to 0 \text{ weakly}^* \text{ in } L^\infty([0,T]) \ \text{for} \ i\neq j
		\end{align}
	\end{lemma}
	\begin{proof}[Proof of Lemma \ref{weakquadratic}]
		The first assertion follows from Lemma \ref{estimates}, since the estimates there imply that 
		$\ve X_N^\ve \to 0$ and $\ve \dot{X}_N^\ve = O(1)$
		and if derivatives of strongly convergent to zero sequence are bounded they converge weakly$^*$ to zero.  
		
		To see quadratic convergence, we mimic the proof of Lemma \ref{equipartition}. Taking tensor product of equation \eqref{normal} with $X^\ve_N$, we get 
		\begin{align}
			\frac{\rmd}{\rmd t} \ve  \dot{X}_N^\ve \otimes \ve X_N^\ve  &= \ve^2  \dot{X}_N^\ve \otimes \dot{X}_N^\ve + \ve^2  \ddot{X}^\ve_N \otimes X_N^\ve \\
			&= \ve^2 \dot{X}_N^\ve \otimes  \dot{X}_N^\ve -\sum_k\lambda_k^\ve \mathbf{P}_k^\ve X_N^\ve \otimes X_N^\ve  + O(\ve).
		\end{align}
		Since $\ve  \dot{X}_N^\ve \otimes \ve X_N^\ve = O(\ve)$
		and its derivative is bounded, in follows that this derivative in fact weakly$^*$ converges to zero. Thus 
		taking the weak$^*$ limit we conclude that 
		\begin{align}
			0 &= \wlim_{\ve \to 0}(\ve \dot{X}_N^\ve \otimes  \ve \dot{X}_N^\ve - \sum_k \lambda_k^\ve \mathbf{P}_k^\ve X_N^\ve \otimes  X_N^\ve).
		\end{align}
		Thus tensor  $\sum_{k} \lambda_k (\wlim_{\ve\to 0} \mathbf{P}_k^\ve X_N^\ve \otimes  X_N^\ve)$ is symmetric, so multiplying the equality 
		\be 
		\sum_{k} \lambda_k (\wlim_{\ve\to 0} \mathbf{P}_k^\ve X_N^\ve \otimes  X_N^\ve) = \sum_{k} \lambda_k (\wlim_{\ve\to 0} X_N^\ve \otimes  \mathbf{P}_k^\ve X_N^\ve)
		\ee 
		by $\mathbf{P}_i^\ve$ on the left and $\mathbf{P}_j^\ve$ on the right we obtain  
		\be 
		\lambda_i (\wlim_{\ve\to 0} \mathbf{P}_i^\ve X_N^\ve \otimes  \mathbf{P}_j^\ve X_N^\ve) = \lambda_j (\wlim_{\ve\to 0} \mathbf{P}_i^\ve X_N^\ve \otimes  \mathbf{P}_j^\ve X_N^\ve).
		\ee 
		From the condition of no resonance, this implies that
		$
		\wlim_{\ve\to 0} \mathbf{P}_i^\ve X_N^\ve \otimes  \mathbf{P}_j^\ve X_N^\ve = 0
		$ as desired.
	\end{proof}
	With that, we proceed by taking the weak$^*$ limit in \eqref{tangent} to obtain
	\be\label{w-limit 1}
	\wlim_{\ve \to 0} \ddot{X}^\ve_T = -\frac{1}{2} \wlim_{\ve \to 0} \mathbf{P}_{T_{X^\ve_T}}\nabla^3 U(X^\ve_T):(X^\ve_N\otimes X^\ve_N)
	+ \sff(X_{T})(\dot{X}_T,\dot{X}_T).
	\ee 
	We manipulate the third order term:
	\begin{align}
		\wlim_{\ve \to 0} \mathbf{P}_{T_{X^\ve_T}}\nabla^3 U(X^\ve_T):(X^\ve_N\otimes X^\ve_N) &=
		\sum_{i} \wlim_{\ve \to 0} \mathbf{P}_{T_{X^\ve_T}}\nabla^3 U(X_T^\ve):( \mathbf{P}_i X^\ve_N\otimes \mathbf{P}_i X^\ve_N) \\
		&\qquad + \sum_{i\neq j} \mathbf{P}_{T_{X_T}}\nabla^3 U(X_T):( \wlim_{\ve \to 0}(\mathbf{P}_i X^\ve_N\otimes \mathbf{P}_j X^\ve_N))\\
		&= \sum_{i} \wlim_{\ve \to 0} \mathbf{P}_{T_{X^\ve_T}}\nabla^3 U(X_T^\ve):( \mathbf{P}_i X^\ve_N\otimes \mathbf{P}_i X^\ve_N)
	\end{align}
	Here cross terms weakly$^*$ converge to zero due to \eqref{cross_zero}. To take the limits of diagonal terms, we note that $\mathbf{P}_i X_N^\ve$ are eigenvectors of $\hess U_{X_T^\ve}$ with eigenvalues $\lambda_i(X_T^\ve)$. Thus from the Hellmann-Feynman formula (see Appendix  \ref{HFform}), it follows that
	\be 
	\nabla \lambda_i(X_T^\ve) = \frac{\nabla \hess U_{X_T^\ve} : (\mathbf{P}_i X_N^\ve \otimes \mathbf{P}_i X_N^\ve)}{|\mathbf{P}_i X_N^\ve|^2}
	= \frac{\nabla^3 U(X):(\mathbf{P}_i X_N^\ve \otimes \mathbf{P}_i X_N^\ve)}{|\mathbf{P}_i X_N^\ve|^2},
	\ee
	and, consequently,
	\begin{align}\label{nablacubed}
		\wlim_{\ve \to 0} \mathbf{P}_{T_{X^\ve_T}}\nabla^3 U(X^\ve_T):(X^\ve_N\otimes X^\ve_N) &= \sum_{i} \wlim_{\ve \to 0} \mathbf{P}_{T_{X^\ve_T}}\nabla^3 U(X_T^\ve):( (\mathbf{P}_i X^\ve_N\otimes \mathbf{P}_i X^\ve_N))\\
		&= \sum_{i} \wlim_{\ve \to 0} |\mathbf{P}_i X_N^\ve|^2 \mathbf{P}_{T_{X^\ve_T}}\nabla \lambda_i(X_T^\ve)\\
		&= \sum_{i} \eta_i \nabla^T \lambda_i(X_T).
	\end{align}
	We plug \eqref{nablacubed} into \eqref{w-limit 1}, and note the resulting equation together with the fact that $X_T$ is $C^1$ implies that, in fact, $X_T$ is $C^2$ and we have
	\be \label{tangenteta}
	\ddot{X}_T = -\frac{1}{2}\sum_i \eta_i \nabla^T \lambda_i(X_T)
	+ \sff(X_{T})(\dot{X}_T,\dot{X}_T).
	\ee 
	
	This is almost the closed equation for $X_T$, except for the appearance of $\eta_i$. To deal with it, we use the fact that the equation of normal motion \eqref{normal} is that of a slowly modulated harmonic oscillator; it suggests looking for the \textit{adiabatic invariants}.
	\begin{lemma}[Adiabatic invariance of action]\label{adiabatic}
		
		It holds that 
		\be
		\eta_i(t) = \frac{c_i}{\sqrt{\lambda_i(X_T(t))}}.
		\ee 
		where the constants $c_i$ are given by \eqref{constantsci}.
	\end{lemma}
	\begin{proof}[Proof of Lemma \ref{adiabatic} for codimension 1 case] We consider the case when the minimum set of the potential  is a hypersurface, so that  $N \M$ is one dimensional.
		We observe from \eqref{tangenteta} that
		\be
		\frac{\rmd}{\rmd t} \frac{1}{2}|\dot{X}_T|^2 = -\frac{1}{2}\eta\frac{\rmd }{\rmd t}\lambda(X_T).
		\ee 
		On the other hand, using that $X_T^\ve \to X_T$ in $C^1$ to take uniform limit in the energy equation
		\begin{align}
			E_0 &= \frac{1}{2}|\dot{X}^\ve|^2 + \frac{1}{\ve^2}U(X^\ve)\\
			&= \frac{1}{2}|\dot{X}_T^\ve|^2 + \langle \dot{X}_T^\ve, \ve \dot{X}_{N}^\ve \rangle + \frac{1}{2}|\ve\dot{X}_{N}^\ve|^2  + \frac{1}{2}\lambda(X_T)|X_N^\ve |^2 + O(\ve)
		\end{align}
		we get 
		\be
		E_0 = \frac{1}{2} |\dot{X}_T|^2 + \frac{1}{2}\sigma + \frac{1}{2}\lambda(X_T)\eta =  \frac{1}{2} |\dot{X}_T|^2 +  \lambda(X_T)\eta.
		\ee 
		Hence, differentiating in time
		\be 
		\frac{\rmd}{\rmd t}\frac{1}{2}|\dot{X}_T|^2 = - \frac{\rmd}{\rmd t} \bigg(\eta\lambda(X_T)\bigg).
		\ee 
		Equating two expressions for $\frac{\rmd}{\rmd t}\frac{1}{2}|\dot{X}_T|^2$, we get an ODE that can be immediately solved:
		\be \label{adiabaticODE}
		\frac{1}{2}\eta\frac{\rmd }{\rmd t}\lambda(X_T) = \frac{\rmd}{\rmd t} \bigg(\eta\lambda(X_T)\bigg) \ \Rightarrow \  \eta = \frac{c}{\sqrt{\lambda(X_T)}}.
		\ee 
		This completes a proof in the case where the minimum set of the potential is a hypersurface.
	\end{proof}
	Argument above only works in the codimension 1 case. For higher codimension, the argument is more involved; one could see it coming from the fact that while the amount of unknown functions $\eta_i$ grows with codimension, the conservation of total energy still provides only one equation. Thus we need to perform a more careful analysis of \eqref{normal}; in particular the non-resonance conditions play a crucial role in that. In fact, in \cite{T80} an example is presented in which the presence of resonance breaks the validity of lemma, and moreover prevents the selection of unique limiting solution. Before we give the argument which is valid for a minimum set for the potential having arbitrary codimension, let us finish the proof assuming Lemma \ref{adiabatic}. 
	
	Inserting the expression for $\eta_i$ obtained in Lemma \ref{adiabatic} into \eqref{tangenteta}, we conclude that the $X_T$ satisfies the constrained equation with additional potential as claimed:
	\begin{align}
		\ddot{X}_T &= -\sum_i \frac{1}{2}\eta_i\nabla^T \lambda_i(X_T) + \sff(X_{T})(\dot{X}_T,\dot{X}_T)\\
		&= -\sum_i \frac{1}{2}\frac{c_i}{\sqrt{\lambda_i(X_T)}}\nabla^T \lambda_i(X_T) + \sff(X_{T})(\dot{X}_T,\dot{X}_T)\\
		&= -\nabla^T \sum_i c_i \sqrt{\lambda_i(X_T)} + \sff(X_{T})(\dot{X}_T,\dot{X}_T).
	\end{align}
	This concludes the proof of convergence to the Takens limit system.
\end{proof}
\begin{proof}[ Proof of Lemma \ref{adiabatic} for arbitrary codimension]
	We require convergence a specific cubic term.
	\begin{lemma} The following cubic interaction weakly vanishes:
		\be\label{weakcubic}
		\ve \dot{X}_N^\ve \otimes  X_N^\ve \otimes X_N^\ve \to 0 \text{ weakly}^* \text{ in } L^\infty([0,T]).
		\ee
	\end{lemma} 
	\begin{proof}
		The proof is similar to that of Lemma \ref{weakquadratic}. We observe using equation \eqref{normal} that
		\begin{align}
			\frac{\rmd}{\rmd t} \ve  &\dot{X}_N^\ve \otimes \ve \dot{X}_N^\ve \otimes \ve X_N^\ve  = \ve  \ddot{X}_N^\ve \otimes \ve \dot{X}_N^\ve \otimes \ve X_N^\ve + \ve  \dot{X}_N^\ve \otimes \ve \ddot{X}_N^\ve \otimes \ve X_N^\ve + \ve  \dot{X}_N^\ve \otimes \ve \dot{X}_N^\ve \otimes \ve \dot{X}_N^\ve\\
			&= -\sum_k\lambda_k^\ve \mathbf{P}_k^\ve X_N^\ve \otimes \ve \dot{X}_N^\ve \otimes X_N^\ve -\sum_k\lambda_k^\ve  \dot{X}_N^\ve \otimes \ve \mathbf{P}_k^\ve X_N^\ve \otimes X_N^\ve + \ve  \dot{X}_N^\ve \otimes \ve \dot{X}_N^\ve \otimes \ve \dot{X}_N^\ve + O(\ve).
		\end{align}
		Since $\ve\dot{X}_N^\ve \otimes \ve \dot{X}_N^\ve \otimes \ve X_N^\ve = O(\ve)$
		and its derivative is bounded, in follows that this derivative in fact weakly$^*$ converges to zero. Using uniform convergence of $X^\ve_T$ and, consequently, $\lambda_k^\ve$ and $\mathbf{P}_k^\ve$, we get
		\be\label{weak1}
		-\sum_k\lambda_k \mathbf{P}_k X_N^\ve \otimes \ve \dot{X}_N^\ve \otimes X_N^\ve -\sum_k\lambda_k  \dot{X}_N^\ve \otimes \ve \mathbf{P}_k X_N^\ve \otimes X_N^\ve + \ve  \dot{X}_N^\ve \otimes \ve \dot{X}_N^\ve \otimes \ve \dot{X}_N^\ve \rightharpoonup 0 
		\ee 
		Also 
		\be\label{weak2}
		\frac{\rmd}{\rmd t} \ve  X_N^\ve \otimes X_N^\ve \otimes X_N^\ve  = \ve\dot{X}_N^\ve \otimes  X_N^\ve \otimes X_N^\ve + X_N^\ve \otimes \ve \dot{X}_N^\ve \otimes X_N^\ve + X_N^\ve \otimes  X_N^\ve \otimes \ve \dot{X}_N^\ve,
		\ee 
		so for the same reason RHS weakly$^*$ goes to zero.
		Equations \eqref{weak2}, \eqref{weak1}, and two more equations obtained similar to \eqref{weak1} by permuting factors in tensor product can be packaged into the matrix form
		\be 
		\begin{pmatrix}
			I & I& I & 0\\
			-I \otimes \sum_k \lambda_k \mathbf{P}_k \otimes I &-\sum_k \lambda_k \mathbf{P}_k \otimes I \otimes I & 0 & I\\
			-I \otimes I \otimes \sum_k \lambda_k \mathbf{P}_k & 0 & -\sum_k \lambda_k \mathbf{P}_k \otimes I \otimes I & I\\
			0 & -I \otimes I \otimes \sum_k \lambda_k \mathbf{P}_k & -I \otimes \sum_k \lambda_k \mathbf{P}_k \otimes I & I
		\end{pmatrix}\begin{pmatrix}
			\ve \dot{X}_N^\ve \otimes  X_N^\ve \otimes X_N^\ve\\
			X_N^\ve \otimes \ve \dot{X}_N^\ve \otimes X_N^\ve\\
			X_N^\ve \otimes X_N^\ve \otimes \ve \dot{X}_N^\ve\\
			\ve  \dot{X}_N^\ve \otimes \ve \dot{X}_N^\ve \otimes \ve \dot{X}_N^\ve
		\end{pmatrix} \rightharpoonup 0.
		\ee 
		A computation shows that determinant of matrix on the left is given by
		\be 
		-\sum_{i,j,k} (\sqrt{\lambda_i} + \sqrt{\lambda_j} + \sqrt{\lambda_k})
		(-\sqrt{\lambda_i} + \sqrt{\lambda_j} + \sqrt{\lambda_k})
		(\sqrt{\lambda_i} - \sqrt{\lambda_j} + \sqrt{\lambda_k})
		(\sqrt{\lambda_i} + \sqrt{\lambda_j} - \sqrt{\lambda_k})\mathbf{P}_i \otimes \mathbf{P}_j\otimes \mathbf{P}_k,
		\ee 
		so under pointwise non-resonance conditions it is invertible. This shows that 
		\be 
		\begin{pmatrix}
			\ve \dot{X}_N^\ve \otimes  X_N^\ve \otimes X_N^\ve\\
			X_N^\ve \otimes \ve \dot{X}_N^\ve \otimes X_N^\ve\\
			X_N^\ve \otimes X_N^\ve \otimes \ve \dot{X}_N^\ve\\
			\ve  \dot{X}_N^\ve \otimes \ve \dot{X}_N^\ve \otimes \ve \dot{X}_N^\ve
		\end{pmatrix} \rightharpoonup 0,
		\ee 
		which in particular implies the result of the lemma.
	\end{proof}
	
	We define normal energies by 
	\be 
	K_i^\ve = \frac{1}{2} \bigg|\ve \mathbf{P}_i^\ve \dot{X}_N^\ve\bigg|^2,  \quad U_i^\ve = \frac{1}{2}\lambda_i^\ve |\mathbf{P}_i^\ve X_N^\ve|^2, \quad E_i^\ve = K_i^\ve + U_i^\ve.
	\ee 
	
	We will take the weak$^*$ limit of 
	\begin{align}
		\frac{\rmd}{\rmd t}E_i^\ve &= \langle\ve \mathbf{P}_i^\ve \dot{X}_N^\ve, \ve \dot{\mathbf{P}}_i^\ve \dot{X}_N^\ve\rangle + \langle\ve \mathbf{P}_i^\ve \dot{X}_N^\ve,\ve \mathbf{P}_i^\ve \ddot{X}_N^\ve\rangle +\frac{1}{2} \dot{\lambda}_i^\ve |\mathbf{P}_i^\ve X_N^\ve|^2 \\
		&\qquad + \lambda_i^\ve \langle \mathbf{P}_i^\ve X_N^\ve, \dot{\mathbf{P}}_i^\ve X_N^\ve\rangle + \lambda_i^\ve \langle \mathbf{P}_i^\ve X_N^\ve,  \mathbf{P}_i^\ve \dot{X}_N^\ve\rangle\\
		&= \sum_j \langle\ve \mathbf{P}_i^\ve \dot{X}_N^\ve, \ve \dot{\mathbf{P}}_i^\ve \mathbf{P}_j^\ve\dot{X}_N^\ve\rangle +\langle\ve \mathbf{P}_i^\ve \dot{X}_N^\ve,\ve \mathbf{P}_i^\ve \ddot{X}_N^\ve\rangle +\frac{1}{2} \dot{\lambda}_i^\ve |\mathbf{P}_i^\ve X_N^\ve|^2 \\
		&\qquad + \sum_j\lambda_i^\ve \langle \mathbf{P}_i^\ve X_N^\ve, \dot{\mathbf{P}}_i^\ve \mathbf{P}_j^\ve X_N^\ve\rangle + \lambda_i^\ve \langle \mathbf{P}_i^\ve X_N^\ve,  \mathbf{P}_i^\ve \dot{X}_N^\ve\rangle\\
		&= \sum_{j\neq i} \langle\ve \mathbf{P}_i^\ve \dot{X}_N^\ve, \ve \dot{\mathbf{P}}_i^\ve \mathbf{P}_j^\ve\dot{X}_N^\ve\rangle +\frac{1}{2} \dot{\lambda}_i^\ve |\mathbf{P}_i^\ve X_N^\ve|^2 \\
		&\qquad + \sum_{j\neq i}\lambda_i^\ve \langle \mathbf{P}_i^\ve X_N^\ve, \dot{\mathbf{P}}_i^\ve \mathbf{P}_j^\ve X_N^\ve\rangle +\langle\ve \mathbf{P}_i^\ve \dot{X}_N^\ve,\ve \mathbf{P}_i^\ve \ddot{X}_N^\ve + \frac{\lambda_i^\ve}{\ve} \mathbf{P}_i^\ve X_N^\ve\rangle 
	\end{align}
	
	Observe that the diagonal terms cancel from 
	\be 
	\mathbf{P}_i^\ve \dot{\mathbf{P}}_i^\ve \mathbf{P}_i^\ve = 0,
	\ee 
	which follows from differentiating $(\mathbf{P}_i^\ve)^2 = \mathbf{P}_i^\ve$. We now take the weak$^*$ limit. First, note that since $X_T^\ve \to X_T$ in $C^1$, $\mathbf{P}_i^\ve \to \mathbf{P}_i$ and $\lambda_i^\ve \to \lambda_i$ in $C^1$ as well. Thus first and third terms weakly$^*$ go to zero from off-diagonal quadratic convergence \eqref{cross_zero}. 
	\be
	\sum_{j\neq i} \langle\ve \mathbf{P}_i^\ve \dot{X}_N^\ve, \ve \dot{\mathbf{P}}_i^\ve \mathbf{P}_j^\ve\dot{X}_N^\ve\rangle\rightharpoonup 0, \quad \sum_{j\neq i}\lambda_i^\ve \langle \mathbf{P}_i^\ve X_N^\ve, \dot{\mathbf{P}}_i^\ve \mathbf{P}_j^\ve X_N^\ve\rangle \rightharpoonup 0.
	\ee 
	Limit of the second term is also immediate:
	\be 
	\frac{1}{2} \dot{\lambda}_i^\ve |\mathbf{P}_i^\ve X_N^\ve|^2 \rightharpoonup \frac{1}{2}\dot{\lambda}_i \eta_i.
	\ee 
	We proceed to show that the last term weakly$^*$ converges to zero. Using equation \eqref{normal}, we compute 
	\begin{align}
		\langle\ve \mathbf{P}_i^\ve \dot{X}_N^\ve,\ve \mathbf{P}_i^\ve \ddot{X}_N^\ve + \frac{\lambda_i^\ve}{\ve}  \mathbf{P}_i^\ve X_N^\ve\rangle &= \bigg\langle\ve \mathbf{P}_i^\ve \dot{X}_N^\ve, -\frac{1}{\ve}\hess_{X^\ve_T} U (X^\ve_N, \cdot) -\frac{1}{2}\nabla^3 U(X^\ve_T):(X^\ve_N\otimes X^\ve_N)\\
		&\qquad- \sff(X_{T}^\ve)(\dot{X}^\ve_T,\dot{X}^\ve_T) - 2 D^2\pi(X_{T}^\ve):(\dot{X}^\ve_T\otimes\ve\dot{X}^\ve_N) + \frac{\lambda_i^\ve}{\ve}  X_N^\ve\bigg\rangle +O(\ve)\\
		&= -\frac{1}{\ve}\hess_{X^\ve_T} U (X^\ve_N, \ve \mathbf{P}_i^\ve \dot{X}_N^\ve) -\frac{1}{2}\nabla^3 U(X^\ve_T):(X^\ve_N\otimes X^\ve_N \otimes \ve \mathbf{P}_i^\ve \dot{X}_N^\ve)\\
		&\qquad- \langle\sff(X_{T}^\ve)(\dot{X}^\ve_T,\dot{X}^\ve_T), \ve \mathbf{P}_i^\ve \dot{X}_N^\ve\rangle - 2 \langle D^2\pi(X_{T}^\ve):(\dot{X}^\ve_T\otimes\ve\dot{X}^\ve_N),\ve \mathbf{P}_i^\ve \dot{X}_N^\ve\rangle \\
		&\qquad\qquad+ \frac{\lambda_i^\ve}{\ve}  \langle \mathbf{P}_i X_N^\ve,\ve \mathbf{P}_i^\ve \dot{X}_N^\ve\rangle +O(\ve)\\
		&= s_1 + s_2 + s_3 + s_4 + s_5 + O(\ve).
	\end{align}
	We have $s_1 + s_5 = 0$ identically. $s_2 \rightharpoonup 0 $ from lemma \ref{weakcubic}, while $s_3 \rightharpoonup 0$ from lemma \ref{weakquadratic}. To show convergence to zero of the last term $s_4$, we rewrite it as
	\begin{align}
		s_4 &= - 2 \langle D^2\pi(X_{T}^\ve):(\dot{X}^\ve_T\otimes\ve\dot{X}^\ve_N),\ve \mathbf{P}_i^\ve \dot{X}_N^\ve\rangle\\
		&= - 2 \langle D^2\pi(X_{T}^\ve):(\dot{X}^\ve_T\otimes\ve \mathbf{P}_{N_{X^\ve_T}\M}\dot{X}^\ve_N),\ve \mathbf{P}_i^\ve \dot{X}_N^\ve\rangle - 2 \langle D^2\pi(X_{T}^\ve):(\dot{X}^\ve_T\otimes \ve \mathbf{P}_{T_{X^\ve_T}\M} \dot{X}^\ve_N),\ve \mathbf{P}_i^\ve \dot{X}_N^\ve\rangle.
	\end{align} 
	The first term is zero by \eqref{D^2pi}, while the second term is $O(\ve)$ by estimates of lemma \ref{estimates}. Thus $s_4 \rightharpoonup 0$ as well. To summarize, we have 
	\be 
	\frac{\rmd}{\rmd t} E_i^\ve \rightharpoonup \frac{1}{2}\dot{\lambda}_i \eta_i.
	\ee 
	On the other hand, from the definition of $E_i^\ve$ we have 
	\be 
	E_i^\ve = \frac{1}{2} \bigg|\ve \mathbf{P}_i^\ve \dot{X}_N^\ve\bigg|^2 + \frac{1}{2}\lambda_i^\ve |\mathbf{P}_i^\ve X_N^\ve|^2 \rightharpoonup \frac{1}{2}\sigma_i + \frac{1}{2}\lambda_i\eta_i = \lambda_i\eta_i
	\ee 
	Since derivative of the weak$^*$ limit is the weak$^*$ limit of derivatives provided they are bounded (and they are, given the formula above), we conclude that
	\be 
	\frac{1}{2}\dot{\lambda}_i \eta_i = \dot{\lambda}_i \eta_i + \lambda_i \dot{\eta}_i \quad \Rightarrow\quad \frac{\rmd }{\rmd t}\bigg(\sqrt{\lambda_i} \eta_i\bigg) = 0.
	\ee 
	From the computation of $\dot{E_i^\ve}$, we in fact see that $E_i^\ve$ converges to $\lambda_i\eta_i$ uniformly. This allows us to deduce that the constants of integration are given by $E_i(0)$, which concludes the proof of Lemma \ref{adiabatic}.
\end{proof}

\section{Second variation of homogenized potential}\label{hessian-appendix}
In this appendix we compute the second variation of the homogenized potential appearing in Takens limit system
$ 
V(X) = \sum_i c_i \sqrt{\lambda_i(X)}.
$
Such a computation can be used to assess whether a state of the Takens system may inherit some stability, as was the case for the double spring pendulum and for the circular loop.

By linearity, it is enough to compute second variation of $\sqrt{\lambda_i}$. First recall the first variation:
\begin{lemma} The tangential gradient of a simple eigenvalue $\lambda_i$ is
	\be 
	\grad^{T} \sqrt{\lambda_i(X)}  = \frac{1}{2\sqrt{\lambda_i(X)}}\mathbf{P}_T^X \nabla ^3 U(X) : (v_i(X) \otimes v_i(X)).
	\ee
\end{lemma}
\begin{proof}
	For a variation of $X$ with $\delta X = \xi \in T_X \M$ we have by Hellmann-Feynman formula \eqref{HF1} and chain rule
	\begin{align}
	\delta \sqrt{\lambda_i (X)} &= \tfrac{1}{2\sqrt{\lambda_i(X)} }\langle v_i(X), \delta \nabla^2 U(X) v_i(X)\rangle =  \tfrac{1}{2\sqrt{\lambda_i(X)} }\langle v_i(X),  (\nabla^3 U(X):\xi ) v_i(X)\rangle\\ &=  \tfrac{1}{2\sqrt{\lambda_i(X)} }\langle \xi,  \nabla^3 U(X):(v_i(X) \otimes v_i(X))\rangle= \bigg\langle \xi,   \tfrac{1}{2\sqrt{\lambda_i(X)} }\nabla^3 U(X):(v_i(X) \otimes v_i(X))\bigg\rangle.
	\end{align}
	This completes the derivation.
\end{proof}
We proceed to taking the second variation.
\begin{theorem}
	The Hessian of a simple eigenvalue $\sqrt{\lambda_i}$ at point $X$ is given by
	\begin{align}
 	\hess &\sqrt{\lambda_i(X)}(\xi, \xi) = -\frac{1}{4\lambda_i(X)^{3/2}} [\nabla^3 U (X) : (v_i(X) \otimes v_i(X) \otimes \xi)]^2 \\
 	&\quad-\frac{1}{2\sqrt{\lambda_i(X)}} \sum_j \frac{1}{\lambda_j(X)} \nabla^3 U(X) :(v_j(X) \otimes \xi \otimes \xi) \nabla ^3 U(X) : (v_i(X) \otimes v_i(X) \otimes v_j(X))\\
 	&\quad\quad+\frac{1}{2\sqrt{\lambda_i(X)}} \nabla^4 U(X):(v_i(X) \otimes v_i(X) \otimes \xi \otimes \xi)\\
 	&\quad\quad\quad+\frac{1}{\sqrt{\lambda_i(X)}}\sum_{j \neq i}\frac{1}{\lambda_i(X) -\lambda_j(X)} [\nabla^3 U(X) : (v_i(X) \otimes v_j(X) \otimes \xi)]^2 \\
 	&\quad\quad\quad\quad + \frac{1}{\lambda_i(X)^{3/2}} |\mathbf{P}_T^X\nabla^3 U(X) : (v_i(X) \otimes \xi)|^2.
 	 \end{align}
\end{theorem}
\begin{proof}
Consider again a variation of $X$ with $\delta X = \xi \in T_X \M$. By definition, the tangential Hessian is given by 
\begin{align}
	\hess \sqrt{\lambda_i(X)} (\xi, \xi) &= \langle \delta  \grad^{T}\sqrt{\lambda_i} (X), \xi \rangle\\
	&= \bigg\langle  \delta \bigg(\frac{1}{2\sqrt{\lambda_i(X)}}\mathbf{P}_T^X \nabla ^3 U(X) : (v_i(X) \otimes v_i(X))\bigg), \xi \bigg\rangle\\
	&= \bigg(\delta \frac{1}{2\sqrt{\lambda_i(X)}}\bigg)\langle  \mathbf{P}_T^X \nabla ^3 U(X) : (v_i(X) \otimes v_i(X)), \xi \rangle\\
	&\quad + \frac{1}{2\sqrt{\lambda_i(X)}}\langle (\delta \mathbf{P}_T^X ) \nabla ^3 U(X) : (v_i(X) \otimes v_i(X)), \xi \rangle\\
	&\quad\quad + \frac{1}{2\sqrt{\lambda_i(X)}}\langle   \mathbf{P}_T^X  (\delta \nabla ^3 U(X)) : (v_i(X) \otimes v_i(X)), \xi \rangle\\
	&\quad\quad\quad + \frac{1}{\sqrt{\lambda_i(X)}}\langle   \mathbf{P}_T^X \nabla ^3 U(X) : (\delta v_i(X) \otimes v_i(X)), \xi \rangle\\
	&= s_1 + s_2 + s_3 + s_4.
\end{align}
We proceed computing term by term. We immediately have
\begin{align}\label{s1}
s_1 &= -\tfrac{\delta \sqrt{\lambda_i(X)}}{2\lambda_i(X)}  \langle  \mathbf{P}_T^X \nabla ^3 U(X) : (v_i(X) \otimes v_i(X)), \xi \rangle= -\tfrac{1}{4\lambda_i(X)^{3/2}} [\nabla^3 U (X) : (v_i(X) \otimes v_i(X) \otimes \xi)]^2 \\
s_3 &= \tfrac{1}{2\sqrt{\lambda_i(X)}} \nabla^4 U(X):(v_i(X) \otimes v_i(X) \otimes \xi \otimes \xi). \label{s3}
\end{align}
To compute $s_4$ we first derive an expression for $\delta v_i(X)$.
\begin{lemma}\label{delta_eigenvector} The variation of the eigenfunction is given by
\begin{align}
\delta v_i(X) &= \sum_{j \neq i}\frac{1}{\lambda_i(X) - \lambda_j(X)}\nabla^3 U(X): (v_i(X) \otimes v_j(X) \otimes \xi) v_j(X) \\
&\qquad + \frac{1}{\lambda_i(X)}\mathbf{P}_T^X (\nabla^3 U(X) :(v_i(X) \otimes \xi))\\
&=: \mathbf{P}_N^X \delta v_i(X) + \mathbf{P}_T^X \delta v_i(X).
\end{align}
\end{lemma}
\begin{proof}
	We have by definition 
	\begin{align} 
	\nabla^2 U(X) v_i(X) &= \lambda_i (X) v_i(X),\\
	\label{delta_eigenvector_eq}
	\delta \nabla^2 U(X) v_i(X) + \nabla^2 U(X) \delta v_i(X) &= \delta \lambda_i (X) v_i(X) + \lambda_i (X) \delta v_i(X).
	\end{align}
	Taking inner product with $v_j(X)$ for $j \neq i$,
	\begin{align} 
	\langle v_j(X), \delta \nabla^2 U(X) v_i(X) \rangle + \langle v_j(X),\nabla^2 U(X) \delta v_i(X)\rangle &= \delta \lambda_i (X) \langle v_j(X),v_i(X)\rangle+ \lambda_i (X) \langle v_j(X), \delta v_i(X)\rangle\\
	\nabla^3 U(X): (v_i(X) \otimes v_j(X) \otimes \xi) + &\lambda_j(X) \langle v_j(X),\delta v_i(X)\rangle = \lambda_i(X) \langle v_j(X), \delta v_i(X)\rangle\\
	\langle v_j(X), \delta v_i(X)\rangle = \tfrac{1}{\lambda_i(X) - \lambda_j(X)}&\nabla^3 U(X): (v_i(X) \otimes v_j(X) \otimes \xi)
	\end{align}
Additionally,
$
\langle v_i(X), \delta v_i (X)\rangle = \frac{1}{2}\delta |v_i(X)|^2 = 0. 
$
Since $v_j(X)$ form an orthonormal basis for $N_X \M$, we conclude that
\be 
\mathbf{P}_N^X \delta v_i(X) = \sum_{j \neq i}\tfrac{1}{\lambda_i(X) - \lambda_j(X)}\nabla^3 U(X): (v_i(X) \otimes v_j(X) \otimes \xi) v_j(X).
\ee  
To compute tangential component of $\delta v_i(X)$ we project \eqref{delta_eigenvector_eq} onto tangent space at $X$ to obtain 
\be 
\mathbf{P}_T^X (\nabla^3 U(X) :(v_i(X) \otimes \xi)) = \lambda_i(X) \mathbf{P}_T^X \delta v_i(X)
\ee 
so that
\be 
\mathbf{P}_T^X \delta v_i(X) = \frac{1}{\lambda_i(X)}\mathbf{P}_T^X (\nabla^3 U(X) :(v_i(X) \otimes \xi))
\ee 
where $\lambda_i(X) \neq 0$ since $U$ is assumed to be constraining to $\M$.
\end{proof}
Continuing on with the computation of $s_4$,
\begin{align}\label{s4}
	s_4 &=  \tfrac{1}{\sqrt{\lambda_i(X)}}\langle   \nabla ^3 U(X) : (\delta v_i(X) \otimes v_i(X)), \xi \rangle\\
	&= \tfrac{1}{\sqrt{\lambda_i(X)}}\langle   \nabla ^3 U(X) : \bigg(\sum_{j \neq i}\tfrac{1}{\lambda_i(X) - \lambda_j(X)}\nabla^3 U(X): (v_i(X) \otimes v_j(X) \otimes \xi) v_j(X) \otimes v_i(X)\bigg), \xi \rangle\\
	&\quad + \tfrac{1}{\sqrt{\lambda_i(X)}}\langle   \nabla ^3 U(X) : \bigg(\tfrac{1}{\lambda_i(X)}\mathbf{P}_T^X (\nabla^3 U(X) :(v_i(X) \otimes \xi)) \otimes v_i(X)\bigg), \xi \rangle\\
	&= \tfrac{1}{\sqrt{\lambda_i(X)}}\sum_{j \neq i}\tfrac{1}{\lambda_i(X) -\lambda_j(X)} [\nabla^3 U(X) : (v_i(X) \otimes v_j(X) \otimes \xi)]^2 \\
	&\quad + \tfrac{1}{\lambda_i(X)^{3/2}} |\mathbf{P}_T^X\nabla^3 U(X) : (v_i(X) \otimes \xi)|^2.
\end{align}
The last term to compute is $s_2$, for which we need the variation of the tangent projector. To simplify the computation, we first note that since 
\begin{align}
s_2 &=  \frac{1}{2\sqrt{\lambda_i(X)}}\langle (\delta \mathbf{P}_T^X ) \nabla ^3 U(X) : (v_i(X) \otimes v_i(X)), \xi \rangle\\
&= \frac{1}{2\sqrt{\lambda_i(X)}}\langle \mathbf{P}_T^X(\delta \mathbf{P}_T^X ) \nabla ^3 U(X) : (v_i(X) \otimes v_i(X)), \xi \rangle,
\end{align}
only the composition $\mathbf{P}_T^X(\delta \mathbf{P}_T^X )$ is necessary for our purpose. Is it computed in the following lemma:
\begin{lemma} The tangential variation of the projector is
	\be 
	\mathbf{P}_T^X(\delta \mathbf{P}_T^X)= -\sum_j \frac{1}{\lambda_j(X)}\mathbf{P}_T^X (\nabla^3 U(X) :(v_j(X) \otimes \xi)) \otimes v_j(X).
	\ee 
\end{lemma}
\begin{proof}
	We note that since $v_j(X)$ form an orthonormal basis for $N_X \M$,
	\be 
	\mathbf{P}_N^X = \sum_{j} v_j(X) \otimes v_j(X), \quad \text{and} \quad \mathbf{P}_T^X = I- \mathbf{P}_N^X = I - \sum_{j} v_j(X) \otimes v_j(X).
	\ee 
	Thus the variation of $\mathbf{P}_T^X$ is given by
	\begin{align}
		\delta \mathbf{P}_T^X &= - \sum_{j} \delta v_j(X) \otimes v_j(X) + v_j(X) \otimes \delta v_j(X)\\
		&= - \sum_{j} \bigg(\mathbf{P}_N^X \delta v_j(X) \otimes v_j(X) + v_j(X) \otimes \mathbf{P}_N^X\delta v_j(X)  \\
		&\qquad + \mathbf{P}_T^X \delta v_j(X) \otimes v_j(X) +  v_j(X) \otimes \mathbf{P}_T^X\delta v_j(X)\bigg).
	\end{align}
Contracting with $\mathbf{P}_T^X$ from the left, only one term survives:
\be
\mathbf{P}_T^X \delta \mathbf{P}_T^X = -\sum_j \mathbf{P}_T^X \delta v_j(X) \otimes v_j(X).
\ee 
Hence plugging in the result of lemma \ref{delta_eigenvector} the result follows.
\end{proof}
With that, $s_2$ is given by
\begin{align}\label{s2}
	s_2 &= \frac{1}{2\sqrt{\lambda_i(X)}}\langle \mathbf{P}_T^X(\delta \mathbf{P}_T^X ) \nabla ^3 U(X) : (v_i(X) \otimes v_i(X)), \xi \rangle\\
	&= -\frac{1}{2\sqrt{\lambda_i(X)}} \sum_j \frac{1}{\lambda_j(X)} \langle \mathbf{P}_T^X (\nabla^3 U(X) :(v_j(X) \otimes \xi)) \otimes v_j(X) : \nabla ^3 U(X) : (v_i(X) \otimes v_i(X)), \xi \rangle\\
	&= -\frac{1}{2\sqrt{\lambda_i(X)}} \sum_j \frac{1}{\lambda_j(X)} (\nabla^3 U(X) :(v_j(X) \otimes \xi \otimes \xi)) (\nabla ^3 U(X) : (v_i(X) \otimes v_i(X) \otimes v_j(X))).
\end{align}
This completes the derivation.
\end{proof}

 \subsection*{Acknowledgments}  
 We would like to thank Paolo Celli, Boris Khesin and Erwin Luesink for useful discussions.
  TDD and DG were partially supported by the NSF DMS-2106233 grant and NSF CAREER award \#2235395. 
  TDD  is also grateful  for support from  a Stony Brook University
Trustee’s award,  an Alfred P. Sloan Fellowship as well as the support of the Institut Henri Poincaré (UAR 839 CNRS-Sorbonne Université), and LabEx CARMIN (ANR-10-LABX-59-01). DG is also grateful for the hospitality of University of Amsterdam where part of this work was completed.


\begin{thebibliography}{99}
\bibitem{ACG}
A. Abanov, T. Can, and S. Ganeshan,
Odd surface waves in two-dimensional incompressible fluids,
\textit{SciPost Physics} \textbf{5} (2018), no.~1, 010.

\bibitem{Al95}
S. S. Antman,
\textit{Nonlinear Problems of Elasticity},
Applied Mathematical Sciences, vol.~107,
Springer-Verlag, New York, 1995.

\bibitem{A12}
V. I. Arnol'd,
\textit{Geometrical Methods in the Theory of Ordinary Differential Equations},
vol.~250,
Springer Science \& Business Media, 2012.

\bibitem{A13}
V. I. Arnol'd,
\textit{Mathematical Methods of Classical Mechanics},
vol.~60,
Springer Science \& Business Media, 2013.

\bibitem{A14}
V. I. Arnol'd,
\textit{Mathematical Understanding of Nature: Essays on Amazing Physical
	Phenomena and Their Understanding by Mathematicians},
American Mathematical Society, 2014.

\bibitem{A66}
V. I. Arnol'd,
Sur la géométrie différentielle des groupes de Lie de dimension infinie
et ses applications à l'hydrodynamique des fluides parfaits,
\textit{Annales de l'Institut Fourier} \textbf{16} (1966), 316--361.

\bibitem{AK}
V. I. Arnol'd and B. A. Khesin,
\textit{Topological Methods in Hydrodynamics},
Springer, New York, 1998.

\bibitem{Abaro}
K. Asano,
On the incompressible limit of the compressible Euler equation,
\textit{Japan Journal of Applied Mathematics} \textbf{4} (1987),
no.~3, 455--488.

\bibitem{BLR}
D. Bao, J. Lafontaine and T. Ratiu,
On a nonlinear equation related to the geometry of the diffeomorphism Group,
\textit{Pacific Journal of Mathematics} \textbf{158} (1993), no. 2, 223--242.

\bibitem{B98}
F. Bornemann,
\textit{Homogenization in Time of Singularly Perturbed Mechanical Systems},
Springer, Berlin--Heidelberg, 1998.

\bibitem{BS97}
F. Bornemann and C. Schütte,
Homogenization of Hamiltonian systems with a strong constraining potential,
\textit{Physica D: Nonlinear Phenomena} \textbf{102} (1997),
nos.~1--2, 57--77.

\bibitem{BDG}
D. Bresch, B. Desjardins, and E. Grenier,
Oscillatory limits with changing eigenvalues: A formal study,
in \textit{New Directions in Mathematical Fluid Mechanics},
A. V. Fursikov, G. P. Galdi, and V. V. Pukhnachev, eds.,
Advances in Mathematical Fluid Mechanics,
Birkhäuser, Basel, 2009.

\bibitem{BDGL}
D. Bresch, B. Desjardins, E. Grenier, and C.-K. Lin,
Low Mach number limit of viscous polytropic flows:
Formal asymptotics in the periodic case,
\textit{Studies in Applied Mathematics} \textbf{109} (2002), 125--149.
doi:10.1111/1467-9590.01440.

\bibitem{B18}
M. Bruveris,
The L2 metric on $C^\infty(M,N)$” (2018), arXiv:1804.00577.

\bibitem{CHL}
R. Camassa, D. D. Holm, and C. Levermore,
Long-time shallow-water equations with a varying bottom,
\textit{Journal of Fluid Mechanics} \textbf{349} (1997), 173--189.

\bibitem{E77}
D. G. Ebin,
The motion of slightly compressible fluids viewed as a motion with strong
constraining force,
\textit{Annals of Mathematics} \textbf{105} (1977),
no.~1, 141--200.

\bibitem{EM}
D. G. Ebin and J. E. Marsden,
Groups of diffeomorphisms and the motion of an incompressible fluid,
\textit{Annals of Mathematics} \textbf{92} (1970), 102--163

\bibitem{GA}
S. Ganeshan and A. G. Abanov,
Odd viscosity in two-dimensional incompressible fluids,
\textit{Physical Review Fluids} \textbf{2} (2017),
no.~9, 094101.

\bibitem{HL}
D. D. Holm and E. Luesink,
Stochastic wave--current interaction in thermal shallow-water dynamics,
\textit{Journal of Nonlinear Science} \textbf{31} (2021),
no.~2, article 29.


\bibitem{KMM21}
B. Khesin, G. Misiołek, and K. Modin,
Geometric hydrodynamics and infinite-dimensional Newton's equations,
\textit{Bulletin of the American Mathematical Society}
\textbf{58} (2021), 377--442.

\bibitem{G83}
G. Gallavotti,
\textit{The Elements of Mechanics},
Texts and Monographs in Physics,
Springer-Verlag, New York, 1983.



\bibitem{M07}
N. Masmoudi,
Rigorous derivation of the anelastic approximation,
\textit{Journal de Mathématiques Pures et Appliquées}
\textbf{88} (2007), no.~3, 230--240.

\bibitem{MS03}
G. Métivier and S. Schochet,
Averaging theorems for conservative systems and the weakly compressible
Euler equations,
\textit{Journal of Differential Equations} \textbf{187} (2003),
no.~1, 106--183.


\bibitem{P11}
S. Preston,
The motion of whips and chains,
\textit{Journal of Differential Equations} \textbf{251} (2011),
no.~3, 504--550.

\bibitem{P12}
S. Preston,
The geometry of whips,
\textit{Annals of Global Analysis and Geometry}
\textbf{41} (2012), no.~3, 281--305.

\bibitem{RU57}
H. Rubin and P. Ungar,
Motion under a strong constraining force,
\textit{Communications on Pure and Applied Mathematics}
\textbf{10} (1957), 65--87.


\bibitem{S}
A. Shnirelman,
The motion of an inextensible thread and incompressible fluid
and foundations of mechanics,
in \textit{Abstracts of the International Conference
	``Mathematical Hydrodynamics''},
Moscow, Russia, June 12--17, 2006,
Steklov Mathematical Institute and Lomonosov Moscow State University,
Moscow, 2006, 72--74.

\bibitem{Sm79}
N. K. Smolentsev,
Principle of Maupertuis,
\textit{Siberian Mathematical Journal} \textbf{20} (1979), 772--776.

\bibitem{T80}
F. Takens,
Motion under the influence of a strong constraining force,
in \textit{Global Theory of Dynamical Systems},
Z. Nitecki and C. Robinson, eds.,
Lecture Notes in Mathematics, vol.~819,
Springer, Berlin--Heidelberg, 1980.




\end{thebibliography}
\end{document}